\documentclass[a4paper,11pt]{article}

\usepackage[utf8]{inputenc}
\usepackage[T1]{fontenc}
\usepackage{lmodern}
\usepackage{amsmath,amssymb,amsthm}
\usepackage[margin=2cm]{geometry}
\usepackage{color}
\usepackage{mathtools}
\usepackage{appendix}
\usepackage[french, english]{babel}
\usepackage{dsfont}
\usepackage{mathrsfs}
\usepackage{stmaryrd}
\usepackage{hyperref}
\usepackage{graphicx}
\usepackage{comment}
\usepackage{soul}

\usepackage{multicol}

\usepackage{enumitem}

\newcommand\MTkillspecial[1]{
  \bgroup
  \catcode`\&=9
  \let\\\relax%
  \scantokens{#1}%
  \egroup
}
\newcommand\DeclarePairedDelimiterMultiline[3]{
  \DeclarePairedDelimiter{#1}{#2}{#3}
  \reDeclarePairedDelimiterInnerWrapper{#1}{star}{
    \mathopen{##1\vphantom{\MTkillspecial{##2}}\kern-\nulldelimiterspace\right.}
    ##2
    \mathclose{\left.\kern-\nulldelimiterspace\vphantom{\MTkillspecial{##2}}##3}}
}

\DeclarePairedDelimiterMultiline{\abs}{\lvert}{\rvert}
\DeclarePairedDelimiterMultiline{\norm}{\lVert}{\rVert}
\DeclarePairedDelimiterMultiline{\pare}{(}{)}

\newcommand{\inter}{\overset{\circ}} 

\renewcommand{\epsilon}{\varepsilon}

\newcommand{\N}{\mathbb{N}}

\newcommand{\R}{\mathbb{R}}

\providecommand{\keywords}[1]
{
  \textbf{\textit{Keywords---}} #1
}

\providecommand{\MSCclass}[1]
{
  \textbf{\textit{MSC 2020 Classification---}} #1
}

\renewcommand{\phi}{\varphi}
\renewcommand{\tilde}{\widetilde}

\numberwithin{figure}{section}

\DeclareMathOperator*{\esssup}{ess\,sup}
\DeclareMathOperator*{\essinf}{ess\,inf}

\newcommand{\ceil}[1]{\lceil {#1} \rceil}
\newcommand{\floor}[1]{\lfloor {#1} \rfloor}

\newtheorem{theo}{Theorem}[section]
\newtheorem{lemm}[theo]{Lemma}
\newtheorem{coro}[theo]{Corollary}
\newtheorem{prop}[theo]{Proposition}

\newtheorem{defi}[theo]{Definition}
\newtheorem{fact}[theo]{Fact}

\theoremstyle{definition}
\newtheorem{rema}[theo]{Remark}
\newtheorem{notation}[theo]{Notation}

\title{Measure of maximal entropy for $H$-flows on non-compact manifolds}
\author{Anna Florio\thanks{CEREMADE, Université Paris-Dauphine, Université PSL, UMR CNRS 7534, 75016 Paris, France},  Barbara Schapira\thanks{IMAG, Université de Montpellier, UMR CNRS 5149, 34090 Montpellier, France}, Anne Vaugon\thanks{Université Paris-Saclay, CNRS, Laboratoire de mathématiques d’Orsay, 91405, Orsay, France} }
\date{version \today}

\begin{document}
\selectlanguage{english} \maketitle

\begin{abstract}
In this work, we introduce a natural class of chaotic flows on non-compact manifolds, called {\em $H$-flows}, which includes geodesic flows on non-compact manifolds with pinched negative curvature. We show that, under the additional assumption, called {\em  strong positive recurrence}, that their entropy at infinity is strictly smaller than the topological entropy, such flows admit an invariant probability measure maximizing entropy. In particular, we compare several notions of entropy in a non-compact setting.

\end{abstract}

\vspace{5pt}

\MSCclass{37A10, 37A35, 37C27, 37C35, 37C50, 37D20.}

\tableofcontents


\section{Introduction}

\subsection{Anosov flows on compact manifolds}
Anosov flows are the archetype of chaotic dynamics.
Defined by Anosov in \cite{AnosovBook}, an Anosov flow  is a differentiable flow on a compact manifold such that the tangent bundle of the manifold splits into the direct sum of three invariant subbundles, the (one-dimensional) direction of the flow,  the stable bundle, uniformly contracted in the future by the differential of the flow, and the unstable bundle, uniformly contracted in the past.
Anosov flows exhibit many remarkable properties. The {\em shadowing property}, also known as the pseudo-orbit tracing property, tells us that any path not too far from an orbit is very well approximated by  a true  trajectory. These flows are {\em expansive}: two points   on different orbits  separate one from another at some time, in the future or in the past. They also satisfy the {\em closing lemma}: near any almost closed orbit, one can find a periodic one.
As a consequence, there are infinitely many periodic orbits  and they are dense in the non-wandering set of the  flow. This is a manifestation of chaotic behaviour.

From the ergodic point of view, an Anosov flow admits infinitely many (ergodic) invariant probability measures, whose typical points describe infinitely many different typical behaviours. It is another manifestation of chaotic behaviour.

Anosov flows have {\em positive topological entropy}.  Topological entropy is the exponential growth rate, when the time $T$ goes to infinity, of the maximal number of different ``behaviours'' of orbits of period $T$. Therefore,  positive entropy is another typical property of very chaotic dynamics.

The closing lemma allows to show that for Anosov flows, this topological entropy coincides with the {\em Gurevic entropy}, i.e., the exponential growth rate when $T\to +\infty$, of the number of periodic orbits with period at most $T$.  This does not hold in general: there exist  examples of minimal flows (so, with no periodic orbits) of positive topological entropy, see \cite{Rees,BCLR}.

The \emph{variational principle} (see \cite[\S 8.2]{Walters}) states that  for any continuous dynamical system on a compact metric space, the  topological entropy equals the {\em variational entropy}, that is the supremum of the measure-theoretic entropies of all invariant (ergodic) probability measures. When it exists, a measure that realizes the supremum, i.e., whose entropy equals the topological entropy, is called a \emph{measure of maximal entropy}. Typical orbits of such a measure reflect the most chaotic behaviour of the dynamics.

For transitive Anosov flows  (on a compact manifold), in \cite{Bowen72} Bowen showed that there exists a   measure of maximal entropy, obtained as a limit of measures equidistributed on longer and longer periodic orbits. Bowen-Ruelle \cite{Bowen-Ruelle} obtain the uniqueness  (in the more general context of equilibrium states) for transitive Axiom A flows. Bowen's proof crucially relies on the \emph{specification property} of Anosov flows: for it, the compactness of the manifold is essential.
Margulis~\cite{Margulis, Margulis2} provided an alternative construction, using equidistribution of larger and larger pieces of unstable manifolds, where compactness is also used to guarantee the convergence of the construction.
Sullivan \cite{Sullivan} proposed a geometric construction that holds only in the particular case of geodesic flows in negative curvature. His construction is very robust and extends to noncompact manifolds, see for example \cite{OP,PPS,ST19}, but only for geodesic flows.


\subsection{Dynamics on non-compact manifolds, motivations}
In the non-compact setting, the picture is less clear, for many reasons. Let us highlight some of the main difficulties and questions.
\begin{itemize}
    \item The notion of  Anosov flow does not generalize easily to the non-compact case. The definition and several resulting properties, closing lemma in particular, depend  strongly on the compactness of the manifold.
    \item Under what kind of assumptions do   topological entropy and   Gurevic entropy coincide?
    \item There exist many different, useful notions of measure-theoretic entropy, such as Kolmogorov-Sinai entropy, Katok entropy, Brin-Katok upper and lower local entropies. Do they coincide in general?
    \item Thanks to \cite{HanKit}, the variational principle still holds for a non-compact manifold, and so we can still talk about measures of maximal entropy. Nevertheless, by lack of compactness, there is no easy argument to ensure neither the existence nor the uniqueness of an invariant probability measure maximizing entropy.
\end{itemize}

Although  these questions do not have clear answers in general, there are two very important  classes of (non-compact) dynamical systems for which some results exist: suspension flows over a symbolic dynamics with countable alphabet on the one hand, and  geodesic flows on the unit tangent bundle of (non-compact) negatively curved manifolds on the other hand.
In both cases, different relevant notions of entropy coincide. See for example~\cite{OP, PPS, ST19, GST} for the geodesic flow, and \cite{VereJones, Sarig99, MauUrb, Sarig03, BBG06, CyrSar, BBG} for subshifts and suspension flows over infinite alphabets.

In the case of geodesic flows, under an assumption called {\em strong positive recurrence}, see \cite{ST19, GST}, one can prove the existence of a (unique) invariant probability measure that maximizes entropy.
This notion appeared in different ways in different contexts: for Markov shifts~\cite{VereJones, Sarig03}, for geodesic flows~\cite{ST19,GST}, for diffeomorphisms on closed manifolds~\cite{BuzCroSar}, in geometric group theory \cite{ACT15,Yang}, for Hénon-like diffeomorphisms~\cite{Berger}, \dots

{\em Strong positive recurrence} means that the exponential growth of the complexity of trajectories that spend most of the time close to infinity (or to any other bad zone, in different contexts), also called {\em entropy at infinity}, is strictly smaller  than the global exponential growth of the dynamics, i.e., the {\em entropy}. In \cite{GST}, the authors introduce different notions of entropy at infinity, in terms of periodic orbits, or invariant measures, or critical exponent, i.e., orbital growth of the fundamental group, and show that they coincide. However, as in \cite{ST19}, the core of the proof of the existence (and uniqueness) of a measure of maximal entropy relies on geometric arguments specific to geodesic flows, since it uses  the {\em critical exponent} and the {\em critical exponent at infinity}. These critical exponents make sense only in this geometric setting, and the measure of maximal entropy is obtained in a geometric way, through the so-called Patterson-Sullivan construction.

The initial motivation of this work was to propose a general  dynamical framework where this kind of results holds. More precisely, our initial goals were the following ones.

\begin{enumerate}
    \item\label{uno:lista} Propose a relevant definition of Anosov flow in the non-compact setting, that includes geodesic flows of non-compact hyperbolic manifolds.
    \item\label{due:lista} Compare different notions of entropy in this context.
    \item\label{tre:lista} Give one or several definition(s) of entropy at infinity and compare them.  \item\label{quattro:lista} Under a \emph{strongly positive recurrent} hypothesis, construct a measure of maximal entropy.
    \item\label{cinque:lista} Study the uniqueness and ergodicity of this measure.
    \item\label{sei:lista} Construct new families of examples.
\end{enumerate}

\subsection{Our results}
In the present work, we address points \ref{uno:lista}, \ref{due:lista}, \ref{tre:lista} and \ref{quattro:lista}. We postpone point \ref{sei:lista} to \cite{FSV2}.

The definition of an Anosov flow (on a compact manifold) strongly uses the Riemannian metric. It does not matter in the compact case, where all metrics are equivalent, but becomes an important choice in the noncompact setting.

In \cite{Das-etc}, the authors avoid this choice by proposing a definition of \emph{topological Anosov flows}, relying on the idea that  two points $(x,y)\in M^2$ are said {\em  close} if they belong to a small neighbourhood of the diagonal in $M\times M$.   A topological Anosov flow is then a topologically expansive flow, which satisfies a topological shadowing property. Nevertheless, they observe that this definition does not have clear relations with the usual one, even in the compact case.

In Section~\ref{section:H-flow} we address point~\ref{uno:lista} with another approach, inspired by \cite{CS10}. We propose an axiomatic definition of a \emph{H-flow} on a (non compact) manifold, by requiring  dynamical properties that mimic properties of compact uniformly hyperbolic flows in the non compact setting. These properties are stronger than those of \cite{CS10}. Observe that this axiomatic definition finds echos in \cite{Mann}, where the authors introduce the notion of Anosov-like group actions by asking for the occurrence of certain properties.
More precisely, in Definition~\ref{def:H-flow}, we define a \emph{$H$-flow} on a (non compact) Riemannian   manifold $M$ as a $C^1$ flow $\phi$ with lipschitz bounds that is \emph{transitive}, \emph{expansive}, satisfies a \emph{closing lemma} and a suitable \emph{shadowing property}. We observe in Theorem~\ref{geodesic-flow-H-flow}  that the geodesic flow of a negatively curved manifold with pinched negative curvature, a lower bound on the injectivity radius and a full nonwandering set is a $H$-flow.

Observe that we require a $H$-flow to be transitive on the whole manifold. It is likely that we could weaken this property, defining what could be thought of as a noncompact Axiom A flow, with good uniformly hyperbolic behaviour in restriction to a closed invariant attractor  instead of  a noncompact generalization of an Anosov flow on the whole manifold, and get the same kind of results. This should heuristically work but presents some technical difficulties that we hope to solve in the future.

The definition of H-flow uses a little bit the Riemannian metric and more deeply the induced distance $d$ on $M$, but not so deeply as the usual definition of Anosov flow. A natural alternative approach would be to adapt the definition of Anosov flow in the noncompact setting, by requiring  some contraction and expansion of stable/unstable bundles, with a uniform control on the angle between them. We refer to \cite{FSV2} to compare this approach with the present one, and show that a noncompact Anosov flow  with further natural uniform assumptions is a $H$-flow.
In \cite{FSV2} we will also construct new families of examples, besides geodesic flows, of H-flows on non compact manifolds, providing then answers to point \ref{sei:lista}. These examples will be non compact versions of the systems studied in \cite{FouHas} and \cite{FouHasVau}.

\par
On the way towards the construction of a  measure of maximal entropy for $H$-flows, we needed to define, clarify and compare different notions of topological and measure-theoretic entropies in this noncompact setting. We describe them briefly in this introduction and refer to sections \ref{sec:entropies}, \ref{known}  \ref{chords} for  details.

Let $\phi$ be a flow on $M$ and let $\mu$ be a $\phi$-invariant probability measure. The classical \emph{Kolmogorov-Sinai} entropy $h_\mathrm{KS}(\mu)$ measures the asymptotic growth of the average information of a finite measurable partition iterated under the dynamics (see Definition \ref{KS entropy}).
We will also  largely use the \emph{Katok entropy} $h_{\mathrm{Kat}}(\mu)$, that uses Bowen's definition of dynamical balls. A dynamical ball $B(x,\epsilon,T)$ is the set of points $y\in M$ whose trajectory follows the one of $x$ at the precision $\epsilon$ during a time $T$. Katok entropy is the asymptotic growth rate of the minimal number of dynamical balls needed to cover a set  of a given positive measure, see Definition \ref{Katok entropy}. Given $\mu$, we can also consider the asymptotic lower and upper exponential decay rates of the measure $\mu$ of a typical dynamical ball, called the \emph{Brin-Katok lower and upper local entropies} $\underline{h}_{BK}(\mu)$ and $\overline{h}_{BK}(\mu)$, see Definition \ref{Brin Katok entropies}.
Collecting those of the results of \cite{Katok, BK, Riquelme} that hold in the noncompact setting, we  observe in Corollary~\ref{entropie-var} that when a measure $\mu$ is ergodic, all these notions of entropies coincide.
Comparison between these entropies for a nonergodic measure is less clear, partial results are recalled in Theorem~\ref{entropies-ergodic-measures}.

The well-known variational principle asserts (in the compact setting) that the {\em topological entropy} of a dynamical system coincides with the supremum of Kolmogorov-Sinai entropies over all invariant measures. In the noncompact setting, the historical notion of topological entropy through open covers, due to Adler-Konheim-Weiss~\cite{AKW} is too often infinite, and therefore not relevant. Bowen's definition of topological entropy~\cite{Bowen} strongly relies on the metric on $M$. To bypass this dependence, the noncompact variational entropy proven by Handel and Kitchens~\cite{HK} shows that for a dynamical system on a noncompact manifold, the supremum of Kolmogorov-Sinai entropies over all invariant probability  measures is equal to the infimum of Bowen metric entropies over all distances defining the topology.
Therefore, as in \cite{ST19,GST}, we call  \emph{variational entropy} of the flow, denoted by $h_\mathrm{var}(\phi)$, this supremum of Kolmogorov-Sinai measured entropies over all $\phi$-invariant probability measures, see Definition~\ref{variational entropy}. When it exists, a measure realizing the supremum is called  a \emph{measure of maximal entropy}. The existence of such a measure allows to build orbits that achieve in some sense the maximal possible chaotic behaviour of the dynamics.

In hyperbolic dynamics, topological entropy is often measured through the so-called \emph{Gurevic entropy} $h_\mathrm{Gur}(\phi)$, i.e., the exponential growth rate of the number of periodic orbits of period at most $T$, when $T\to +\infty$.
We study its basic properties in the context of $H$-flows in Theorem~\ref{theo:Gurevic}.
In Theorem~\ref{comparison-Kat-BK}, we show that for a H-flow $\phi$, the Katok entropy of a $\phi$-invariant probability measure is always bounded from above by the Gurevic entropy:
\[
 h_\mathrm{Kat}(\mu)\le h_{\mathrm Gur}(\phi)\,,\]
which implies immediately
\[
h_\mathrm{var}(\phi)\le h_{\mathrm Gur}(\phi)\,.
\]

In Section~\ref{chords}, we introduce a notion of entropy that has not been used yet, to our knowledge.  The \emph{chord entropy} $h_{\mathcal{C}}(\phi)$ of a H-flow is the exponential growth rate of the maximal cardinality of a separated set of chords from (the neighbourhood of) a point to (the neighbourhood of) another point.  We prove in Theorem~\ref{thm h cord=h gur} that this entropy coincides with Gurevic entropy :
\[
h_{\mathcal{C}}(\phi)=h_{\mathrm Gur}(\phi)\,.
\]
It turns out that this definition through chords is much more flexible, and therefore more convenient to use in several arguments.

The results mentioned above answer question~\ref{due:lista}.  In Sections~\ref{sec:entropy-at-infinity}, \ref{chords infinity} and \ref{gur et chords infinity} we address the different possible definitions of entropies at infinity, as proposed in point~\ref{tre:lista}. We are particularly concerned with two of them:
\begin{itemize}
    \item the \emph{Gurevic entropy at infinity} $h_\mathrm{Gur}^\infty(\phi)$  is the exponential growth rate of the number of periodic orbits that spend most of their time outside a large compact set (see Definition~\ref{definition_gurevic_infini});
    \item the \emph{chord entropy at infinity} $h_{\mathcal{C}}^\infty(\phi)$ is the exponential growth rate of the number of separated paths which remain outside a large compact set.
\end{itemize}
In Theorem~\ref{theorem gur infty egal chords infty}, we prove that, for a H-flow, these two notions of entropy at infinity   coincide:
\[
h_\mathrm{Gur}^\infty(\phi)=h_{\mathcal{C}}^\infty(\phi)\,.
\]

\par
The heart of the paper is the construction of a measure of maximal entropy, as asked in the above point~\ref{quattro:lista}. As in \cite{ST19,GST} for geodesic flows, we say that a $H$-flow  $\phi$ is {\em strongly positively recurrent} if its Gurevic entropy at infinity is strictly smaller than the Gurevic entropy:
\[
h_{\mathrm Gur}^\infty(\phi)<h_{\mathrm Gur}(\phi)\,.
\]
All results mentioned above are interesting and useful. However, the main result of the paper is the following Theorem.

\begin{theo}\label{theo:main}
Let $\phi\colon M\to M$ be a $H$-flow on a Riemannian manifold $(M,g)$ such that $h_\mathrm{Gur}^\infty(\phi)<h_\mathrm{Gur}(\phi)$. Then, there exists a $\phi$-invariant probability measure $m_\mathrm{max}$ on $M$ maximizing entropies:
\[
h_\mathrm{KS}(m_\mathrm{max})=\underline{h}_\mathrm{BK}(m_\mathrm{max})=\overline{h}_\mathrm{BK}(m_\mathrm{max})=h_\mathrm{Kat}(m_\mathrm{max})=h_\mathrm{Gur}(\phi)=h_\mathrm{var}(\phi)\,.
\]
\end{theo}
Note that we are interested in $H$-flows, but we could have defined a very similar notion of $H$-diffeomorphisms and proven the same result. We let the verification to the interested reader. Moreover,  as said above and suggested to us by L. Flaminio, it is likely that this Theorem could be extended to flows that satisfy transitivity, expansivity, closing lemma, and shadowing on a closed invariant subset of the Riemannian manifold. This should be checked.

\par
The main idea for the proof of Theorem~\ref{theo:main} is inspired by the approach of Bowen in~\cite{Bowen72}. We construct a sequence of $\phi$-invariant probability measures, each one obtained by normalizing the sum of all measures supported on periodic orbits of (almost) given period. Up to extracting a subsequence, we consider a weak limit of this (sub)sequence. One of the main difficulties on  non compact spaces arises here: the sequence could loose its whole mass at infinity and converge to the zero measure. Strong positive recurrence prevents a total loss of mass: the limit measure $m_\infty$ is non zero. Therefore it can be renormalized into a probability measure $m_\mathrm{max}$, and the latter is the good candidate to be a measure of maximal entropy. It remains to compute carefully its entropy. This is done through  a strong uniform control of the $m_\mathrm{max}$ measure of all dynamical balls of the manifold. More precisely, in Theorem~\ref{prop_boule_dynamique}, we prove a rigorous version of the following heuristics : on every  compact set $K$, up to uniform constants, for every $x\in K$,  the measure of every  dynamical ball $B(x,\epsilon,T)$ such that  $\phi_T(x)$ belongs in $K$ satisfies
\[
m_\mathrm{max}(B(x,\epsilon,T))\asymp \frac{1}{\mathcal N_{\mathcal C}(x, T,\delta)}\,,
\]
where the denominator is the cardinality of a $\delta$-separated set of chords of time-length $T$ from a neighbourhood of $x$ to itself. This strong uniform statement allows easily to compute the entropy of the measure and deduce Theorem~\ref{theo:main}.

Theorem~\ref{prop_boule_dynamique} is proven through subtle subadditivity statements: Propositions~\ref{prop cirm 1}   and \ref{prop cirm 2}.
These subadditivity statements are quite classical for hyperbolic flows on compact manifolds, or in a symbolic context and say essentially that the number of periodic orbits of period $L$ is essentially comparable to the product of the number of periodic orbits of period $L-T$ times the number of periodic orbits of period $T$. This is done by cutting / concatenating periodic orbits at good points.  The main difficulty that we have to deal with is that a very long periodic orbit does not necessarily come back often in a given compact set, so that it is very hard to play the usual game and cut it into two smaller periodic orbits intersecting the same compact set. Indeed, to compare the number of long periodic orbits intersecting a given compact set with the number of shorter periodic orbits, we need to show that most periodic orbits come back almost regularly in a compact set. This is done in Proposition~\ref{prop_pre_coding}.

\par Our approach does not allow us yet to explore whether such a maximal entropy measure is unique or ergodic.   
We would like to address this question in the future. Moreover, as in the case of geodesic flow, it is likely that the strong positive recurrence assumption is not necessary for the existence of a measure of maximal entropy. This should also be investigated.

\par
The structure of the paper is the following.
In Section~\ref{section:H-flow}, we define $H$-flows. In Section~\ref{sec:entropies}, we define different entropies of a $\phi$-invariant probability measure. In Section~\ref{known}, we compare them. In Section~\ref{chords}, we introduce the notion of chord entropy and compare it to the Gurevic entropy. Section~\ref{sec:subadditivity} is the technical heart of the paper. Under the strong positive recurrence assumption, we construct the measure that is the natural candidate to be the measure of maximal entropy, and we deduce from its existence subtle subadditivity properties on the number of periodic orbits of a given period. These properties are classical in a compact setting, but absolutely non trivial without compactness. In the last Section~\ref{sec:finale}, we deduce   that the measure that we constructed is the required measure of maximal entropy.
\bigskip
\par
{\bf Aknowledgements}\\
The authors thank warmly J. Buzzi, Y. Coudene, S. Crovisier, L. Flaminio, S. Gouëzel, M. Herzlich, D. Hulin for enlightening discussions and useful references.  B. Schapira thanks Dr Carton, Dr Dhalluin and the CHP St Grégoire.
The authors have been supported by the ANR grant GALS ANR-23-CE40-0001. A. Florio has been supported by the ANR grant CoSyDy ANR-21-CE40-0014 and  PEPS project ``Jeunes chercheurs et jeunes
chercheuses'' 2025. B. Schapira has been supported by the  ANR Grant GOFR ANR-22-CE40-0004 and the IUF. A. Vaugon has been supported by the ANR grant CoSy ANR-21-CE40-0002.


\section{\texorpdfstring{$H$}{TEXT}-flows on non-compact manifolds}\label{section:H-flow}


We  propose here a definition of hyperbolic flows on non-compact Riemannian manifolds through some important dynamical properties that they satisfy. The definition involves   the distance, but not the Riemannian structure.
We call them {\em $H$-flows}, thinking in particular to {\em hyperbolic flows}, but this word has too many significations so that we prefer to avoid it.

We first describe the dynamical properties used in the description of $H$-flows.

\subsection{Notations}

Let $(M,g)$  be a complete connected Riemannian manifold and denote by $d$ the associated distance.
Let $(\phi_t)_{t\in\R}$ be a $\mathcal C^1$-flow on $M$ generated by a vector field $X$.

If $c: [a,b]\to M$ is a piece of orbit of the flow $\varphi$, i.e. $c(t)=\phi_t(x)$ for some $x\in M$, and $t\in[a,b]$,   denote by $\ell(c)=b-a$ its ``length''. By extension, if $\gamma$ is a periodic orbit (possibly with multiplicity), we denote by $l(\gamma)$ its period (with multiplicity); we will sometimes refer to $\ell(\gamma)$ also as the ``length'' of the periodic orbit.
\label{p:ell}
For $A\subset M$ any measurable subset, and $\gamma$ a periodic orbit, we define also
$\ell(\gamma\cap A)=\{t\in [0,\ell(\gamma)], \gamma(t)\in A\}$.

\subsection{Bounds on  the flow}
We require that for every $\tau\in[-1,1]$, the time $\tau$-map  $\varphi_\tau$  of the flow is {\em Lipschitz continuous} with uniform Lipschitz constant: there exists a constant $lip(\phi)$ such that for all $x,y\in M$, and $\tau\in[-1,1]$,
\begin{equation}\label{lipsch}
d(\phi_\tau(x),\phi_\tau(y))\le lip(\phi)\, d(x,y)\,.
\end{equation}

We also assume that there exist $0<a\le b$     such that for every $x\in M$, there exists $\rho>0$ such that for every $t\in [-\rho,\rho]$,  we have
\begin{equation}\label{eqn:minoration}
a|t|\le d(x,\phi_t(x)) \le b|t|
\end{equation}
We could probably weaken the constraint on the lower bound, with the existence of such a constant $a$ on every compact set. On the contrary, we really need $b$ to be uniform.
Notice that the lower bound is true locally for $t\in[-\rho,\rho]$, whereas triangular inequality implies that the upper bound $d(x,\phi_t(x)) \le b|t|$ is satisfied for every $t\in\R$.

Property (\ref{eqn:minoration}) can be obtained by a uniform control on the norm of the infinitesimal generator $X$ of the flow as stated in the following lemma.

\begin{lemm}
If there exist $a',b'>0$ such that for every $x\in M$,
\begin{equation}\label{norm_bounded}
0<a'\le \|X(x)\|\le b' <\infty\, ,
\end{equation}
then there exist $0<a<b$ such that property (\ref{eqn:minoration}) is satisfied.
\end{lemm}

\begin{proof}
Let $x\in M$. Consider the exponential map $\exp_x : U\subset T_xM\to V\subset M$ and let $Y=\exp_x^*X$. Choose an orthonormal basis $(e_1,\cdots, e_n)$ on $T_xM$   endowed with the scalar product $g_x$ given by the metric. This allows to identify $T_xM$ with the canonical euclidean space $\R^n$. Without loss of generality, we may assume $Y(0) = \| X(x) \| \frac{\partial}{\partial x_1}$.
By definition of the exponential map, for every point $y\in V$, $d(x,y) = \| \exp_x^{-1}(y)\|_{\mathrm{eucl}}$ where $\|\cdot\|_\mathrm{eucl}$ is the Euclidean norm on $\mathbb R^n$.
Write $Y=(Y_1,\dots,Y_n)$.
One may shrink $U$ to ensure that, for all $v\in U$, $\| Y(v) \|_\mathrm{eucl}\leq 2b'$ and $Y_1(v)\geq a'/2$.
Therefore, for $t$ small enough
\[
\frac{a'}{2}|t|\leq \left|\int_0^t Y_1(\exp_x^{-1}(\phi_s(x)))ds \right| \leq d(x,\phi_t(x))= \left\|\int_0^t Y(\exp_x^{-1}(\phi_s(x)))ds \right\|_\mathrm{eucl} \leq 2b'|t|\,.
\]
\end{proof}


\subsection{Topological transitivity}

\begin{defi}[Transitivity]\label{def-transitive}
The flow $\phi$ is {\em topologically transitive} on $M$  if for all open sets $U,V\subset M$ and every $T>0$, there exists $t\ge T$ such that $\phi_t( U)\cap V\neq \emptyset$.
\end{defi}

Most of the time, we will  use the stronger property of Lemma~\ref{transitivite prop}, where   $t$ can be made almost constant for open sets $U,V$ inside a given compact set $K$.

\subsection{Closing lemma}

\begin{defi}[Closing lemma]\label{closing lemma} The flow satisfies the {\em closing lemma} if for every $\epsilon>0$ and $x\in M$ there exists $\delta>0$ and $T_\mathrm{min}>0$ such that for every $y\in B(x,\epsilon)$ and $t\geq T_\mathrm{min}$ with $d(\phi_t( y),y)\le \delta$, there exist $z\in B(y,\epsilon)$ and $\tau \in [t-\epsilon,t+\epsilon]$ such that $\phi_\tau( z)=z$ and $d(\phi_s (z), \phi_s (y))\le \epsilon $ for every $0\le s \le t $.
\end{defi}

\begin{figure}[ht]
    \centering
    \includegraphics{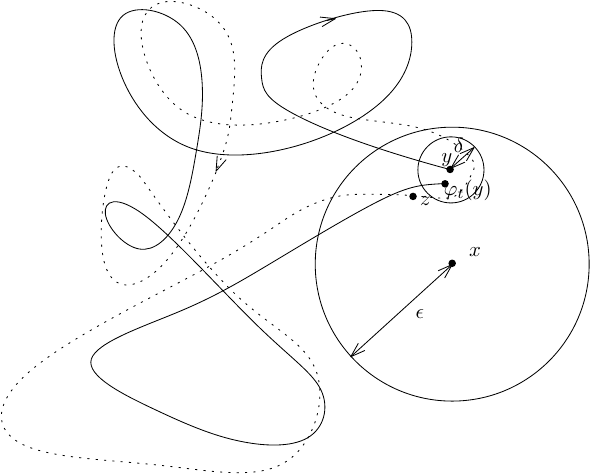}
    \caption{Closing Lemma}
    \label{fig:closing_lemma}
\end{figure}

\subsection{Expansivity}

We follow \cite[Definition 1.7.2]{Fisher-Hasselblatt}.
\begin{defi}[Expansivity]\label{def:expansive} The flow $\phi$ is said expansive if for every $\nu>0$ there exists $\epsilon>0$ such that for all $x,y\in M$, if there exists a continuous map $s:\mathbb{R}\to\mathbb{R}$ such that $s(0)=0$ and for every $t\in\mathbb{R}$, $d(\phi_t(x), \phi_{s(t)}(y))\le \epsilon$, then there exists $\tau\in [-\nu,\nu]$ such that $y=\phi_\tau (x)$.
\end{defi}

 \begin{prop}\label{geod-flow-expansive} The  geodesic flow $(g_t)_{t\in\R}$ on the unit tangent bundle $M=T^1N$ of a (not necessarily compact) pinched negatively curved manifold $N$ whose injectivity radius is bounded from below by a positive constant is expansive with respect to the distance on $T^1N$ induced by the Sasaki metric on $TTN$.
 \end{prop}

 We refer for example to \cite[p.18-19]{PPS} for definitions and elementary useful facts on the Sasaki metric.

\begin{proof} Let $\rho>0$ be the positive infimum of all injectivity radii $\rho_{inj}(x)$ at all points $x\in N$.
Choose $\epsilon<\rho/2$. Consider two vectors $v,w\in M=T^1N$ and a map $s:\R\to \R$ that  satisfy $d(g_tv,g_{s(t)} w)<\epsilon$ for every $t\in\R$.

Denote by $\tilde N$ the universal covering of $N$. We claim that there exist $\tilde v, \tilde w\in T^1\widetilde N$ such that $d(g_t\tilde v,g_{s(t)} \tilde w)<\epsilon$ for every $t\in\R$.
Indeed, choose $\tilde v$ and $\tilde w$ two lifts of $v$ and $w$ such that $d(\tilde v , \tilde w)<\epsilon$. Assume by contradiction that there exists $t$ such that $d(g_t\tilde v,g_{s(t)} \tilde w) =\epsilon$. Then, as $d(g_tv,g_{s(t)} w)<\epsilon$, we can find a lift $\tilde w_t$ of $g_{s(t)} w$ such that $d(g_t\tilde v,\tilde w_t)<\epsilon$. Therefore, the two lifts  $g_{s(t)} \tilde w$ and $\tilde w_t$ of $g_{s(t)} w$ satisfy $d(g_{s(t)} \tilde w , \tilde w_t) <2\epsilon <\rho$. This is in contradiction with the definition of injectivity radius.

Now, recall that $\widetilde N$ is a Hadamard manifold with pinched negative curvature. It   implies  that there are no distinct parallel geodesics. Therefore, as   $d(g_t\tilde v,g_{s(t)} \tilde w)<\epsilon$ for every $t\in\R$, there exists $\tau\in\R$ such that $w=g_\tau v$. Moreover, we have $|\tau|<\epsilon$ by assumption.
\end{proof}

This condition of positive lower bound on the injectivity radius can fail for different reasons. When it fails, the expansivity property can still hold or it can also fail.
If the manifold admits one cusp, and no other end, expansivity still holds. If there are infinitely many periodic orbits with lengths arbitrarily small, then one can build pairs of distinct orbits that stay arbitrarily close, one turning around a small periodic orbit and not the other, so that expansivity fails.
\subsection{Shadowing properties}

\begin{defi}[Finite exact shadowing]\label{weak shadowing property}
The flow $\phi\colon M\to M$ satisfies the \textit{finite exact shadowing property} if for every compact subset $K\subset M$, every $\delta>0$ and every integer $N\in\mathbb{N}^*$, there exists $\eta>0$ such that the following shadowing holds. Given $N$ orbits  $(\phi_t (x_i))_{0\le t\le t_i}\in K$ for $i=1,\dots,N$ starting in $K$ that satisfy
\[
d(\phi_{t_i}(x_i),x_{i+1})<\eta\,, \forall i=1,\dots, N-1\, ,
\]
there exists $y\in M$ such that for  every $1\le i\le N$, and $0\le s\le t_i$,
\[
d(\phi_{s+\sum_{j=1}^{i-1}t_j }(y),\phi_s(x_i))<\delta \,.
\]
\end{defi}

\begin{figure}[ht]
    \centering
    \includegraphics{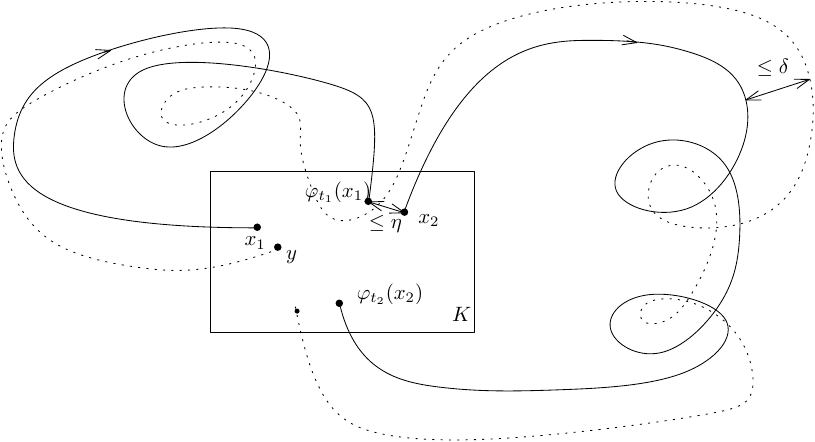}
    \caption{Finite exact shadowing for $N=2$ }
    \label{fig:finite_exact_shadowing}
\end{figure}
See Figure~\ref{fig:finite_exact_shadowing}.
Here, the word {\em finite} refers to the finite number of orbits defined on finite intervals of time and the word {\em exact} to the fact that the orbit of $y$ follows exactly the orbits of the $(x_i)$, without reparametrization of time,  by contrast with the usual definitions of shadowing, see for example \cite[Definition 1.5.29]{Fisher-Hasselblatt} and \cite[Definition 5.3.1]{Fisher-Hasselblatt}.
For the sake of clarity, let us compare this shadowing property with the assumptions of \cite{CS10}.
In~\cite{CS10}, the authors work with flows satisfying  transitivity, closing lemma and a local product structure. This property, defined below, is
stronger than the finite exact shadowing of Definition \ref{weak shadowing property}, as proven in Lemma~\ref{local-product-versus-finite-shadowing} below.

\begin{defi}[Local product structure]\label{local poduct structure} We say that $\phi$ satisfies the local product structure if for every $x_0\in X$ there exits $\epsilon_0>0$ such that, for every $\delta>0$ there exists $\eta>0$ satisfying the following property. For all $x,y\in B(x_0,\epsilon_0)$ such that $d(x,y)\leq\eta$, there exists $z\in B(x_0,\epsilon_0)$ and $|\tau|<\delta$ such that
$d(\phi_tx,\phi_tz)\le \delta$ for all $t\le 0$ and $d(\phi_{t+\tau}y, \phi_t z)\le \delta$ for all $t\ge 0$. We will use the standard notation $z= \langle x,y\rangle$.
\end{defi}

\begin{lemm}\label{local-product-versus-finite-shadowing}
Let $\phi\colon M\to M$ be a $C^1$-flow satisfying \eqref{eqn:minoration}. If $\phi$ satisfies the local product structure (see Definition \ref{local poduct structure}), then $\phi$ satisfies the finite exact shadowing property (see Definition \ref{weak shadowing property}).
\end{lemm}

\begin{proof}
It is enough to prove it for $N=2$, and the general case follows by successive uses of the $N=2$ version.
Let $K$ be a compact subset and fix $\delta>0$.
Let $\delta' = \frac{\delta}{1+b}$ where $b$ comes from the upper bound  constant in (\ref{eqn:minoration}).
As $\overline{B(K,1)}$ is compact, using the local product structure, we obtain $\eta>0$ such that for all $x,y\in \overline{B(K,1)}$ such that $d(x,y)\leq\eta$, there exists $z\in M$ and $|\tau|<\delta'$ such that
$d(\phi_t x,\phi_t z)\le \delta'$ for all $t\le 0$ and $d(\phi_{t+\tau}y, \phi_t z)\le \delta'$ for all $t\ge 0$.

Fix $x_1,x_2\in K$ and $t_1,t_2\geq 0$ such that $d(\phi_{t_1}(x_1),x_2)<\eta$, that is, as in the definition of finite exact shadowing. Apply then the local product structure for $\phi_{t_1}(x_1)$ and $x_2$: it gives us $z\in M$ and $|\tau|<\delta'$ such that $d(\phi_{t+t_1} (x_1),\phi_t (z))\le \delta'$ for all $t\le 0$ and $d(\phi_{t+\tau} (x_2), \phi_t (z))\le \delta$ for all $t\ge 0$.

Set $y=\phi_{-t_1}(z)$.
It satisfies $d(\phi_t(y),\phi_t(x_1))\le\delta'<\delta$ for $0\le t\le t_1$ and, by \eqref{eqn:minoration},
\[d(\phi_{t+t_1} (y),\phi_t(x_2))=d(\phi_{t} (z),\phi_t(x_2))\le d(\phi_t (z),\phi_{t+\tau}(x_2))+d(\phi_{t+\tau}(x_2),\phi_t (x_2))\le \delta'+b|\tau|<\delta\] for $t\ge 0$. The result follows.
\end{proof}

\subsection{\texorpdfstring{$H$}{TEXT}-flows}
\begin{defi}\label{def:H-flow}
A $H$-flow is a $C^1$-flow $(\phi_t)_{t\in\R}$ on a complete Riemannian manifold satisfying the following properties\,:
\begin{enumerate}
\item for every $\tau\in [-1,1]$, the time $\tau$ map $\phi_\tau$ is Lipschitz, as in (\ref{lipsch});
 \item the parametrization of the flow satisfies the bounds   (\ref{eqn:minoration});
 \item the flow is  topologically transitive, see Definition~\ref{def-transitive};
 \item  the flow  satisfies the finite exact shadowing property, see Definition~\ref{weak shadowing property};
 \item the flow satisfies the closing lemma, see Definition~\ref{closing lemma};
 \item the flow  is expansive, see Definition~\ref{def:expansive}.
 \end{enumerate}
\end{defi}

\begin{rema}
By Lemma~\ref{local-product-versus-finite-shadowing}, an expansive flow satisfying the assumptions (transitivity, local product and closing lemma) of \cite{CS10} and the controls \eqref{lipsch} and \eqref{eqn:minoration}  is a $H$-flow.
\end{rema}

\begin{theo}\label{geodesic-flow-H-flow} Let $N$ be a Riemannian manifold with pinched negative curvature, bounded from above and from below by uniform negative constants. Assume that the injectivity radius is bounded from below by a positive constant. Let $M=T^1N$ be its unit tangent bundle, endowed with the distance associated with the Sasaki metric.
If the geodesic flow
 $(g_t)_{t\in\R}$ is topologically transitive, then it is  a $H$-flow.
\end{theo}
\begin{proof} The bound (\ref{lipsch}) saying that the geodesic flow on $M$ at time $\tau$, for $\tau\in[-1,1]$, is Lispchitz continuous, with a uniform Lipschitz constant, will follow from a uniform bound on the differential of the geodesic flow on the unit tangent bundle $M$ for the Sasaki metric.

First note that, as the sectional curvature is bounded, the Riemann curvature tensor, which is completely determined by the sectional curvature, is also bounded (see for instance \cite[Chapter III, Theorem~3.8]{GHL} or \cite[Section~6.1]{buserkarcher}).
Now, the differential of the geodesic flow on the tangent bundle $M$ is
\[
D_{(p,v)} g_t (X,Y) = ( J(t), J'(t))\, ,
\]
where $J$ is the Jacobi field along the geodesic $\gamma$ defined by $(\gamma(0),\gamma'(0))=(p,v)$ with initial condition $J(0) = X$ and $J'(0) = Y$ (see for instance \cite[Lemma~1.13]{Ballmann}).
Therefore, in order to bound uniformly the differential of the geodesic flow, it is enough to bound uniformly $\Vert J(t)\Vert$ and $\Vert J'(t)\Vert$ for $t\in[-\tau,\tau]$. From a direct computation (see \cite[Proof of Proposition~1.19]{Ballmann}), we obtain
\begin{align*}\left(||J(t)||^2 + ||J'(t)||^2 \right)'
&= 2\left\langle J'(t), J(t)+J''(t)\right\rangle
 = 2\left\langle J'(t), J(t)- R(J(t),\gamma'(t))\gamma'(t)\right\rangle\\
&\leq 2C ||J'(t)|| ||J(t)|| \leq  C\left(||J(t)||^2 + ||J'(t)||^2 \right)
\end{align*}
for some $C>0$ where $R$ is the curvature tensor.
Using Grönwall's inequality, we obtained the desired bound.
 \color{black}

The bounds (\ref{eqn:minoration}) follow from an elementary computation\,: in the classical horizontal/vertical coordinates of $TTN$, the geodesic vector field generating the geodesic flow satisfies $X(v)=(v,0)$ so that its Sasaki norm satisfies $\|X(v)\|=\|v\|=1$.
Topological transitivity is an assumption. Finite exact shadowing follows from the fact that geodesic flows satisfy a local product structure, and from Lemma \ref{local-product-versus-finite-shadowing}.  Closing lemma is due to Eberlein \cite{Eberlein} in nonpositive curvature, and a short proof in the particular case of negative curvature is given in \cite{CS10}. Expansivity follows from
Proposition \ref{geod-flow-expansive}.

\end{proof}
\subsection{Properties of \texorpdfstring{$H$}{TEXT}-flows}\label{section_technical_prop_h_flows}

In the whole section, $\phi$ is a $H$-flow on a complete Riemannian manifold $M$, and $d$ is a Riemannian distance.

\begin{lemm}
Let $\phi$ be a $H$-flow. For every nonempty open set $U\subset M$, there exists a periodic orbit that intersects $U$.
\end{lemm}
\begin{proof} Choose $x\in U$ and $\epsilon>0$  such that $B(x,2\epsilon)\subset U$.
The closing lemma (see Definition \ref{closing lemma}) associates to $x$ and $\epsilon$ some $\delta>0$ and $T_\mathrm{min}>0$.
Without loss of generality, $\delta\leq\epsilon$.
By transitivity (see Definition \ref{def-transitive}), there exists $t\ge T_\mathrm{min}$ such that $B(x,\delta/2)\cap\phi_t(B(x,\delta/2))\neq \emptyset$. Choose some $\phi_t(y)$ in this intersection, and apply the closing lemma \ref{closing lemma} to $y$.
As $d(\phi_t( y),y)<\delta\leq\epsilon$, there exists a periodic point $z\in B(y,\epsilon)\subset B(x,2\epsilon)\subset U$. This concludes the proof. \end{proof}

\noindent \textbf{Notation.} For a compact set $K\subset M$ whose interior is nonempty, denote by $\tau_K$ the minimal period of a periodic orbit (possibly a multiply covered orbit) with period $\geq 1$ intersecting $K$. Such an orbit always exists for a $H$-flow as the interior of $K$ is nonempty.
\label{p:tau_K}

\begin{lemm}[Separation of orbits]\label{lemma on same po} Let $\phi$ be a $H$-flow on $M$.
For every $\nu>0$, there exists $\tau_0>0$ such that  for every $\tau_1\ge 1$, there exists $\epsilon_1>0$ such that the following holds. For all  periodic points $x_0,x_1$ with respective periods $T_0,T_1\geq \tau_1$, satisfying $0\leq T_1-T_0\leq \tau_0$, if, for all $s\in[0,T_0-\tau_1]$ we have
\begin{equation}\label{eqn:separation}
d(\phi_s(x_0),\phi_s(x_1))<\epsilon_1
\end{equation}
then  $x_1=\phi_u(x_0)$ for some $u\in[-\nu,\nu]$.
\end{lemm}

\begin{proof}
Fix $\nu>0$. Let $\epsilon'>0$ be given by Definition~\ref{def:expansive} of expansivity associated with $\nu$.  Let $\tau_0=\epsilon'/2b$, where $b>0$ is the constant of \eqref{eqn:minoration}.
Fix $\tau_1\geq 1$. Let $C(\tau_1)=C\geq 1$ be a uniform Lipschitz constant for every $\phi_t$ with $t\in[0,\tau_1]$ (such a constant exists by the Lipschitz condition in the definition of $H$-flows). Let $\epsilon_1=\epsilon'/2C$. Note that $\epsilon_1\leq \epsilon'/2$.

Consider $x_0$, $x_1$, $T_0$ and $T_1$ as in the statement of the lemma. Then, for every $t\in[T_0-\tau_1,T_0]$,
\[
d(\phi_t (x_0),\phi_t (x_1))\leq C d(\phi_{T_0-\tau_1} (x_0), \phi_{T_0-\tau_1} (x_1))< C\epsilon_1
=
\frac{\epsilon'}{2}\, ,
\]
where the last inequality holds because of \eqref{eqn:separation}.
Therefore, for every $t\in[0,T_0]$,
\[
d(\phi_t (x_0),\phi_t (x_1))<\frac{\epsilon'}{2}\,.
\]
Thus, for every $t\in [0,T_0]$, because of this last inequality and because of \eqref{eqn:minoration},
\[
d\left(\phi_{t\frac{T_1}{T_0}} (x_1),\phi_t (x_0)\right) \leq d\left(\phi_{t\frac{T_1}{T_0}} (x_1),\phi_t (x_1)\right)+d\left(\phi_t (x_1),\phi_t (x_0)\right)< b t \left|\frac{T_1}{T_0}-1 \right|+\frac{\epsilon'}{2}\leq b\tau_0+\frac{\epsilon'}{2}\leq \epsilon'\,.
\]
As $\phi_t (x_0)$ and $\phi_{t\frac{T_1}{T_0}} (x_1)$ are $T_0$ periodic, the inequality
\[
d\left(\phi_{t\frac{T_1}{T_0}} (x_1),\phi_t (x_0)\right) <\epsilon'
\]
holds for every $t\in\mathbb R$. From the expansivity property, see Definition~\ref{def:expansive}, we get a parameter $u$ such that $|u|\leq \nu$ and $x_1=\phi_u (x_0)$.
\end{proof}

\begin{lemm}[Uniform transitivity]\label{transitivite prop bis}\label{transitivite prop}
Let $\phi$ be a $H$-flow on $M$. Let $K'\subset K\subset M$ be two compact subset  with nonempty interior. Let $\delta>0$. There exists $\sigma>0$ 
such that for every $x,y\in K$ and for every $S\ge \sigma$
there are $z\in M$ and $T\in[S-\tau_K,S+\tau_K]$ such that
\[
\phi_{[0,T]}(z)\cap K'\neq\emptyset,\quad d(x,z)<\delta \quad\text{and}\quad d(y,\phi_T(z))<\delta\, .
\]
\end{lemm}

\begin{figure}[ht]
    \centering
    \includegraphics{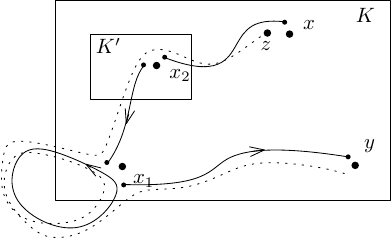}
    \caption{Proof of Uniform Transitivity}
    \label{fig:uniform_transitivity}
\end{figure}

\begin{proof}
See Figure~\ref{fig:uniform_transitivity}.
Recall that $\tau_K$ is the length of the shortest periodic orbit with period $\geq 1$ intersecting $K$.
Let $\eta$ be given by the finite exact shadowing property \ref{weak shadowing property} with parameters $K$, $\delta$ and $N=4$.
Cover $K$ with finitely many balls $B_i=B(x_i,\eta/2)$ of radius $\eta/2$. Without loss of generality, we may assume that the periodic orbit $\gamma$ used to define $\tau_K$ intersects $B_1$ and that $B(x_2,\delta+\eta/2)\subset K'$.

For all $i,j$, there exists $T_{i,j}> 0$ such that $\phi_{T_{i,j}}(B_i)\cap B_j\neq\emptyset$ (this is the transitivity property from Definition~\ref{def-transitive}).
Let $\sigma_0 = 3\max_{i,j} \{T_{i,j}\}+\tau_K$.

Let $x,y\in K$. In particular, there exists $i,j$ such that $x\in B_i$, $y\in B_j$. Let $S\geq \sigma_0$.
Fix $n\geq 2$ such that $T_{i,2}+T_{2,1}+n\tau_K+T_{1,j} \in [S-\tau_K,S+\tau_K]$.
Use the finite exact shadowing (see Definition \ref{weak shadowing property}) to concatenate an orbit from $B_i$ to $B_2$ (given by the definition of $T_{i,2}$), and orbit from $B_2$ to $B_1$ (given by the definition of $T_{2,1}$) $n\gamma$ and an orbit from $B_1$ to $B_j$  (given by the definition of $T_{1,j}$). The initial point of the concatenated orbit is the desired point $z$.
\end{proof}

The following lemma allows us to build a periodic orbit from a pseudo-orbit, see Figure~\ref{fig:multiple_cloing_lemma}.

\begin{lemm}[Multiple closing lemma]\label{petal unis} Let $\phi$ be a $H$-flow on $M$.
Let $K\subset M$ be a compact subset. For all $\delta>0$, $\nu>0$ and $N\in\mathbb{N}^*$, there exist $T_\mathrm{min}>0$ and $\eta>0$ such that for all $x_1,\dots, x_N\in K$, and $T_1,\dots, T_N>0$ with $\sum_{i=1}^NT_i\geq T_\mathrm{min}$ and $\phi_{T_i}(x_i)\in B(x_{i+1},\eta)$ for $i=1,\dots, N$ (where $x_{N+1}=x_1$), the following property holds. There exists a periodic orbit $\gamma$ with period
\[
\ell(\gamma)\in\left[\sum_{i=1}^NT_i-\nu,\sum_{i=1}^NT_i+\nu\right]
\]
such  that for every $0\le i\le N$, and every $s\in [0,T_i]$, we have
\[
d(\gamma(s+\sum_{j=1}^{i-1}T_j),\phi_s(x_i))<\delta \, .
\]
\end{lemm}

\begin{figure}[ht]
    \centering
    \includegraphics{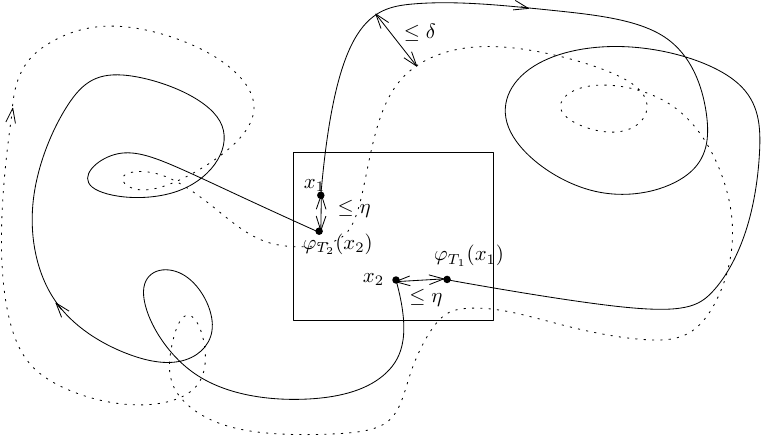}
    \caption{Multiple closing Lemma for $N=2$}
    \label{fig:multiple_cloing_lemma}
\end{figure}

\begin{proof} This follows immediately from the finite exact shadowing property \ref{weak shadowing property} and the closing lemma \ref{closing lemma}.
\end{proof}

The following lemma allows us to build a periodic orbit from a finite number of orbits with endpoints in a compact set, see Figure~\ref{fig:uniform_multiple_closing_lemma}.

\begin{lemm}[Uniform multiple closing lemma]\label{petal separe} Let $\phi$ be a $H$-flow on $M$.
Let $K'\subset K\subset M$ be two compact subsets with nonempty interior, let $\nu>0$, $\delta>0$ and $N\in\mathbb N^*$.
Then, there exist $\sigma>0$ and $T_\mathrm{min}>0$ such that for every $S\geq \sigma$, for all $x_1,\dots,x_N\in K$, all $T_1,\dots,T_N>0$ that satisfy  $\sum_{i=1}^N T_i\geq T_\mathrm{min}$, and such that for every $1\le i\le N$, $\phi_{T_i}(x_i)\in K$,  the following property holds. There exists a periodic orbit $\gamma$ with period
\[
\ell(\gamma)\in\left[\sum_{i=1}^NT_i+NS-\tau_K-\nu,\sum_{i=1}^NT_i+NS+\tau_K+\nu\right]
\]
that intersects $\inter{K'}$ and there exist $\tau_1,\dots,\tau_{N-1}$, $\tau_i\in[S-\tau_K,S+\tau_K]$ such that that for every $0\le i\le N$, and every $s\in [0,T_i]$, we have
\[
d(\gamma(s+\sum_{j=1}^{i-1}(T_j+\tau_j)),\phi_s(x_{i}))<\delta\, .
\]
\end{lemm}

\begin{figure}[ht]
    \centering
    \includegraphics{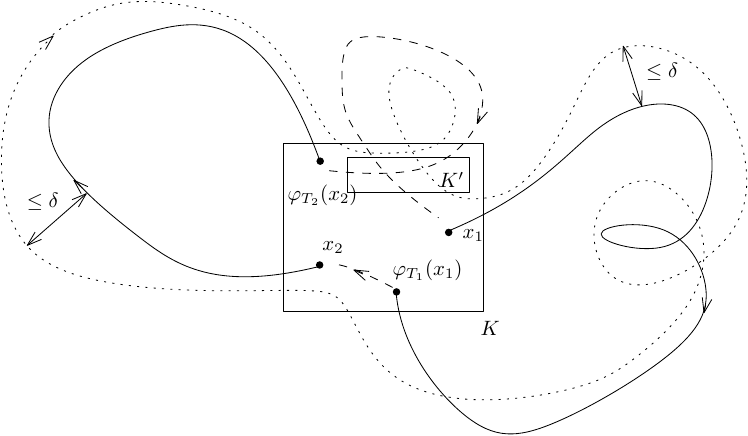}
    \caption{Uniform multiple closing Lemma for $N=2$ }
    \label{fig:uniform_multiple_closing_lemma}
\end{figure}

\begin{proof}
We will use the transitivity property~\ref{def-transitive} to construct pieces of orbits from $\phi_{T_i} (x_i)$ to $x_{i+1}$, then the multiple closing lemma~\ref{petal unis} to obtain a periodic orbit.
We will carefully choose the length of the last orbit to ensure $\ell(\gamma)\in\left[\sum_{i=1}^NT_i+NS-\tau_K-\nu,\sum_{i=1}^NT_i+NS+\tau_K+\nu\right]$
as opposed to
$\ell(\gamma)\in\left[\sum_{i=1}^NT_i+NS-N\tau_K-\nu,\sum_{i=1}^NT_i+NS+N\tau_K+\nu\right]$ if one ins not careful.

More precisely, assume $\delta$ is small enough so that there exist a ball $B(x',2\delta)\subset K'\subset K$.
Lemma~\ref{petal unis} with parameters $K$, $N$, $\delta$ and $\nu$ gives us $T_\mathrm{min}$ and $\eta$.
Let $\sigma_0$ be given by the uniform transitivity lemma (Lemma~\ref{transitivite prop bis}) with parameters $K$, $\delta$ and $B(x',\delta)$.
Let $\sigma = \sigma_0+(N-1)\tau_K$.

Let $x_1,\dots,x_N$ and $T_1,\dots T_N$ and $S$ be as in the statement of the Lemma.
By successive uses of the uniform transitivity~\ref{transitivite prop}, we can build pieces of orbits of respective lengths $\tau_i\in [S-\tau_K,S+\tau_K]$ from $B(\phi_{T_i}(x_i),\eta)$ to $B(x_{i+1},\eta)$ for $i=1,\dots N-1$.
The uniform transitivity~\ref{transitivite prop bis}, with constant $S_0=NS-\sum_{i=1}^N \tau_i \geq NS-(N-1)S-(N-1)\tau_K\geq\sigma_0$, gives us a piece of orbit from $B(\phi_{T_N} (x_N),\eta)$ to $B(x_1,\eta)$ intersecting $B(x',\delta)$ and of length $\tau_n\in[NS-\sum_{i=1}^N \tau_i-\tau_K,NS-\sum_{i=1}^N \tau_i+\tau_K]$.
Therefore, we have $\sum_{i=1}^N \tau_i\in [NS-\tau_K,NS+\tau_K]$.
Lemma~\ref{petal unis} gives us a periodic orbit which satisfies all the desired conditions.
\end{proof}

With the same arguments, we get the following variant of lemmas \ref{petal unis} and \ref{petal separe}. We let the proof to the reader.
\begin{lemm}[Variation around the multiple closing lemma]\label{petales reunis} Let $\phi$ be a $H$-flow on $M$.
Let $K'\subset K\subset M$ be two compact subsets with nonempty interior, let $\nu>0$, $\delta>0$ and $N\in\mathbb N^*$.
Then, there exist $\sigma>0$ and $T_\mathrm{min}>0$
and $\eta>0$ such that for every $S\geq \sigma$, for all $x_1,\dots,x_N\in K$, all $T_1,\dots,T_N>0$ that satisfy  $\sum_{i=1}^N T_i\geq T_\mathrm{min}$ , and such that for every $1\le i\le N$, $\phi_{T_i}(x_i)\in B(x_{i+1},\eta)$ (with $x_{N+1}=x_1$),  the following property holds. There exists a periodic orbit $\gamma$ with period
\[
\ell(\gamma)\in\left[\sum_{i=1}^NT_i+ S-\tau_K-\nu,\sum_{i=1}^NT_i+ S+\tau_K+\nu\right]
\]
that intersects $\inter{K'}$ and  such that for every $0\le i\le N-1$,  and every $s\in [0,T_i]$, we have
\[
d(\gamma(s+\sum_{j=1}^{i-1}T_j),\phi_s(x_i))<\delta \, .
\]
\end{lemm}

\section{Entropies}\label{sec:entropies}

This section is devoted to the introduction of different notions of entropy, both associated to a dynamics $(\phi_t)_{t}$ and to a $\phi$-invariant probability measure. Some subadditivity properties for periodic orbits are shown in the framework of $H$-flows, which imply some consequences for the Gurevich entropy. The definitions of entropies at infinity are also presented.

\begin{rema}\label{time-one} \rm All notions of entropies, denoted here $h$, concern a single transformation, and satisfy the relation $h(\phi^n)=|n|\, h(\phi)$. Classically, the entropy of a flow $\phi=(\phi_t)_{t\in\R}$ is defined as the entropy of its time-one map $\phi_1$. In this section, we write $\phi$ instead of $\phi_1$.
\end{rema}

\subsection{Kolmogorov-Sinai entropy}


Let $\mu$ be a $\phi$-invariant measure and $\xi$ be a finite or countable $\mu$-measurable partition. The entropy of the partition $\xi$ with respect to $\mu$ is the quantity
\[
H(\xi,\mu)= \sum_{\substack{C\in\xi\\ \mu(C)>0}}-\mu(C)\log(\mu(C))\geq 0\,.
\]
Given two measurable partitions $\xi_1,\xi_2$, we denote by $\xi_1\vee\xi_2$ the refinement of the two partitions, that is the partition whose elements are the sets $A\cap B$ with positive measure, for $A\in \xi_1$ and $B\in \xi_2$.
Given a  finite or countable $\mu$-measurable partition $\xi$, for every $n\ge 1$,  set $\xi_n:=\bigvee_{i=0}^{n-1}\phi_{-i}(\xi)$.
\begin{defi}\label{KS entropy}
The Kolmogorov-Sinai entropy of $\phi$ with respect to $\mu$ is
\[
h_{\mathrm{KS}}(\mu)=\sup_\xi \lim_{n\to+\infty}\dfrac{1}{n}H(\xi_n,\mu)\,,
\]
where the supremum is taken over all finite or countable $\mu$-measurable partitions $\xi$ with $H(\xi,\mu)<\infty$.
\end{defi}
Denote by $\mathcal{M}_\phi$ (resp.  $\mathcal{M}_\phi^\mathrm{erg}$) the  set of $\phi$-invariant (resp. $\phi$-invariant ergodic) probability measures. The set $\mathcal{M}_\phi$ is a convex set whose extremal points are exactly the ergodic invariant measures $\mu\in \mathcal{M}_\phi^\mathrm{erg}$. Moreover,  the entropy map $\mu\in \mathcal{M}_\phi\mapsto h_\mathrm{KS}(\mu)\in \mathbb{R}^+$ is affine \cite[Theorem 8.1]{Walters}. This justifies the following theorem-definition.
\label{p:M_phi_M_erg_phi}
\begin{defi}\label{variational entropy}
The variational entropy of the flow $\phi$ is the supremum
\[
h_{\mathrm{var}}(\phi)= \sup_{\mu\in\mathcal{M}_\phi}h_{\mathrm{KS}}(\mu)=\sup_{\mu\in\mathcal{M}^\mathrm{erg}_\phi}h_{\mathrm{KS}}(\mu)\,.
\]
A measure realizing this supremum, when it exists, is called a {\em measure of maximal entropy} of $\phi$.
\end{defi}


\subsection{Katok and Brin-Katok entropies}\label{sec:Katok-BK-entropy}

For $x\in M$, $\epsilon,T>0$, let $B(x,\epsilon,T)$ be the associated \emph{dynamical ball}, i.e.
\[
B(x,\epsilon,T)
=
\{y\in M :\ \forall t\in [0,T],\; d(\phi_t(x),\phi_t(y))<\epsilon \}.
\]
\label{p:dynamical_ball}

\begin{defi}\label{defi spanning} Let $\mu \in\mathcal{M}_\phi$,  $\alpha\in(0,1)$ and $T,\epsilon >0$.
A finite set $E$ is said to be \emph{$(T,\epsilon,\alpha,\mu)$-spanning} if
\[
\mu\left(\cup_{x\in E} B(x,\epsilon,T)\right) \geq \alpha.
\]
\end{defi}

\begin{fact}Let $M$ be a manifold and $\mu$ a probability measure on $M$.
For all $T>0$, $\epsilon>0$, $\alpha\in(0,1)$  there exists a finite $(T,\epsilon,\alpha,\mu)$-spanning set.
\end{fact}

\begin{proof}
As $M$ is exhaustible by compact subsets, there exists a compact subset $K\subset M$ such that $\mu(K)\geq \alpha$. Then $\cup_{x\in K} B(x,\epsilon,T)$ is an open cover for $K$. Any finite subcover provides us with a finite $(T,\epsilon,\alpha,\mu)$-spanning set.
\end{proof}

The following fact is immediate from the definition.
\begin{fact}\label{fact:M_Katok_entropy}
Let $M(T,\epsilon,\alpha,\mu)$ be the minimal cardinality of a $(T,\epsilon,\alpha,\mu)$-spanning set.
 It is a non-decreasing quantity in $\alpha>0$ and non-increasing in $\epsilon>0$.
\end{fact}
\begin{defi}[Katok entropy \cite{Katok}]\label{Katok entropy}The \emph{Katok entropy} of
$\mu$ is defined as
\[
h_{\mathrm{Kat}}(\mu) = \inf_{\alpha>0}\, \sup_{\epsilon>0}\, \limsup_{T\to\infty}\, \frac{1}{T}\log\,M(T,\epsilon,\alpha,\mu)\,=\,
\lim_{\alpha\to 0}\, \lim_{\epsilon\to 0}\, \limsup_{T\to\infty}\, \frac{1}{T}\log\,M(T,\epsilon,\alpha,\mu)\,.
\]
\end{defi}

\begin{defi}[Brin-Katok local entropy \cite{BK}]\label{Brin Katok entropies} Let $K$ be a compact subset of $M$. The \emph{upper, resp. lower, local entropy} on $K$ is defined as
\[
\overline{h}_{\mathrm{loc}}(\mu,K)
=
\sup_{\epsilon>0} \,\esssup_{x\in K} \,\limsup_{\substack{T\to\infty \\ \phi_T(x)\in K}}\,-\frac{1}{T}\log\mu\left(B(x,T,\epsilon)\right)
=
\lim_{\epsilon\to 0} \,\esssup_{x\in K} \,\limsup_{\substack{T\to\infty \\ \phi_T(x)\in K}}\,-\frac{1}{T}\log\mu\left(B(x,T,\epsilon)\right)\, ,
\]
resp.
\[
\underline{h}_{\mathrm{loc}}(\mu,K)
=
\sup_{\epsilon>0} \,\essinf_{x\in K} \,\liminf_{\substack{T\to\infty \\ \phi_T(x)\in K}}\,-\frac{1}{T}\log\mu\left(B(x,T,\epsilon)\right)
=
\lim_{\epsilon\to 0} \,\essinf_{x\in K} \,\liminf_{\substack{T\to\infty \\ \phi_T(x)\in K}}\,-\frac{1}{T}\log\mu\left(B(x,T,\epsilon)\right)\,.
\]
The \emph{upper (resp. lower) Brin-Katok} entropy of $\mu$ is defined as
\[
\overline{h}_\mathrm{BK}(\mu) =
\sup_{K\subset M \text{compact}}\, \overline{h}_{\mathrm{loc}}(\mu,K)\,,
\]
resp.
\[
\underline{h}_\mathrm{BK}(\mu) =
\inf_{K\subset M \text{compact}}\, \underline{h}_{\mathrm{loc}}(\mu,K)\,.
\]
\label{p:BK_entropy}
\end{defi}

\begin{rema}
Poincaré recurrence theorem implies that $\mu$-almost every point in $K$ returns to $K$ infinitely often and therefore $$\limsup_{\substack{T\to\infty \\ \phi_T(x)\in K}}\,-\frac{1}{T}\log\mu\left(B(x,T,\epsilon)\right) $$ is well defined $\mu$-almost everywhere.
\end{rema}


\subsection{Gurevic entropy}

\subsubsection{Definition}\label{subsection def orbit periodique}

A periodic point of $\phi$ is a couple $(x,T)$, with $x\in M$ and $T>0$ such that $\phi_T(x)=x$.
A periodic orbit $\gamma$ is a couple $\left(\{\phi_{t}(x), t\in\R\},T\right)$, where $(x,T)$ is a periodic point. The period $T>0$ is denoted by $\ell(\gamma)$. Note that, with this definition, the orbits of $(x,T)$ and $(x,nT)$, for $n\ge 2$, are distinct.
A parametrized periodic orbit is a periodic map from $\R$ to $ M$, still denoted by $\gamma$, of the form
\[
s\in\R \mapsto \gamma(s)=\phi_s(\phi_{t_0}(x))\,,\quad\mbox{for a fixed real number } \, t_0\in\R.
\]
By abuse of notation, we speak of a point $z\in\gamma$ instead of a point in the image of the orbit associated with $\gamma$.
We denote by $\mathcal{P}$ the set of  periodic orbits of  the flow $\phi$.

Let $K\subset M$ be a compact set. Given any $x\in \gamma$, denote by $\ell(\gamma\cap K)=|\{t\in [0,\ell(\gamma)],\, \phi_t(x)\in K\}|$ the length of the intersection of the periodic orbit $\gamma$ with the set $K$. This quantity does not depend on the point $x$ used in its definition.

Given any $0<T_0\leq T_1$, we define\label{p:P_K_T}
\[
\mathcal{P}_K(T_0, T_1)=\{ \gamma\in  \mathcal{P},\,  \ell(\gamma)\in [T_0,T_1], \gamma\cap K\neq \emptyset\}\,.
\]
Define also
\[
\mathcal{P}_K(T_0)=\{ \gamma\in\mathcal{P},\,  \ell(\gamma)\leq T_0, \gamma\cap K\neq \emptyset\}\,.
\]

\begin{theo}\label{theo:Gurevic}
Let $\phi$ be a $H$-flow on $(M,d)$. Let $K\subset M$ be a compact subset with nonempty interior and $\sigma>0$.
The  quantity
\begin{equation}\label{def:h_Gur}
h_\mathrm{Gur}(\phi,K,\sigma)=\limsup_{L\to\infty}\frac{1}{L}\log \#   \mathcal{P}_K(L,L+\sigma)
\end{equation}
does depend neither on $\sigma>0$ nor on $K$. It is called the {\em Gurevic entropy of the flow $\phi$} and denoted by {\em  $h_\mathrm{Gur}(\phi)$}.

Moreover, for every compact subset $K\subset M$ with nonempty interior,  and every $\sigma\ge 5\tau_K$, with $\tau_K$ being the period of the shortest periodic orbit with period at least $1$ that intersects $K$, the Gurevic entropy is a true limit\,:
\begin{equation}\label{corollaire_entropie_gurevic_est_une_limite}
h_\mathrm{Gur}(\phi)=\lim_{L\to +\infty}\frac{1}{L}\log \#\mathcal{P}_K(L,L+\sigma)\,.
\end{equation}
\end{theo}

The first part of the Theorem is relatively elementary and classical, and follows from Facts~\ref{ind gur entropy} and \ref{dependance-compact}. The second part is more difficult and follows from subadditivity properties proved in Section~\ref{sec:subadd}.

Before the proof of the theorem, let us give an immediate corollary, that will be useful in Section~\ref{section:estimation_boules_dynamiques}.

\begin{coro}\label{corollaire_limite_nulle_quotient_P_K}
Under the assumptions of Theorem~\ref{theo:Gurevic}, if  $h_{\mathrm{Gur}}(\phi)>0$, then
\[
\lim_{L\to\infty}\frac{\#\mathcal P_{K}\left(\frac{L+\sigma}{2}\right)}{\#\mathcal P_{K}(L,L+\sigma)}
=
0\,.
\]
\end{coro}

\subsubsection{The Gurevic entropy does depend neither on \texorpdfstring{$K$}{TEXT} nor on \texorpdfstring{$\sigma$}{TEXT}}

We prove here the first part of Theorem \ref{theo:Gurevic}.

\begin{fact}\label{ind gur entropy}\label{lemma growth leq}
Under the assumptions of Theorem \ref{theo:Gurevic}, the Gurevic entropy  satisfies
\[
h_\mathrm{Gur}(\phi,K,\sigma)
=
\limsup_{L\to +\infty}\frac{1}{L}\log \#\mathcal{P}_K(L)\,.
\]
In particular, $h_\mathrm{Gur}(\phi,K,\sigma)=h_\mathrm{Gur}(\phi,K)$ does not depend on the constant $\sigma$.
\end{fact}

\begin{proof} Fix some $\sigma>0$.
As $\mathcal{P}_K(L,L+\sigma)\subset \mathcal{P}_K(L+\sigma)$, the inequality \[
h_\mathrm{Gur}(\phi,K,\sigma)\le\limsup_{L\to +\infty}\frac{1}{L}\log \#\mathcal{P}_K(L)
\]
is immediate.

The proof of $h_\mathrm{Gur}(\phi,K,\sigma)\ge\limsup_{L\to +\infty}\frac{1}{L}\log \#\mathcal{P}_K(L) $ is classical.
Let $\delta>0$ and denote as $h$ the entropy $h_{\mathrm{Gur}}(\phi, K, \sigma)$.
If $h=+\infty$, there is noting to prove.
We now assume $h<\infty$.
There exists $T_0>0$ such that for $T>T_0$,
$\#\mathcal P_{K}(T,T+\sigma)\leq e^{T (h+\delta)}$.
Let $N_T = \left\lfloor \frac{T-T_0}{\sigma} \right\rfloor-1$.
We have
\[
\mathcal{P}_K(T)\,\,\subset\,\,
\mathcal{P}_K(T_0)\,\cup\, \bigcup_{n=0}^{N_T} \mathcal{P}_K(T_0+n\sigma,T_0+(n+1)\sigma)\,.
\]
Therefore
\begin{align*}
\# P_K(T) &\leq P_K(T_0) + \sum_{n = 0}^{N_T}\#\mathcal P_{K}(T_0+n\sigma,T_0+(n+1)\sigma)\\
&\leq \mathcal{P}_K(T_0) + \sum_{n =0}^{N_T}  e^{(T_0+n\sigma)(h+\delta)} \\
&\leq\mathcal{P}_K(T_0) + \frac{T}{\sigma}  e^{T(h+\delta)}.
\end{align*}
It follows that
\[
\limsup_{T\to\infty} \frac{1}{T}\log\#\mathcal \mathcal{P}_{K}(T)\leq h_\mathrm{Gur}(\phi,K,C)+\delta\,.
\]
As the inequality holds for any $\delta >0$, the result follows.
 \end{proof}

\begin{fact}\label{dependance-compact} If $\phi$ is a $H$-flow, and $K\subset M$ is a compact set with nonempty interior, then $h_{\mathrm{Gur}}(\phi,K)=h_{\mathrm{Gur}}(\phi)$. In particular, the Gurevic entropy $h_\mathrm{Gur}(\phi)$  does not depend on $K$.
\end{fact}

\begin{proof} If $K\subset K'$ are any two compact sets, we have   $h_\mathrm{Gur}(\phi,K)\le h_\mathrm{Gur}(\phi,K')$.

\begin{figure}[ht]
    \centering
    \includegraphics{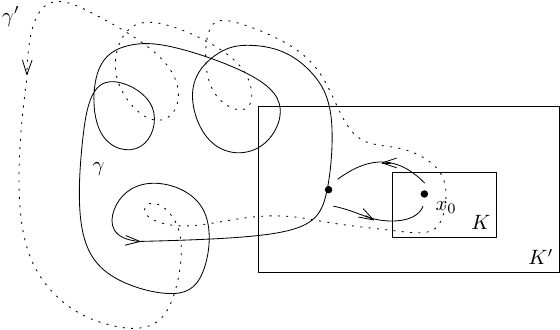}
    \caption{The Gurevic entropy does not depend on the compact}
    \label{fig:h_Gur_independant_K}
\end{figure}

The reverse inequality can be obtained, as explained below, from the transitivity, finite exact shadowing, and closing lemma properties as gathered in Lemma~\ref{petal unis}.

Let $\sigma'>0$.
Choose $\epsilon>0$ and $x_0\in K$ such that the open ball $B_0=B(x_0,\epsilon)$  is included in $K$.
By selecting a smaller $\epsilon$ if necessary, we may also assume that $2\epsilon$ is adapted to the separation of orbits (Lemma~\ref{lemma on same po}) for parameters $\nu=\tau_1=1$. This lemma also gives us $\tau_0$.

Let $0<\eta<\epsilon/2$ and $T_\mathrm{min}$ be given by Lemma \ref{petal unis} with parameters $K'$, $N=3$, $\delta = \epsilon/2$ and $\nu=1$.
Cover $K'$ with finitely many open balls $B_i=B(x_i,\eta)$. Without loss of generality, assume that $x_0\in B_0$. By transitivity (\ref{transitivite prop bis}), there exist pieces of orbits $c_i$ with length $\ell_i\geq 1$ from $B_i$ to $B_0$ and $c'_i$ with length $\ell'_i\geq 1$ from $B_0$ to $B_i$. Set $L_0=\max\{\ell_i,\ell'_i\}$. We may assume that $\ell_i
$ and $\ell'_i$ are bounded below by $T_\mathrm{min}$.

Given a periodic orbit $\gamma$ in $\mathcal{P}_{K'}(L,L+\sigma')$, we may assume, without loss of generality, that $\gamma(0)\in K$. By Lemma~\ref{petal unis}, one can concatenate the pieces of orbits $\gamma$ (starting at $\gamma(0)$), $c_i$ and $c'_i$ to get an $\epsilon/2$-close periodic orbit $\gamma'$ that goes through $B(x_0,\epsilon)\subset K$, and has length in
$[\ell(\gamma)+\ell_i+\ell'_i-1,\ell(\gamma)+\ell_i+\ell'_i+1]\subset[L,L+\sigma'+2L_0+1]$. 

This construction gives us a map from $\mathcal{P}_{K'}(L,L+\sigma')$ to $\mathcal{P}_K(L,L+\sigma)$, for $\sigma=\sigma'+2L_0+1$ depending on $K$ and $K'$. Let us control the cardinal of its preimages.
Let $\gamma_0$ and $\gamma_1$ be periodic orbits in $\mathcal{P}_{K'}(L,L+C')$ of period $\ell(\gamma_0)$ and $\ell(\gamma_1)$, lying in the preimage of $\gamma\in \mathcal{P}_K(L,L+\sigma)$.
One may choose the orgins of $\gamma_0$, $\gamma_1$ and $\gamma$ such that there exist $t_1$ satisfying the following
\begin{itemize}
    \item for all $0\leq t\leq \ell(\gamma_0)$, we have $d(\gamma_0(t), \gamma(t))\leq \epsilon/2$
    \item for all $0\leq t\leq \ell(\gamma_1)$, we have $d(\gamma_1(t), \gamma(t+t_1))\leq \epsilon/2$.
\end{itemize}
Therefore, for all $0\leq t \leq \min(\ell(\gamma_0),\ell(\gamma_1))$, by the previous inequalities and by \eqref{eqn:minoration},
\[d(\gamma_0(t),\gamma_1(t))\leq \epsilon +  b|t_1|.\]
Thus if $|t_1|\leq\epsilon/b$ and $|\ell(\gamma_1)-\ell(\gamma_0)|\leq\tau_0$, Lemma~\ref{lemma on same po} proves that $\gamma_0=\gamma_1$.
As $|t_1|\leq L+\sigma$ and $|\ell(\gamma_1)-\ell(\gamma_0)|\leq \sigma'$, we get the inequality
\[
\#\mathcal{P}_{K'}(L,L+\sigma')\le \left\lceil\frac{(L+\sigma)b}{\epsilon}\right\rceil\, \left\lceil\frac{\sigma'}{\tau_0}\right\rceil\, \#\mathcal{P}_K(L,L+\sigma)\,.
\]
The equality $h_{\mathrm{Gur}}(\phi,K)=h_{\mathrm{Gur}}(\phi,K')$ follows immediately.
This proves $h_{\mathrm{Gur}}(\phi,K)$ does not depend on $K$.
\end{proof}

  \subsubsection{First subadditivity  properties}\label{sec:subadd}

We now prove a subadditivity property, that is a comparison between counts of orbits of periods $L_1$, $L_2$ and $L_1+L_2$.
This will be a key ingredient to prove the last part of Theorem~\ref{theo:Gurevic} which says that the Gurevic entropy is a true limit.

Recall that $\tau_K$ is the period of the shortest periodic orbit with period $\ge 1$ that intersects the interior of $K$.

\begin{prop}\label{prop subadd gen}
Let $\phi$ be a $H$-flow on $M$. Let $K\subset M$ be a compact subset with nonempty interior.
Let $\sigma_0\geq 4\tau_K$.
There exist constants  $D$ and $\sigma$ such that for all $L_1,L_2\gg 1$,
\begin{equation}\label{subadditivity general}
\# \mathcal{P}_K(L_1,L_1+\sigma_0)\, \#\mathcal{P}_K(L_2,L_2+\sigma_0)\leq D\, (L_1+L_2)\, \#\mathcal{P}_K(L_1+L_2+\sigma, L_1+L_2+\sigma+\sigma_0)\, .
\end{equation}
\end{prop}

\begin{figure}[ht]
    \centering
    \includegraphics{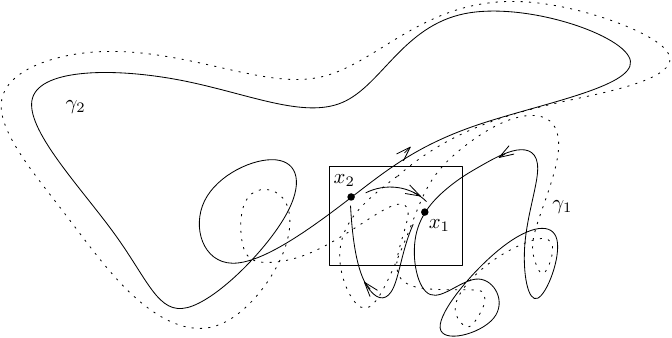}
    \caption{First subadditivity property}
    \label{fig:sous_additivite_simple}
\end{figure}

\begin{proof}
See Figure~\ref{fig:sous_additivite_simple}.
Fix the compact set $K$. We will build a map
\[
f\colon \mathcal{P}_K(L_1,L_1+\sigma_0)\times \mathcal{P}_K(L_2,L_2+\sigma_0)\to \mathcal{P}_K(L_1+L_2+\sigma,L_1+L_2+\sigma+\sigma_0)\, ,
\]
whose preimages, for $L_1,L_2$ large enough and a suitable choice of the constant $\sigma$,  have a cardinality bounded by $D\times (L_1+L_2)$, for a suitable constant $D>0$. Inequality \ref{subadditivity general} will follow immediately.

{\bf{Step 1.}} Construction of $f$.

Lemma~\ref{lemma on same po} associates with $\nu=1$ some $\tau_0>0$, and for $\tau_1=2\tau_0+2\tau_K$ some $\epsilon>0$.
Lemma~\ref{petal separe} applied to the compact set $K$ associates to $\nu = 1$,  $\delta=\frac{\epsilon}{4}$ and $N=2$ some numbers
$T_{\mathrm{min}}>0$ and $\sigma'>0$.

Consider $L_1,L_2\geq T_{\min}$, and a pair of periodic orbits $\gamma_1\in\mathcal{P}_K(L_1,L_1+\sigma_0)$ and $\gamma_2\in\mathcal{P}_K(L_2,L_2+\sigma_0)$, with respective lengths $\ell(\gamma_1)$ and $\ell(\gamma_2)$. The periodic orbit $f(\gamma_1,\gamma_2)$ is defined as follows. As $\gamma_1$ and $\gamma_2$ intersect $K$, we can reparametrize them so that their origins $\gamma_1(0), \gamma_2(0)$ belong to $K$.
Apply Lemma \ref{petal separe} with $x_1=\gamma_1(0),x_2=\gamma_2(0), T_1=\ell(\gamma_1),T_2=\ell(\gamma_2)$ and
\[
S=S(\gamma_1,\gamma_2)=\dfrac{L_1+L_2+2\sigma_0-\ell(\gamma_1)-\ell(\gamma_2)}{2}+\sigma'+\tau_K\geq \sigma'\,.
\]
We get a periodic orbit $\gamma=f(\gamma_1,\gamma_2)$ that intersects $K$, with length
\[
\ell(\gamma)\in [\ell(\gamma_1)+\ell(\gamma_2)+2S-\tau_K-1,\ell(\gamma_1)+\ell(\gamma_2)+2S+\tau_K+1]\subset [L_1+L_2+\sigma,L_1+L_2+\sigma+\sigma_0]\, ,
\]
where $\sigma=2\sigma_0+2\sigma'+\tau_K-1$.
Moreover, there exists $\tau\in[S-\tau_K,S+\tau_K]$ such that
\begin{itemize}
    \item for every $s\in [0,\ell(\gamma_1)]$, we have $d(\gamma_1(s),\gamma(s))<\dfrac{\epsilon}{4}$;
    \item for every $s\in [0,\ell(\gamma_2)]$, we have $d(\gamma_2(s),\gamma(\ell_1+\tau+s))<\dfrac{\epsilon}{4}$.
\end{itemize}
The periodic orbit  $f(\gamma_1,\gamma_2)=\gamma$ belongs therefore to $\mathcal{P}_K(L_1+L_2+\sigma,L_1+L_2+\sigma+\sigma_0)$. Moreover, our construction provides an origin of the orbit $\gamma$, i.e. a marked point on its image.

{\bf{Step 2.}} Bound on the cardinality of each preimage.

Assume that $\gamma_1,\tilde\gamma_1\in\mathcal{P}_K(L_1,L_1+\sigma_0)$ and $\gamma_2,\tilde\gamma_2\in\mathcal{P}_K(L_2,L_2+\sigma_0)$ are such that $f(\gamma_1,\gamma_2)=f(\tilde\gamma_1,\tilde\gamma_2)=\gamma$.
The constructions of $f(\gamma_1,\gamma_2)$ and $f(\tilde \gamma_1,\tilde \gamma_2)$ lead to the same orbit $\gamma$ by assumption, but with maybe different origins.
Without loss of generality, we can shift the parametrization of $\gamma$ so that the origin given by the construction of $\gamma=f(\gamma_1,\gamma_2)$ starting from $\gamma_1$ and $\gamma_2$  is $\gamma(0)$.
Let $\gamma(s_0)$ be the origin of $\gamma$ given by the construction of $\gamma=f(\tilde \gamma_1,\tilde \gamma_2)$ starting from $\tilde\gamma_1$ and $\tilde\gamma_2$.

{\bf{Step 2.a}} Bound on the cardinality of each preimage, when the lengths are prescribed.

We prove the following  statement.
If the orbits satisfy $f(\gamma_1,\gamma_2)=f(\tilde \gamma_1,\tilde\gamma_2)$ and
\begin{equation}\label{assumptions subadd}
\vert s_0\vert <\dfrac{\epsilon}{4b}\, ,\qquad \vert\ell(\gamma_1)-\ell(\tilde\gamma_1)\vert<\tau_0\, ,\qquad \vert\ell(\gamma_2)-\ell(\gamma_2)\vert<\tau_0\, ,
\end{equation}
where $b$ is the constant of property  \eqref{eqn:minoration} in  definition \ref{def:H-flow},  then the orbits coincide:  $(\gamma_1,\gamma_2)=(\tilde\gamma_1,\tilde\gamma_2)$.

By the definition of the function $f$, the following holds:
\begin{itemize}
    \item for every $s\in[0,\ell_1]$, we have $d(\gamma(s),\gamma_1(s))<\dfrac{\epsilon}{4}$;
    \item for every $s\in[0,\tilde\ell_1]$, we have $d(\gamma(s_0+s),\tilde\gamma_1(s))<\dfrac{\epsilon}{4}$.
\end{itemize}
Thus, for every $s\in[0,\min(\ell(\gamma_1),\ell(\tilde\gamma_1)]$ we obtain
\[
d(\gamma_1(s),\tilde\gamma_1(s))\leq d(\gamma_1(s),\gamma(s))+d(\gamma(s),\gamma(s_0+s))+d(\gamma(s_0+s),\tilde\gamma_1(s))<\dfrac{\epsilon}{4}+\vert s_0\vert b+\dfrac{\epsilon}{4}<\epsilon\, .
\]
By Lemma~\ref{lemma on same po}, we deduce that $\gamma_1=\tilde\gamma_1$.

Again by the construction of $f$, there exist $\tau\in[S(\gamma_1,\gamma_2)-\tau_K,S(\gamma_1,\gamma_2)+\tau_K]$ and $\tilde\tau\in[S(\tilde\gamma_1,\tilde\gamma_2)-\tau_K,S(\tilde\gamma_1,\tilde\gamma_2)+\tau_K]$ such that
\begin{itemize}
    \item for every $s\in[0,\ell(\gamma_2)]$, we have $d(\gamma(\ell(\gamma_1)+\tau+s),\gamma_2(s))<\dfrac{\epsilon}{4}$;
    \item for every $s\in[0,\ell(\tilde\gamma_2)]$, we have $d(\gamma(s_0+\ell(\tilde\gamma_1)+\tilde\tau+s),\tilde\gamma_2(s))<\dfrac{\epsilon}{4}$
\end{itemize}
Let $\bar s=\ell(\tilde\gamma_1)-\ell(\gamma_1)+\tilde\tau-\tau$.
Up to swapping $\gamma_1$ and $\tilde\gamma_1$, we may assume $\bar s \geq 0$.
Moreover, we have $\vert \bar s\vert \leq 2\sigma_0+2\tau_K$.
Then, for every $s\in[0,\min(\ell(\gamma_2)-\bar s,\ell(\tilde\gamma_2))]$, we have
\begin{align*}
  d(\gamma_2(\bar s+s),\tilde\gamma_2(s))\leq &  d(\gamma_2(\bar s+s),\gamma(\ell(\gamma_1)+\tau+\bar s+s)) +d(\gamma(\ell(\gamma_1)+\tau + \bar s+s),\gamma(s_0+\ell(\tilde\gamma_1) + \tilde\tau+s))\\
  &+d(\gamma(s_0+\ell(\tilde\gamma_1)+\tilde\tau+s),\tilde\gamma_2(s)) \\
  <&\dfrac{\epsilon}{4}+\vert s_0\vert b+\dfrac{\epsilon}{4}<\epsilon\, .
\end{align*}
By Lemma~\ref{lemma on same po}, since $\vert \bar s\vert \leq 2\sigma_0+2\tau_K= \tau_1$, we conclude that also $\gamma_2=\tilde\gamma_2$.

{\bf{Step 2.b.}} Conclusion.

So far, we have shown that, as soon as $(\gamma_1,\gamma_2)$ and $(\tilde\gamma_1,\tilde\gamma_2)$ satisfy (\ref{assumptions subadd}), if they have the same image under $f$, then they are the same periodic orbits. Since
\[
\vert s_0\vert
\leq
L_1+L_2+\sigma+\sigma_0\, ,\qquad \vert \ell(\gamma_1)-\ell(\tilde\gamma_1)\vert \leq \sigma_0\, ,\qquad \vert \ell(\gamma_2)-\ell(\gamma_2)\vert \leq \sigma_0\, ,
\]
we deduce that the cardinality of any preimage through $f$ of a periodic orbit in $\mathcal{P}_K(L_1+L_2+\sigma,L_1+L_2+\sigma+\sigma_0))$ is bounded by
\[
\left\lceil\dfrac{\sigma_0}{\tau_0}\right\rceil^2\, \left\lceil\dfrac{4b}{\epsilon}(L_1+L_2+\sigma+\sigma_0)\right\rceil\, .
\]
Choose the constant $D$ so that
\[
\left\lceil\dfrac{\sigma_0}{\tau_0}\right\rceil^2\, \left\lceil\dfrac{4b}{\epsilon}(L_1+L_2+\sigma+\sigma_0)\right\rceil
\leq D(L_1+L_2)\, .
\]
The desired bound \ref{subadditivity general} follows immediately.
\end{proof}

We will use the subadditivity property shown in Proposition \ref{prop subadd gen} to deduce that the exponential growth rate of the cardinality of the set of periodic orbits intersecting $K$ has a limit. We need the following result, see \cite[Theorem 23]{Erdos}.
\begin{theo}[de Bruijn--Erdös]\label{erdos}
 Let $t>0\mapsto \psi(t)$ be a positive and increasing map. Assume that $\int_1^\infty \frac{\psi(t)}{t^2}\, dt<\infty$. Let $(u_n)_{n\in\mathbb{N}}$ be a sequence such that
 \[
 u_{n+m}\leq u_n+u_m+\psi(n+m)\qquad \text{for }\dfrac n 2 \leq m\leq 2n\, .
 \]
 Then $\lim_{n\to\infty}\frac{u_n}{n}=L$ for some $L\in\mathbb{R}\cup\{-\infty\}$.
\end{theo}
We can now conclude the proof of Theorem \ref{theo:Gurevic}, by showing that \eqref{corollaire_entropie_gurevic_est_une_limite} holds.

\begin{proof}[Proof of \eqref{corollaire_entropie_gurevic_est_une_limite} in Theorem~\ref{theo:Gurevic}]
Fix a compact set $K$ with nonempty interior.
Let $\hat\sigma$ be the constant $\sigma$ given by Proposition~\ref{prop subadd gen} for $\sigma_0=4\tau_K$ (with respect to the notation of Proposition~\ref{prop subadd gen}).

{\bf{Step 1.}} The sequence $\frac{1}{n}\log\#\mathcal{P}_K(n-\hat\sigma,n-\hat\sigma+4\tau_K)$ converges to $h_\mathrm{Gur}(\varphi)$.

{\bf{Step 1.a}} The sequence $\frac{1}{n}\log\#\mathcal{P}_K(n-\hat\sigma,n-\hat\sigma+4\tau_K)$ has a limit.

Define a sequence $(u_n)_{n\in\mathbb{N}}$ as
\[
u_n=-\log\,\#\mathcal{P}_K(n-\hat\sigma,n-\hat\sigma+4\tau_K)
\]
for $n\in\mathbb{N}, n\gg 1$. By Proposition~\ref{prop subadd gen}, for $n,m$ large enough, the sequence $(u_n)_{n\in\mathbb{N}}$ satisfies
\[
u_{n+m}\leq \log(D(n+m-2\hat\sigma))+u_n+u_m\leq \log(D(n+m)+1)+u_n+u_m\, .
\]
Observe that the function $\psi:t>0\mapsto \log(Dt+1)$ is positive and increasing for $t>0$ and satisfies $\int_1^\infty \frac{\log(Dt+1)}{t^2}\, dt<\infty$. Therefore, by Theorem \ref{erdos},  the sequence $(\frac{u_n}{n})_{n\in\mathbb{N}}$ converges  and   $\lim_{n\to\infty}\frac{-u_n}{n}\leq h_{\mathrm{Gur}}(\phi)$.

{\bf{Step 1.b.}} The limit is $h_\mathrm{Gur}(\varphi)$.

We now prove $\lim_{n\to\infty}\frac{-u_n}{n} =  h_{\mathrm{Gur}}(\phi)$. It is enough to prove $\lim_{n\to\infty}\frac{-u_n}{n} \geq  h_{\mathrm{Gur}}(\phi)$.
Let $(L_n)_{n\in\mathbb N}$ be such that
\[
\lim_{n\to+\infty}\dfrac{\log \#\mathcal{P}_K(L_n,L_n+3\tau_K)}{L_n}=h_{\mathrm{Gur}}(\phi)\, .
\]
Let $(N_n)_{n\in\mathbb N}$ be the sequence of integers such that
\[N_n-\hat\sigma\leq L_n <N_n-\hat\sigma +1.\]
Then (as $\tau_K\geq 1$)
\[\mathcal P_K(L_n,L_n+3\tau_K)\subset \mathcal P_K(N_n-\hat\sigma,N_n-\hat\sigma+4\tau_K)\, . \]
Therefore
\[h_\mathrm{Gur}(\phi)=\lim_{n\to\infty} \dfrac{\log \#\mathcal{P}_K(L_n , L_n +3\tau_K)}{ L_n} \leq \lim_{n\to\infty} \dfrac{\log \#\mathcal{P}_K(N_n-\hat{\sigma},N_n-\hat{\sigma}+4\tau_K)}{N_n-\hat\sigma}=\lim_{n\to+\infty}\dfrac{-u_{N_n}}{N_n}\, .\]
Thus $\lim_{n\to\infty}\frac{-u_n}{n} =  h_{\mathrm{Gur}}(\phi)$.

{\bf{Step 2.}} The sequence $\frac{1}{L_n}\log\#\mathcal{P}_K(L_n,L_n+ 5\tau_K)$ converges to $h_{\mathrm{Gur}}(\phi)$.

Let $(L_n)_n$ be a sequence such that $L_n\to\infty$.
We now prove that
$\left(\dfrac{\log \#\mathcal{P}_K(L_n , L_n +5\tau_K)}{ L_n}\right)_{n\in\mathbb N}$
can only have $h_\mathrm{Gur}(\phi)$ as subsequential limit and therefore converges to $h_\mathrm{Gur}(\phi)$. Recall that $\hat\sigma>0$ is the constant given by Proposition~\ref{prop subadd gen} for $\sigma_0=4\tau_K$. Let $(N_n)_{n\in\mathbb N}$ be the sequence of integers such that
\[
L_n\leq N_n-\hat\sigma< L_n +1.
\]
From the previous step, we know that
$
\left(\dfrac{\log \#\mathcal{P}_K(N_n-\hat\sigma ,N_n-\hat\sigma +4\tau_K)}{N_n}\right)_{n\in\mathbb N}
$
converges to $h_{\mathrm{Gur}}(\phi)$.
Moreover, as $\tau_K\geq 1$,
\[
\mathcal{P}_K(N_n-\hat\sigma ,N_n-\hat\sigma +4\tau_K)\subset \mathcal{P}_K(L_n , L_n +5\tau_K)\]
and we obtain
\[
h_{\mathrm{Gur}}(\phi) =\lim_{n\to\infty}\dfrac{\log \# \mathcal{P}_K(N_n-\hat\sigma ,N_n-\hat\sigma +4\tau_K)}{N_n}
\leq
\limsup_{n\to\infty}\dfrac{\log \# \mathcal{P}_K(L_n ,L_n +5\tau_K)}{ L_n}\leq
h_{\mathrm{Gur}}(\phi)\,.
\]
Thus
the sequence
$
\left(\dfrac{\log \#\mathcal{P}_K(L_n , L_n +5\tau_K)}{ L_n}\right)_{n\in\mathbb N}
$
converges to $h_{\mathrm{Gur}}(\phi)$
and
\[
\lim_{L\to+\infty}\dfrac{\log \#\mathcal{P}_K(L,L+5\tau_K)}{L}
=
h_{\mathrm{Gur}}(\phi)\, .
\]

{\bf{Step 3.}} General case.

Let $\sigma\geq5\tau_K$. As
\[\dfrac{\log \#\mathcal{P}_K(L , L +5\tau_K)}{ L}\leq \dfrac{\log \#\mathcal{P}_K(L , L +\sigma)}{ L}\]
and
\[\limsup_{L\to\infty}\dfrac{\log \#\mathcal{P}_K(L , L +\sigma)}{ L}= h_{\mathrm{Gur}}(\phi),\]
we obtain
\[
\lim_{L\to+\infty}\dfrac{\log \#\mathcal{P}_K(L,L+\sigma)}{L}=h_{\mathrm{Gur}}(\phi)\, .
\]
\end{proof}


\subsection{Entropies at infinity}\label{sec:entropy-at-infinity}

Definition \ref{def:SPR} of {\em Strong positive recurrence}  involves a notion of {\em entropy at infinity}. We introduce here different notions of entropy at infinity and compare them in section \ref{known}.  The rough idea is to measure the exponential growth rate of the dynamics outside a large compact set $K$ and then let $K$ grow to exhaust $M$.

More precisely, for defining Gurevic entropy at infinity, we consider periodic orbits that intersect $K$ but spend only a small proportion of time in $K$.
For the variational entropy at infinity, we shall  consider the supremum of measured entropies of  probability measures that give a small measure to a large compact set $K$ .

\subsubsection{Gurevic entropy at infinity}

\begin{defi}\label{definition_gurevic_infini} Let $K\subset M$ be a compact subset with nonempty interior. Let $\alpha>0$, $L>0$ and $\sigma>0$. Define
\[
\mathcal{P}_{K}^\alpha(L,L+\sigma)=
\{ \gamma\in\mathcal{P}_K(L, L+\sigma) :\ \ell(\gamma\cap K)<\alpha\ell(\gamma) \}\,,
\]
and
\[
h_\mathrm{Gur}^{K,\alpha}(\phi)=
\limsup_{L\to \infty}\frac{1}{L}\log \#\mathcal{P}_K^\alpha(L,L+\sigma)\,.
\]
The {\em Gurevic entropy at infinity} of the flow $\phi$ is defined by
\[
h_\mathrm{Gur}^\infty(\phi):=
\, \inf_{K}\,\lim_{\alpha\to 0} h_\mathrm{Gur}^{K,\alpha}(\phi)\,,
\]
where the infimum is taken over all compact subsets $K\subset M$ with nonempty interior.
\end{defi}

\begin{fact}
Under the hypotheses of the definition, $\limsup_{L\to \infty}\frac{1}{L}\log \#\mathcal{P}_K^\alpha(L,L+\sigma)$ does not depend on $\sigma$ and therefore $h_\mathrm{Gur}^{K,\alpha}(\phi)$ is well-defined.
\end{fact}

\begin{proof}
By definition, if $\sigma\leq\sigma'$ then
\[\limsup_{L\to \infty}\frac{1}{L}\log \#\mathcal{P}_K^\alpha(L,L+\sigma)\leq \limsup_{L\to \infty}\frac{1}{L}\log \#\mathcal{P}_K^\alpha(L,L+\sigma').\]
We now prove
\[\limsup_{L\to \infty}\frac{1}{L}\log \#\mathcal{P}_K^\alpha(L,L+2\sigma)\leq \limsup_{L\to \infty}\frac{1}{L}\log \#\mathcal{P}_K^\alpha(L,L+\sigma).\]
This is enough to conclude the proof of the fact.
We have
\[ \mathcal P_{K}(L,L+2\sigma)\subset \mathcal P_{K}(L,L+\sigma)\cup \mathcal P_{K}(L+\sigma,L+2\sigma)\]
therefore
\[
\# \mathcal P_{K}(L,L+2\sigma) \leq 2 \max \left(\# \mathcal P_{K}(L,L+\sigma), \#\mathcal P_{K}(L+\sigma,L+2\sigma)\right).
\]
As
\[
\limsup_{L\to\infty} \frac{1}{L} \log \# \mathcal P_{K}(L,L+\sigma) = \limsup_{L\to\infty} \frac{1}{L} \log \left(\max \left(\# \mathcal P_{K}(L,L+\sigma), \#\mathcal P_{K}(L+\sigma,L+2\sigma)\right)\right)
\]
we have
\[\limsup_{L\to \infty}\frac{1}{L}\log \#\mathcal{P}_K^\alpha(L,L+2\sigma)\leq \limsup_{L\to \infty}\frac{1}{L}\log \#\mathcal{P}_K^\alpha(L,L+\sigma)\]
as required.
\end{proof}

\begin{fact}\label{fait:h_Gur_K_epsilon}
Let $K$ be a compact subset with nonempty interior.
The map $\alpha>0 \mapsto h_\mathrm{Gur}^{K,\alpha}(\phi)$ is non-decreasing.
Moreover, let $K'$ be a compact subset such that $K\subset\inter{K'}$.
We then have \[h_\mathrm{Gur}^{K',\alpha}(\phi)\le h_\mathrm{Gur}^{K,2\alpha}(\phi)\,.
\]
\end{fact}
\begin{proof} The first assertion is a direct consequence of the definition.

For the second assertion, we follow the arguments of the proof of Fact~\ref{dependance-compact}.
Let $\sigma>0$.
We import the notation from the proof  of Fact~\ref{dependance-compact}.
Let us assume additionally that $\epsilon\leq d(\partial K, \partial K')$.
We associate to any periodic orbit of $\mathcal{P}_{K'}^\alpha(L, L+\sigma)$ a periodic orbit of $\mathcal{P}_K(L, L+\sigma')$, for some $\sigma'\geq \sigma$, which spends a time at most $\alpha L+\sigma'$ in $K$.
For $L$ large enough, $\alpha L+\sigma'\le 2\alpha L$ and we obtain a map from $\mathcal P_{K'}^{2\alpha}(L,L+\sigma')$ to $\mathcal P_{K}^{\alpha}(L,L+\sigma)$.
The bound
\[
\left\lceil\dfrac{\sigma_0}{\tau_0}\right\rceil^2\, \left\lceil\dfrac{4b}{\epsilon}(L_1+L_2+\sigma+\sigma_0)\right\rceil\, .
\]
on the number of preimages  given in proof  of Fact~\ref{dependance-compact} remain valid.
Therefore, there exists some $D>0$ such that, for $L\gg 1$
\[
\#\mathcal{P}_{K'}^\alpha(L,L+\sigma)\le L\#\mathcal{P}_K^{2\alpha}(L,L+\sigma')\,.
\]
Since $h^{K,\alpha}_{\mathrm{Gur}}(\phi)$ does not depend on the constant $\sigma$, the result follows.
\end{proof}



\subsubsection{Variational entropy at infinity}

As in \cite{GST}, we introduce the variational entropy at infinity.
\begin{defi}\label{variational infty entropy}
The variational entropy at infinity of the flow $\phi$ is
\begin{eqnarray*}
h_{\mathrm{var}}^\infty(\phi)&:=&
\lim_{\epsilon\to 0}\inf_{K}\sup \{ h_\mathrm{KS}(\mu) :\ \mu \in \mathcal M_\phi, \mu(K)\leq \epsilon \}\\
&=&\lim_{\epsilon\to 0}\inf_{K}\sup \{ h_\mathrm{KS}(\mu) :\ \mu \in \mathcal M_\phi^{\rm erg}, \mu(K)\leq \epsilon \}\,,\\
&=&\lim_{\epsilon\to 0}\inf_{K}\sup \{ h_\mathrm{Kat}(\mu) :\ \mu \in \mathcal M_\phi^{\rm erg}, \mu(K)\leq \epsilon \}\,,
\end{eqnarray*}
where the infimum is taken over all compact subsets $K\subset M$.
\end{defi}
The equality between the two first quantities on the right follows from the fact that the entropy map $\mu\in \mathcal M_\phi\mapsto h_\mathrm{KS}(\mu)$ is convex and ergodic measures are the extremal points of $\mathcal M_\phi$.
The last equality follows from Theorem \ref{entropies-ergodic-measures}.

\begin{rema}\label{invert order in h infty var}
Observe that, in Definition \ref{variational infty entropy}, the quantity $\sup \{ h_\mathrm{Kat}(\mu) :\ \mu \in \mathcal M^{\rm erg}_\phi, \mu(K)\leq \epsilon \}$ (as well as the others appearing in the equalities) is non-decreasing in $\epsilon$ and non-increasing in $K$. Therefore, it is possible to invert the order of $\lim_{\epsilon\to 0}$ and $\inf_K$, i.e.,
\[
\lim_{\epsilon\to 0}\inf_{K}\sup \{ h_\mathrm{Kat}(\mu) :\ \mu \in \mathcal M_\phi^{\rm erg}, \mu(K)\leq \epsilon \}=\inf_{K}\lim_{\epsilon\to 0}\sup \{ h_\mathrm{Kat}(\mu) :\ \mu \in \mathcal M_\phi^{\rm erg}, \mu(K)\leq \epsilon \}\, .
\]
\end{rema}

\section{Comparison of entropies}\label{known}


\subsection{Comparison of measure-theoretic entropies}

Our main theorem (Theorem~\ref{theo:main}) establishes the existence of a measure that maximizes all notions of measured entropy. This measure will be obtained as a limit of  averages of periodic measures. As a consequence, on the one hand, it is a priori not known to be ergodic, and on the other hand,  its Katok and Brin-Katok entropies are the only ones that are computable.  Therefore, we will need general statements to be able to compare all kinds of entropies.

\begin{prop}[Riquelme \cite{Riquelme}] \label{Katok-noncompact}
 Let $\phi$ be a Lipschitz flow on a manifold $M$ and  $\mu\in \mathcal{M}_\phi$ an invariant  probability measure.
If $K\subset M$ is a compact subset, then, for $\epsilon>0$ small enough,
    \[
    \int_K\limsup_{\substack{n\to+\infty \\ \phi^n(x)\in K}}-\frac{1}{n}\log\,\mu(B(x,\epsilon,n))\, d\mu\leq h_{\mathrm{KS}}(\mu)\,.
    \]
\end{prop}
This proposition is proven in \cite{Riquelme} (see also \cite[Appendix A]{GST}) when $\mu$ is ergodic, and the non ergodic case, very similar, is only briefly mentioned. As it is crucial for us, we give a proof of this statement.

\begin{proof} In \cite[Theorem 2.10]{Riquelme}, Riquelme uses a proposition due to Ledrappier, see  \cite[Proposition 6.3]{Ledrappier}, to build a partition $\mathcal P$   such that (without ergodicity), for
$\mu$-almost every $x$, for every $n$ such that $\phi^n(x)\in K$, we have $\mathcal{P}^{n}(x)\subset B(x,n,\epsilon)$, where for a partition $\mathcal{P}$ we denote by $\mathcal{P}(x)$ the element of the partition containing $x$, and where $\mathcal{P}^n$ is the measurable partition consisting of all possible intersections of elements of $\varphi^{-i}\mathcal{P}$, for $i=0,\dots, n-1$. It follows that
\[
\int_K\limsup_{\substack{n\to+\infty \\ \phi^n(x)\in K}}-\frac{1}{n}\log\left(\mu(B(x,\epsilon,n)\right)\, d\mu \le
\int_M \limsup_{n\to\infty} -\frac{1}{n}\log \mu\left(\mathcal{P}^n(x)\right)\,d\mu\, .
\]
The non-ergodic version of Shannon-McMillan-Breiman theorem  ensures that $-\frac{1}{n}\log \mu\left(\mathcal{P}^n(x)\right)$ converges almost surely, so that the right hand side is in fact a true (almost sure) limit.
This theorem is stated without proof in
 \cite[Theorem 1.2, Chapter IV]{Mane}. It is stated and proven in \cite[Theorem 2.5]{Krengel} in a more general framework, and the proof of \cite[Theorem 2.3, p. 261]{Petersen} in the ergodic case adapts almost \textit{verbatim} to the non-ergodic case.

By Fatou's Lemma, we get
\[
\int_M \lim_{n\to\infty} -\frac{1}{n}\log \mu\left(\mathcal{P}^n(x)\right)\,d\mu \le \liminf_{n\to \infty}\int_M -\frac{1}{n}\log \mu(\mathcal{P}^n(x))\,d\mu\,.
\]
By definition of the entropy of a partition, we have
\[
\int_{M}-\frac{1}{n}\log \mu\left(\mathcal{P}^n(x)\right)\,d\mu=\frac{1}{n}H(\mathcal{P}^n,\mu)\,,
\]
and this quantity converges to $h(\mu,\mathcal{P})\le h_{\mathrm{KS}}(\mu)$.
\end{proof}

\begin{theo}[ Brin-Katok \cite{BK}, Katok \cite{Katok}, Riquelme \cite{Riquelme}]\label{entropies-ergodic-measures} Let $\phi$ be a Lipschitz flow on a complete Riemannian manifold $M$.
Let $\mu\in\mathcal{M}_\phi^{\mathrm{erg}}$. Then
\[
h_\mathrm{Kat}(\mu)= \underline{h}_\mathrm{BK}(\mu)=\overline{h}_\mathrm{BK}(\mu)=h_{\mathrm{KS}}(\mu)\, .
\]
\end{theo}

\begin{proof} The inequality $h_{\mathrm{KS}}(\mu)\le h_{\mathrm{Kat}}(\mu)$ is stated in \cite{Katok} in the compact case, but the proof does not use compactness. The inequality $h_{\mathrm{KS}}(\mu)\le \underline{h}_{\mathrm{BK}}(\mu)$ is stated in \cite{BK} in the compact case but the proof does not use compactness either.

 Riquelme \cite[Theorems 2.8, 2.9, 2.10, 2.13]{Riquelme} establishes the other (in)equalities.
 \end{proof}

The following   intermediate result, of independent interest, is proven in section \ref{inegalite-Gurevic}.
\begin{theo} \label{comparison-Kat-BK} Let $\varphi$ be a $H$-flow on a manifold $M$. For every invariant probability measure
$\mu\in \mathcal{M}_\varphi$, one has
\[
h_{\mathrm{Kat}}(\mu)\le h_{\mathrm{Gur}}(\varphi)\,.
\]
\end{theo}

Our construction of a measure of maximal entropy in this paper will produce a measure that is \textit{a priori} not necessarily ergodic, so that it is worth noting the following corollary. Recall that we use the notation $\mathcal{M}_\phi$ for the set of $\phi$-invariant probability measures, and $\mathcal{M}^{\rm erg}_\phi$ for the set of $\phi$-invariant, ergodic, probability measures.
\begin{coro}\label{entropie-var}
Let $\phi$ be a $H$-flow on a manifold $M$. We have
\begin{eqnarray*}
h_{\rm var}(\phi)
&=&
\sup_{\mu\in\mathcal{M}_\phi} h_\mathrm{KS}(\mu)
=
\sup_{\mu\in \mathcal{M}^{\rm erg}_\phi} h_\mathrm{KS}(\mu)
=
\sup_{\mu\in \mathcal{M}^{\rm erg}_\phi} h_\mathrm{Kat}(\mu)=
\sup_{\mu\in \mathcal{M}^{\rm erg}_\phi} \underline{h}_\mathrm{BK}(\mu)\\
&\le&
\sup_{\mu\in\mathcal{M}_\phi} h_\mathrm{Kat}(\mu)\\
&\le& h_\mathrm{Gur}(\phi)
\,.\end{eqnarray*}
\end{coro}

\subsection{Katok entropy is smaller than Gurevic entropy - Proof of Theorem~\ref{comparison-Kat-BK}}\label{inegalite-Gurevic}

In this subsection, we are going to prove Theorem~\ref{comparison-Kat-BK}. The rough idea goes as follows. Given a $\phi$-invariant probability measure $\mu$ and  compact set $K$, we can assume that each point of a spanning set for $\mu$ (whose cardinality is used to calculate its Katok entropy) lies in $K$ and comes back to $K$ after a time $T$. Thanks to the transitivity property and the finite exact shadowing property, we can close up the piece of orbit of each point of the spanning set to obtain a periodic orbit intersecting $K$. By controlling the default of injectivity of such a procedure, we will conclude that the Gurevich entropy is larger than the Katok entropy of $\mu$. We start now with the details of the proof.

Firstly, we introduce the notion of separating spanning sets and prove that they can be equivalently used to define the Katok entropy of a measure.

\begin{defi}
Let $T>0$ and $\epsilon>0$. A set $E\subset X$ is $(\epsilon,T)$-separating if for all $x,y\in E, x\neq y$ there exists $t\in[0,T]$ such that
\[
d(\phi_t(x),\phi_t(y))\geq \epsilon\, .
\]
\end{defi}
Let $\mu\in\mathcal{M}_\phi$.
\begin{defi}
Let $T>0$, $\epsilon>0$ and $\alpha\in(0,1)$. A set $E$ is a \textit{separating} $(T,\epsilon,\alpha,\mu)$-spanning set if it is a $(T,\epsilon,\alpha,\mu)$-spanning set (see Definition \ref{defi spanning}) and it is $(\frac\epsilon 2, T)$-separating.
\end{defi}

Recall that $M(T,\epsilon,\alpha,\mu)$ denotes the minimal cardinality of a $(T,\epsilon,\alpha,\mu)$-spanning set. Similarly, denote by $M'(T,\epsilon,\alpha,\mu)$ the minimal cardinality of a separating $(T,\epsilon,\alpha,\mu)$-spanning set. We then have $M(T,\epsilon,\alpha,\mu)\leq M'(T,\epsilon,\alpha,\mu)$, since any separating $(T,\epsilon,\alpha,\mu)$-spanning set is also a $(T,\epsilon,\alpha,\mu)$-spanning set.
\label{p:M'}

\begin{lemm}\label{lemme:comparaison_M_M'}
Let $T>0$, $\epsilon>0$ and $\alpha\in(0,1)$. Then $M'(T,2\epsilon,\alpha,\mu)\leq M(T,\epsilon,\alpha,\mu)$.
\end{lemm}
\begin{proof}
Let $E$ be a $(T,\epsilon,\alpha,\mu)$-spanning set of minimal cardinality $M=M(T,\epsilon,\alpha,\mu)$. Enumerate the elements of $E$ as $\{x_1,\dots,x_M\}$. We select a subset $E'$ of $E$ as follows.
\begin{enumerate}
    \item The point $x_1\in E'$.
    \item Consider the dynamical ball centered at $x_1$ of radius $2\epsilon$. For $i>1$, we erase the point $x_i$, i.e., $x_i\notin E'$, if and only if the dynamical ball $B(x_i,\epsilon,T)\subset B(x_1,2\epsilon,T)$.
    \item We consider the next point $x_j$ among the remaining ones. We have not erased it at the previous step; we then keep it and say that it belongs to $E'$.
    \item  We iterate now the erasing procedure, starting with $x_j$. Consider the dynamical ball $B(x_j,2\epsilon,T)$. For $i>j$, we erase the point $x_i$, i.e., $x_i\notin E'$, if and only if $B(x_i,\epsilon,T)\subset B(x_j,2\epsilon,T)$.
\end{enumerate}
We have then $\# E'\leq \# E=M(T,\epsilon,\alpha,\mu)$. We are now going to show that $E'$ is a separating $(T,2\epsilon,\alpha,\mu)$-spanning set. This will imply then
$
M'(T,2\epsilon,\alpha,\mu)\leq \# E'
$, concluding our proof. For every $x\in E$
\begin{enumerate}
    \item either $x\in E'$. In this case observe that $B(x,\epsilon,T)\subset B(x,2\epsilon,T)$;
    \item or $x\notin E'$. In this case, it means that there exists another $\bar x\in E'$ such that $B(x,\epsilon,T)\subset B(\bar x,2\epsilon,T)$.
\end{enumerate}
Thus
\[
\bigcup_{x\in E}B(x,\epsilon,T)\subset \bigcup_{x\in E'}B(x,2\epsilon,T)\,,
\]
and so
\[
\mu\left(   \bigcup_{x\in E'}B(x,2\epsilon,T)\right)\geq \mu\left(\bigcup_{x\in E}B(x,\epsilon,T)\right)\geq \alpha\,,
\]
where the last inequality comes from the fact that $E$ is a $(T,\epsilon,\alpha,\mu)$-spanning set. So, the set $E'$ is a $(T,2\epsilon,\alpha,\mu)$-spanning set.

Moreover, let $x_i,x_j\in E'$ with $i>j$. There exists $t\in[0,T]$ such that $d(\phi_t(x_i),\phi_t(x_j))\geq \epsilon$. Indeed, if not, it would imply that
\[
B(x_i,\epsilon,T)\subset B(x_j,2\epsilon,T)\,,
\]
which is in contradiction with the construction of $E'$. That is, $E'$ is also $(\epsilon,T)$-separating, which concludes the proof.
\end{proof}

We then deduce the following corollary.
\begin{coro}\label{coro katok entropy}
Let $\mu\in \mathcal{M}_\phi$. 
Then
\[
h_\mathrm{Kat}(\mu)= \inf_{\alpha>0}\sup_{\epsilon>0}\limsup_{T\to+\infty}\dfrac{1}{T}\log\left(M'(T,\epsilon,\alpha,\mu)\right)\,.
\]
\end{coro}

Notice that $\epsilon\mapsto M(T,\epsilon,\alpha,\mu)$ and $\alpha\mapsto M(T,\epsilon,\alpha,\mu)$ are non increasing.
Yet, while $\alpha\mapsto M'(T,\epsilon,\alpha,\mu)$ is also non-decreasing, we have a prioi no control on $\epsilon\mapsto M'(T,\epsilon,\alpha,\mu)$.

We now prove a lemma analogous to Lemma~\ref{lemme:comparaison_M_M'} with the extra condition that the points in the separating-spanning set should belong to a fixed compact set $K$.
This is the first step to prove Theorem~\ref{comparison-Kat-BK}.

\begin{lemm}\label{lemma replace}
Let $E$ be a $(T,\epsilon,\alpha,\mu)$-spanning set and let $K\subset M$ be a subset such that $\mu(K)=\beta$. Then there exists a separating $(T,2\epsilon, \alpha+\beta-\mu(M),\mu)$-spanning set $E'\subset K$ such that $\# E'\leq \# E$.
\end{lemm}

\begin{proof}
We construct the set $E'$ inductively as follows. Let us enumerate $E$ as $\{x_1,\dots,x_N\}$.
\begin{enumerate}
    \item If $B(x_1,\epsilon,T)\cap K\neq \emptyset$, then we choose $y_1\in B(x_1,\epsilon,T)\cap K$ and set $E'_1 = \{y_1\}$. Otherwise $E'_1=\emptyset$.
    \item For $i\geq 2$, if \[B(x_i,\epsilon,T)\cap K\cap (M\setminus \bigcup_{y_j\in E'_{i-1}}B(y_j,\epsilon,T))\neq \emptyset\, ,\]
    then we choose $y_i\in B(x_i,\epsilon,T)\cap K\cap (M\setminus \bigcup_{y_j\in E'_{i-1}}B(y_j,\epsilon,T))$ and add $y_i$ to $E'_{i-1}$ to obtain $E'_i$. Otherwise let $E'_{i}=E'_{i-1}$.
\end{enumerate}
Let $E'=E'_N$.
Observe that $\#E'\leq \#E$ and that $E'\subset K$.
Moreover, we have
\begin{equation}\label{inter dans K aussi dans boules 2 eps}
 \bigcup_{x\in E}B(x,\epsilon,T)\cap K\subset \bigcup_{y\in E'}B(y,2\epsilon,T)\, .
\end{equation}
Indeed, if $y_i\in E'$, then in particular $y_i\in B(x_i,\epsilon,T)\cap K$ and so
\[
B(x_i,\epsilon,T)\cap K\subset B(x_i,\epsilon,T)\subset B(y_i,2\epsilon,T)\,.
\]
If $x_i$ is such that there is no corresponding $y_i\in E'$, then $B(x_i,\epsilon,T)\cap K\cap (M\setminus \bigcup_{y_j\in E'_{i-1}}B(y_j,\epsilon,T))=\emptyset$. If $B(x_i,\epsilon,T)\cap K=\emptyset$, then the empty set is clearly contained in $\bigcup_{y\in E'}B(y,2\epsilon,T)$.
If not, the only possibility is that 
$B(x_i,\epsilon,T)\cap K\subset \bigcup_{y_j\in E'_{i-1}}B(y_j,\epsilon,T)$; then
\[
B(x_i,\epsilon,T)\cap K\subset \bigcup_{y\in E'}B(y,\epsilon,T)\subset \bigcup_{y\in E'}B(y,2\epsilon,T)\, .
\]
We now argue that $E'$ is a $(T,2\epsilon, \alpha+\beta-\mu(M),\mu)$-separating set.
Indeed, from \eqref{inter dans K aussi dans boules 2 eps},
\begin{align*}
    \mu\left( \bigcup_{y\in E'}B(y,2\epsilon,T) \right)&\geq \mu\left( \bigcup_{x\in E}B(x,\epsilon,T) \cap K\right)\\
    &\geq \mu\left( \bigcup_{x\in E}B(x,\epsilon,T) \right)+\mu(K)-\mu(M)\\
    &\geq \alpha+\beta-\mu(M)\, ,
\end{align*}
where the last inequality comes from $E$ being a $(T,\epsilon,\alpha, \mu)$-spanning set.
Moreover the set $E'\subset K$ is also separating. Indeed, let $y,y'\in E'$, $y\neq y'$. In particular, by the construction of $E'$, it means that $y'\notin B(y,\epsilon,T)$ (or viceversa): there exists $\tau\in[0,T]$ such that $d(\phi_{\tau}(y),\phi_{\tau}(y'))\geq \epsilon$.
\end{proof}

\begin{proof}[Proof of Theorem~\ref{comparison-Kat-BK}]
Let $\mu\in\mathcal{M}_\phi$. We are now going to prove that the Katok entropy of $\mu$ is smaller or equal to the Gurevic entropy.
Let $\epsilon>0$ and $\alpha_1\in(0,1)$.
Pick $\alpha'\in(0,1)$ such that $ \alpha_1+\alpha'\in(0,1)$.
Let $\alpha_0=\alpha_1+\alpha'$.

Let $K\subset M$ be a compact subset with nonempty interior such that $\mu(K)\geq 1-\alpha'/2$ (such a $K$ exists as $M$ is exhaustible by compact sets).
Let us apply Lemma~\ref{petal separe} at $K$, $\delta=\frac{\epsilon}{3}$, $\nu=\tau_K$ and $N=1$. The lemma gives us constants $\sigma>0, T_\mathrm{min}>0$. Let $T\geq T_\mathrm{min}$ and $K'=K\cap \phi_{-T}(K)$.
Observe that $\mu(K')\geq 1-\alpha'$ as $\mu$ is $\phi$-invariant.
Let $E$ be a
$(T,\epsilon,\alpha_0,\mu)$ spanning set. By Lemma \ref{lemma replace}, there exists a separating $(T,2\epsilon,\alpha_1,\mu)$ spanning set $E'$ such that $\#E'\leq \#E$ and $E'\subset K'$.

We can associate to every $x\in E'$ a periodic point $y$, thanks to Lemma \ref{petal separe}. Indeed, since $x\in E'\subset K'$, it holds that both $x$ and $\phi_T(x)$ belongs to $K$, and by construction $T\geq T_\mathrm{min}$. Thus, there exists $y\in M$ and $L\in[T+\sigma-2\tau_K,T+\sigma+2\tau_K]$ so that $\phi_L(y)=y$ and $d(\phi_t(y),\phi_t(x))<\frac{\epsilon}{3}$ for every $t\in[0,T]$.

We can give an upper bound on the number of points in $E'$ that could be associated to the same periodic point. Fix $x,x'\in E'$ with $x\neq x'$. Let  $\gamma$ and $\gamma'$ be the associated periodic orbits. Assume $\gamma=\gamma'$.
Then, there exists $u$ such that, for all $s\in[0,T]$,
\begin{equation*}
    d(\phi_s(x),\gamma(s))<\frac{\epsilon}{3}\qquad\text{and}\qquad d(\phi_s(x'),\gamma(s+u))<\frac{\epsilon}{3}\,.
\end{equation*}
Moreover, since $E'$ is $(\epsilon,T)$ separating, then there exists $\tau\in[0,T]$ such that
\begin{equation*}
    d(\phi_\tau(x),\phi_\tau(x'))\geq \epsilon\,.
\end{equation*}
Therefore,
\begin{equation*}
    d(\phi_\tau(y),\phi_\tau(y'))=d(\gamma(\tau),\gamma(\tau+u))\geq \frac{\epsilon}{3}\,.
\end{equation*}
From the right inequality of \eqref{eqn:minoration}, we deduce that $\vert u\vert \geq \epsilon/3b$. Consequently, there are at most $\left\lceil(T+\sigma+2\tau_K)\frac{3b}{\epsilon}\right\rceil$ points of $E'$ that could correspond to points on the same periodic orbit. This implies that
\[
\# E'\leq \left\lceil\frac{3b}{\epsilon}(T+\sigma+2\tau_K)\right\rceil\,\#\mathcal{P}_K(T+\sigma-2\tau_K,T+\sigma+2\tau_K)\,.
\]
By the definition of Katok entropy for $\mu$, by Proposition \ref{ind gur entropy} and by Corollary \ref{coro katok entropy}, we deduce that
\[ \limsup_{T\to+\infty}\dfrac{1}{T}\log\left(M'(T,2\epsilon,\alpha_1,\mu)\right)\leq h_\mathrm{Gur}(\phi)\]
and therefore
\[
h_\mathrm{Kat}(\mu)\leq h_\mathrm{Gur}(\phi)\,.
\]
\end{proof}

\subsection{Comparison of entropies at infinity}

With similar ideas, we can also compare Gurevic and variational entropies at infinity.

\begin{theo}[Comparison of entropies at infinity]\label{theo:comparison-infinity} Let $\phi$ be a $H$-flow on $(M,d)$. The entropies at infinity satisfy
 \[
  h_{\mathrm{var}}^\infty(\phi) \le h_{\mathrm{Gur}}^\infty(\phi)\,.
 \]
\end{theo}

\begin{proof}
If $ h_{\mathrm{Gur}}^\infty(\phi)=+\infty$ there is nothing to prove.
We now assume $h_{\mathrm{Gur}}^\infty(\phi)<\infty$.

Assume $h_{\mathrm{var}}^\infty(\phi)<\infty$.
Fix  some small $\alpha>0$. Thanks to Remark \ref{invert order in h infty var}, we choose a large compact set $K_0$, a small $0<\eta<\frac 1 2$, and an ergodic probability measure
$\mu\in\mathcal{M}_\varphi^\mathrm{erg}$ with $\mu(K_0)\le \eta$ so that
\[
\left| h_\mathrm{Gur}^\infty(\varphi)-h_\mathrm{Gur}^{K_0,4\eta}(\phi)\right|\le \alpha\quad\mbox{and}\quad
\left|h_\mathrm{var}^\infty(\varphi)-h_\mathrm{Kat}(\mu)\right|\le \alpha\,.
\]

We follow very closely the proof of Theorem~\ref{comparison-Kat-BK}, with a few modifications.

Let $\epsilon>0$ and $\alpha_1\in(0,1)$.
Pick $\alpha'$ such that $\alpha_0= \alpha_1+\alpha'\in(0,1)$. Fix $w\in \mathrm{int}(K_0)$.
Choose $\delta\leq \frac{\epsilon}{3}$ small enough so that $\mu(B(K_0,\delta))\leq 2\eta$ and $B(w,\delta)\subset K_0$.
Fix $K_2$ a compact set such that $K_2\supset B(K_0,\delta)$ and $\mu(K_2)\ge 1-\frac{\alpha'}{4}$.
Let $\Omega = B(K_0,\delta)$.

By Birkhoff ergodic theorem, for $\mu$-almost every $x\in M$,
\[
\lim_{T\to\infty}\frac{1}{T}\int_0^T 1_{\Omega}(\phi_t(x))\,dt  = \mu(\Omega)\, ,
\]
where $1_\Omega$ is the indicator function of $\Omega$.
As $\mu(\Omega)\le 2\eta$ and $\mu(K_2)\ge 1-\frac{\alpha'}{4}$, there exists a subset $A\subset K_2$ with $\mu(A)\ge 1-\frac{\alpha'}{2}$ and $S_\mathrm{min}>0$, such that for all $x\in A$ and all $T\geq S_{\mathrm{min}}$,
\begin{equation}\label{Birk}
\frac{1}{T}\int_0^T 1_{\Omega}(\phi_t(x))\,dt \le 3\eta\,.
\end{equation}
Note that, elements of the closure of $A$, denoted as $\overline{A}$, also satisfies condition~(\ref{Birk}). Indeed, let $(x_n)_n$ be a sequence of elements of $A$ converging to $x$. Then, for all $T\geq S_\mathrm{min}$, as $\Omega$ is an open set and by Fatou's lemma
\[
\frac{1}{T}\int_0^T 1_{\Omega}(\phi_t(x))\,dt \leq \frac{1}{T}\int_0^T \liminf_{n\to\infty} 1_{\Omega}(\phi_t(x_n))\,dt \leq \liminf_{n\to\infty}\frac{1}{T}\int_0^T 1_{\Omega}(\phi_t(x_n))\,dt\leq 3\eta.
\]

We will use Lemma~\ref{petal separe} with parameters $\overline{B(w,\delta)}\subset K_2$, $\delta$ as above, $\nu=\tau_{K_2}$ and $N=1$.
This lemma gives us constants $T_\mathrm{min}$ and $\sigma$.
Let $T\geq \max(T_\mathrm{min}, S_\mathrm{min})$.
Set $K_1=\overline{A}\cap \phi_{-T}(\overline{A})$, and observe that $\mu(K_1)\geq  1-\alpha'$.

Let $E$ be a separating $(T,\epsilon,\alpha_0,\mu)$-spanning set. Without loss of generality, we can assume that for every $x\in E$, $\mu(B(x,\epsilon,T))>0$.

By Lemma~\ref{lemma replace}, there exists a separating $(T,2\epsilon,\alpha_1,\mu)$ spanning set $E'$ such that $\#E'\leq \#E$ and $E'\subset K_1$.
Without loss of generality, we can assume that for every $x\in E'$, $\mu(B(x,2\epsilon,T))>0$.

Observe that, since $E'\subset \overline{A}$, every point $x\in E'$ satisfies inequality \eqref{Birk}. Fix now a point $x\in E'$. Note that $x\in K_2$ and $\phi_T(x)\in K_2$.

By Lemma~\ref{petal separe}, we obtain a periodic orbit $\gamma$ such that
\begin{itemize}
    \item for all $t\in [0,T]$, we have $d(\phi_t(x),\gamma(t))\leq \delta$;
    \item $\ell(\gamma)\in [T+\sigma-2\tau_{K_2},T+\sigma+2\tau_{K_2}] $;
    \item $\gamma$ intersects $B(w,\delta)$.
\end{itemize}
Therefore, since by \eqref{Birk} the piece of orbit $\phi_{[0,T]}(x)$ spends at most a total amount of time of $3\eta T$ in $B(K,\delta)=\Omega$, we have
\[
\ell(\gamma\cap K_0) = \ell(\gamma_{[0,T]}\cap K_0) +\ell(\gamma_{[T,\ell(\gamma)]}\cap K_0)\leq 3\eta T + \sigma + 2\tau_{K_2}.
\]
Thus $\ell(\gamma\cap K_0) \leq 4\eta \ell(\gamma)$ for $T\gg 1$ and therefore $\gamma\in \mathcal{P}_{K_0}^{4\eta}(T+\sigma-2\tau_{K_2},T+\sigma+2\tau_{K_2})$.

The end of the proof is the same as the one of Theorem~\ref{comparison-Kat-BK}.
Since $E'$ is $(\epsilon,T)$-separating, we deduce that
\[
\# E'\leq \left\lceil\frac{3b}{\epsilon}(T+\sigma+2\tau_{K_2})\right\rceil\,\#\mathcal{P}_{K_0}^{4\eta}(T+\sigma-2\tau_{K_2},T+\sigma+2\tau_{K_2})\,.
\]
By the definition of Katok entropy for $\mu$, by Proposition \ref{ind gur entropy} and by Corollary \ref{coro katok entropy}, we deduce that
\[
\limsup_{T\to+\infty}\dfrac{1}{T}\log  M'(T,2\epsilon,\alpha_1,\mu) \leq h_{\mathrm{Gur}}^{K_0,4\eta}(\phi)
\]
and therefore
\[
h_\mathrm{Kat}(\mu)\leq h_\mathrm{Gur}^{K_0,4\eta}(\phi)\,,
\]
and at the end
\[
h_{\mathrm{var}}^\infty(\phi)\le h_{\mathrm{Gur}}^\infty(\phi)+2\alpha\,.
\]
As $\alpha$ was arbitrary, the result follows.

If $h_{\mathrm{var}}^\infty(\phi)=+\infty$, we choose $\mu$ ergodic such that $h_\mathrm{Kat}(\mu) \geq N$ and proceed as in the previous case. We obtain
$N\leq h_\mathrm{Gur}^{K_0,4\eta}(\phi)$ for arbitrarily big $N$
and therefore $h_{\mathrm{Gur}}^\infty(\phi) = \infty$.
\end{proof}

\subsection{Strong positive recurrence}

In \cite{ST19,GST}, for geodesic flows in negative curvature, the geodesic flow is said {\em  strongly positively recurrent} if its entropy at infinity is strictly smaller than its topological entropy. In this context, all notions of entropy (resp. entropy at infinity) coincide, as proven in \cite{GST}. Here, it  is not the case. That motivates the following terminology.

\begin{defi}\label{def:SPR}
The flow $(\phi_t)_{t\in\mathbb R}$ is \emph{$h_{\mathrm Gur}$-strongly positively recurrent} if $h_{\mathrm Gur}^\infty(\phi) < h_{\mathrm Gur}(\phi)$.
It is  \emph{$h_{\mathrm var}$-strongly positively recurrent} if $h_{\mathrm var}^\infty(\phi) < h_{\mathrm var}(\phi)$.
\end{defi}

\begin{rema}
Theorems~\ref{theo:comparison-infinity}  and \ref{theo:main}  show that  if $\phi$ is a $h_{\mathrm Gur}$-strongly positively recurrent $H$-flow, then it is also  \emph{$h_{\mathrm var}$-strongly positively recurrent}.
\end{rema}

\subsection{Gurevic entropies}

The end of this section is devoted to the proof of the following proposition.

\begin{prop}\label{prop_entropies_Gur_finies_simultanement}
Let $\phi$ be a $H$-flow. Then $h_\mathrm{Gur}^\infty (\phi)<\infty$ if and only if $h_\mathrm{Gur}(\phi)<\infty$.
\end{prop}

We will start with two preliminary results.

The following lemma is a rephrasing of the finiteness of entropy on compact sets.
It will be useful to control the entropy using the entropy at infinity in the proof of the  proposition and also in the proof of Theorem~\ref{difficult-technical}.

\begin{lemm}\label{lemme:entropie_gur_compact}
Let $K_0\subset K_1$ be two compact subsets of $M$.
Let $\phi$ be an expansive flow on $M$.
Let $C>0$.
Then
\[
\limsup_{T\to\infty}\frac{1}{T}\log \#\{\gamma\in\mathcal P_{K_0}(T,T+C), \gamma\subset K_1 \} <\infty\, .
\]
\end{lemm}

\begin{proof}
Consider the closure of the set of all periodic orbits of $\phi$ that are contained in the compact set $K_1$. Denote such a closed set by $X$. Then, endowing it with the distance inherited from that of the whole $M$, the set $X$ is a metric space. The flow $\phi$ restricted to $X$ is still expansive. Since we can find a bigger compact set that contains it, the set $X$ is also compact. Applying then \cite[Theorem 5]{BowWal}, the result follows immediately.
\end{proof}

The following proposition is an adaptation of Proposition~\ref{prop subadd gen}.

\begin{prop}\label{prop subadd gen variante}
Let $\phi\colon M\to M$ be a $H$-flow. Let $K_0$, $K_1$ be two compact subsets of $M$ with nonempty interior and such that $K_0\subset \inter{K_1}$.
Let $0<\alpha<1$ and $\sigma_0\geq 4\tau_{K_1}$.
There exist constants $D>0$ and $\sigma>0$ such that for all $L_1,L_2\gg 1$ with $\frac{\alpha}{3}L_2\ge L_1$, one has
\begin{equation*}
\# \mathcal{P}_{K_0}(L_1,L_1+\sigma_0)\, \#\mathcal{P}^{\alpha/3}_{K_1}(L_2,L_2+\sigma_0)\leq D\, (L_1+L_2)\, \#\mathcal{P}^\alpha_{K_0}(L_1+L_2+\sigma, L_1+L_2+\sigma+\sigma_0)\, .
\end{equation*}
\end{prop}

\begin{figure}
    \centering
    \includegraphics{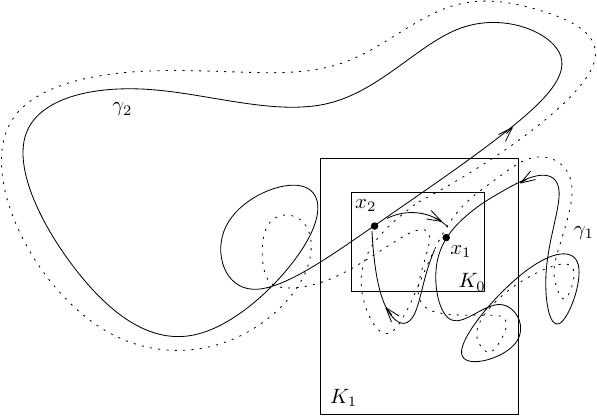}
    \caption{Modified subadditivity property}
    \label{fig:sous_additivite_simple_variante}
\end{figure}

\begin{proof}[Proof of Proposition~\ref{prop subadd gen variante}]
The proof is a direct adaptation of the proof of Proposition~\ref{prop subadd gen}.
See Figure~\ref{fig:sous_additivite_simple_variante}.
We import the notation from this proof. Choose $\epsilon>0$ as in the proof of Proposition~\ref{prop subadd gen}.
Without loss of generality, we can assume that $\epsilon/4<d(\partial K_0,\partial K_1)$.
The only new property to check is that if
$\gamma_1\in \mathcal{P}_{K_0}(L_1,L_1+\sigma_0)$,
$\gamma_2\in \mathcal{P}^{\alpha/3}_{K_1}(L_2,L_2+\sigma_0)$ and $\frac{\alpha}{3}L_2\ge L_1$
then $\gamma = f(\gamma_1,\gamma_2)$ satisfies $\ell(\gamma\cap K_0)< \alpha \ell(\gamma)$. The very same construction in the proof of Proposition~\ref{prop subadd gen} gives us the constants $D>0$ and $\sigma>0$.
Recall that $\gamma$ consists in four pieces
\begin{itemize}
    \item a piece $\epsilon/4$-close to $\gamma_1$ on an interval of length $\ell(\gamma_1)$;
    \item a piece $\epsilon/4$-close to $\gamma_2$ on an interval of length $\ell(\gamma_2)$;
    \item two transition pieces of total length $\leq \sigma+\sigma_0$.
\end{itemize}
Therefore
\begin{align*}
\ell(\gamma\cap K_0)&\leq \ell(\gamma_1)+ \ell(\gamma_2\cap K_1) + \sigma+\sigma_0 \\
&< \ell(\gamma_1)+ \frac{\alpha}{3}\ell(\gamma_2) +  \sigma+\sigma_0\\
&\leq L_1+\sigma_0 + \frac{\alpha}{3}(L_2+\sigma_0) + \sigma + \sigma_0\\
&\leq \frac{2\alpha}{3} L_2 + \sigma + 3\sigma_0\\
&\leq \alpha L_2\\
&\leq \alpha\ell(\gamma)
\end{align*}
for $L_2\gg1$.
Thus $\gamma\in \mathcal{P}^\alpha_{K_0}(L_1+L_2+\sigma, L_1+L_2+\sigma+\sigma_0)$.
The rest of the proof is unchanged.
\end{proof}

\begin{proof}[Proof of Proposition~\ref{prop_entropies_Gur_finies_simultanement}]
If $h_\mathrm{Gur}(\phi)<\infty$ then $h_\mathrm{Gur}^\infty(\phi)<\infty$.

We now assume $h_\mathrm{Gur}^\infty(\phi)<\infty$.
There exists $K_0$ compact and $\alpha_0>0$ such that for all $\alpha\leq \alpha_0$
\[h_\mathrm{Gur}^{K_0,\alpha}(\phi)\le h_\mathrm{Gur}^\infty(\phi)+1.
\]
Therefore, for $L\gg 1$,
\[
\# \mathcal{P}^{\alpha_0}_{K_0}(L,L+\sigma_0)\leq e^{\left(h_\mathrm{Gur}^\infty(\phi)+2\right)L}.
\]

Let $K_1$ be a compact subset of $M$ such that $K_0\subset \inter{K_1}$.
We now use Proposition~\ref{prop subadd gen variante} with parameters $K_0$, $K_1$, $\alpha_0$ and $\sigma_0=5\tau_{K_0}$.
For all $L_1,L_2\gg 1$ with $\frac{\alpha_0}{3}L_2\ge L_1$
\begin{equation*}
\# \mathcal{P}_{K_0}(L_1,L_1+\sigma_0)\, \#\mathcal{P}^{\alpha_0/3}_{K_1}(L_2,L_2+\sigma_0)\leq D\, (L_1+L_2)\, \#\mathcal{P}^{\alpha_0}_{K_0}(L_1+L_2+\sigma, L_1+L_2+\sigma+\sigma_0)\, .
\end{equation*}

Let us first assume there exist an increasing sequence $(L_2^n)_{n\in\mathbb{N}}$ such that $\lim_{n\to\infty} L_2^n = \infty$ and, for all $n\ge 0$,
\[\#\mathcal{P}^{\alpha/3}_{K_1}(L^n_2,L^n_2+\sigma_0)\ge 1.\]
Let $L_1^n = \frac{\alpha_0}{3} L_2^n$.
Then, for all $n\gg 1$
\begin{align*}
    \# \mathcal{P}_{K_0}(L^n_1,L^n_1+\sigma_0)&\leq  D \left(1+\frac{3}{\alpha_0}\right)L^n_1 \#\mathcal{P}^{\alpha_0}_{K_0}\left(\left(1+\frac{3}{\alpha_0}\right)L^n_1+\sigma, \left(1+\frac{3}{\alpha_0}\right)L^n_1+\sigma+\sigma_0\right)\\
    &\leq   D \left(1+\frac{3}{\alpha_0}\right)L^n_1 e^{\left(h_\mathrm{Gur}^\infty(\phi)+2\right)\left(\left(1+\frac{3}{\alpha_0}\right)L^n_1+\sigma \right)}.
\end{align*}
As the Gurevic entropy is a true limit, not only a superior limit when $\varphi$ is a $H$-flow (see Theorem~\ref{theo:Gurevic}), we then have
\[h_\mathrm{Gur}(\phi)\leq \left(h_\mathrm{Gur}^\infty(\phi)+2\right)\left(1+\frac{3}{\alpha_0}\right) < \infty.\]

We now assume
\[
\#\mathcal{P}^{\alpha_0/3}_{K_1}(L_2,L_2+\sigma_0)=0
\]
for all $L_2\gg 1$.
In particular, there do not exist arbitrarily long chords of $\partial K_1$ outside $K_1$ (otherwise, by transitivity and the uniform multiple closing lemma~\ref{petal separe}, one may construct a periodic orbit in some $\mathcal{P}^{\alpha_0/3}_{K_1}(L_2,L_2+\sigma_0)$).
Therefore all the periodic orbits intersecting $K_0$ are contained in some compact $K_2$.
From Lemma~\ref{lemme:entropie_gur_compact}, we obtain that the Gurevic entropy is finite.
\end{proof}

\section{Gurevic entropy versus chord entropy}\label{chords}

It will be useful in the sequel to count chords, i.e., pieces of orbits from the neighbourhood of a point to the neighbourhood of another point, and try to compare their number with the number  of periodic orbits.
In Section~\ref{sec:chords-definition} we explain how to count chords and state some elementary properties of chords counts.

 \subsection{Chords}\label{sec:chords-definition}

The aim of this section is to introduce these chords and prove some counting properties.

Let $K\subset M$
be a compact set. Let $x,y\in K$ and $\eta>0$. Let $0<T^-<T^+$.
The set of chords from $B(x,\eta)$ to  $ B(y,\eta)$  with lengths in $[T^-,T^+]$ is
 \[
 \mathcal{C}(x,y,\eta,T^-,T^+)=
\{ z\in B(x,\eta) :\ \phi_{[T^-,T^+]}(z)\cap B(y,\eta)\neq \emptyset\}\, .
\]
\label{p:C(c,y,eta,T-,T+)}

\begin{defi}\label{defi E-set}
Let $\delta>0$. A set $E$ is a $E(x,y,\eta, T^-,T^+,\delta)$-set if:
\begin{enumerate}
    \item\label{point i def E-set} $E\subset  \mathcal{C}(x,y,\eta,T^-,T^+)$;
    \item\label{point ii def E-set} $E$ is a $(\delta,T^-)$-separating set;
    \item\label{point iii def E-set} the set of chords $ \mathcal{C}(x,y,\eta,T^-,T^+)$ is contained in the union of dynamical balls $\bigcup_{z\in E}B(z,\delta,T^-)$.
\end{enumerate}
Denote by $\mathcal{N}_{\mathcal{C}}(x,y,\eta,T^-,T^+,\delta)$ the maximal cardinality of a $E(x,y,\eta,T^-,T^+,\delta)$-set.
\end{defi}

\begin{fact}\label{E set point un et deux alors trois}
Let $E$ be a set satisfying points \ref{point i def E-set} and \ref{point ii def E-set} of Definition \ref{defi E-set}.
Then, there exists a set $E'$ that is a $E(x, y, \eta, T^-, T^+, \delta)$-set and contains $E$.
\end{fact}
\begin{proof}
If the union of dynamical balls
$\cup_{z\in E}B(z,\delta,T^-)$ does not contain $\mathcal{C}(x,y,\eta,T^-,T^+)$, then we pick a point
\[
z'\in \mathcal{C}(x,y,\eta,T^-,T^+)\setminus \bigcup_{z\in E}\,B(z,\delta,T^-)\, .
\]
By construction, the set $\{z'\}\cup E$ is $(\delta,T^-)$-separating and contained in $\mathcal{C}(x,y,\eta,T^-,T^+)$, i.e., it satisfies points \ref{point i def E-set} and \ref{point ii def E-set} of Definition \ref{defi E-set}. If it is not  a  $E(x,y,\eta,T^-,T^+,\delta)$-set, we iterate the procedure. By compactness of the closure of $\mathcal{C}(x,y,\eta,T^-,T^+)$, this procedure will stop after a finite number of iterations.  At the end, we obtain a set $E'$ that contains $E$ and which is a $E(x,y,\eta,T^-,T^+,\delta)$-set.
\end{proof}

\begin{fact}\label{fact about monotonicity N}
Let $x,y\in K$, $0<\delta$ and $0<T^-<T^+$. The map
\[
\eta\in (0,+\infty)\mapsto \mathcal{C}(x,y,\eta,T^-,T^+)\]
is non-decreasing for the inclusion.
Moreover, the map
\[
\eta\in (0,+\infty)\mapsto \mathcal{N}_{\mathcal{C}}(x,y,\eta,T^-,T^+,\delta)\in\mathbb{N}
\]
is non-decreasing.
\end{fact}
\begin{proof} The first assertion is a direct consequence of the definition of chords. For the second assertion, fix $\eta_1<\eta_2$, use the first assertion and apply Fact~\ref{E set point un et deux alors trois} to a $E(x,y,\eta_1,T^-,T^+,\delta)$-set of maximal cardinality $E$, to build a  $E(x,y,\eta_2,T^-,T^+,\delta)$-set containing $E$.   We deduce that
\[\# E = \mathcal{N}_{\mathcal{C}}(x,y,\eta_1,T^-,T^+,\delta)\leq \mathcal{N}_{\mathcal{C}}(x,y,\eta_2,T^-,T^+,\delta)\, .
\]
\end{proof}

\begin{fact}\label{fact_about_monotonicity_eta}
Let $x,y\in K$, $\eta>0$ and $0<T^-<T^+$. Then, the map
\[
\delta\in (0,+\infty)\mapsto \mathcal{N}_{\mathcal{C}}(x,y,\eta,T^-,T^+,\delta)\in\mathbb{N}
\]
is non-increasing.
\end{fact}

\begin{proof}
Let $\delta_1\leq\delta_2$
First note that if a set is $(\delta_2,T^-)$-separating then it is also $(\delta_1,T^-)$-separating.
Let $E_2$ be an $E(x,y,\eta,T^-,T^+,\delta_2)$-set of maximal cardinality.
Then $E_2$ satisfies the first two points in the definition of $E(x,y,\eta,T^-,T^+,\delta_1)$-set.
Using Fact~\ref{E set point un et deux alors trois}, we obtain
\[
\#E_2 = \mathcal{N}_\mathcal{C}(x,y,\eta,T^-,T^+,\delta_2) \leq \mathcal{N}_\mathcal{C}(x,y,\eta,T^-,T^+,\delta_1)
\]
as required.
\end{proof}

Up to changing the parameter involved in the definition of the counting of chords, we show in Proposition~\ref{prop_chagement_points_cordes} that such a number is uniform with respect to the points $x,y\in K$.

\begin{prop}[Chord counting does not depend on the endpoints of the chords]\label{prop_chagement_points_cordes}
Let $K\subset M$ be a compact subset.
For any $\eta_1$ and $\delta$ such that $0<\eta_1< \frac{\delta}{2}$, there exist $\sigma>0$ and $0<\tilde \eta_0<\eta_1/2$ such that for all $x_0$, $y_0$, $x_1$, $y_1$ in $K$, $0< T_0^-<T_0^+$, $0<\eta_0\le \tilde \eta_0$ and, $S\ge \sigma$, we have
\[
\mathcal{N}_{\mathcal{C}}(x_0,y_0, \eta_0, T_0^-,T_0^+,\delta)\leq \mathcal{N}_{\mathcal{C}}(x_1,y_1,\eta_1,T_0^++S, T_0^++S+2\tau_K,\delta/2)\, .
\]
In particular, if $T_0^++\sigma \leq T_1^-<T_1^+$ and $T_1^+-T_1^-\geq 2\tau_K$, we have
\[
\mathcal{N}_{\mathcal{C}}(x_0,y_0, \eta_0, T_0^-,T_0^+,\delta)\leq\mathcal{N}_{\mathcal{C}}(x_1,y_1,\eta_1,T_1^-,T_1^+,\delta/2).
\]
\end{prop}

\begin{proof}
Figure~\ref{fig:comparaison_extremites_cordes} summarizes the proof.
The first naive idea to prove the proposition is the following. By transitivity property~\ref{transitivite prop}, we find  arcs of length $S/2$ respectively from $B(x_1,\eta_0)$ to $B(x_0,\eta_0)$ and from $B(y_0,\eta_0)$ to $B(y_1,\eta_0)$. The finite exact shadowing property~\ref{weak shadowing property} allows to  concatenate every chord of length in $[T_0^-,T_0^+]$ from  $B(x_0,\eta_0)$ to  $B(y_0,\eta_0)$ with these arcs before and after it, to obtain a chord from $B(x_1,\eta_1)$ to $B(y_1,\eta_1)$. The resulting chord has length in $[T_0^- +S-2\tau_K, T_0^++S+2\tau_K]$ and the uncertainty on its length is higher than desired.

The proof is close to this naive idea, but we choose first an arc  from $B(x_1,\eta_0)$ to $B(x_0,\eta_0)$, with length $S_1\simeq S/2\pm\tau_K$. Second, we consider an arbitrary chord of length $\ell\in[T_0^-,T_0^+]$. Third, by uniform transitivity, we choose an arc from $B(y_0,\eta_0)$ to $B(y_1,\eta_0)$ with length $S_2\pm \tau_K$, where $S_2$ is chosen so that
\[
S_1+\ell+S_2=T_0^++S+\tau_K
\]
so that, after concatenation, the resulting chord has length in $[T_0^++S, T_0^++S+2\tau_K]$.

\begin{figure}
    \centering
    \includegraphics{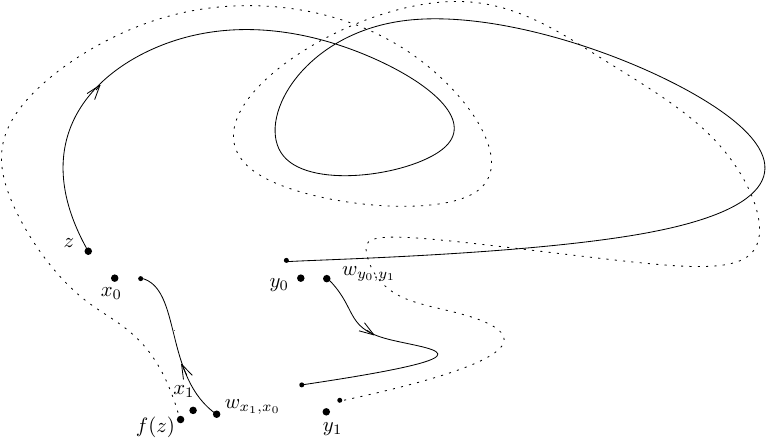}
    \caption{Chord counting does not depend on the endpoints of the chords}
    \label{fig:comparaison_extremites_cordes}
\end{figure}

{\bf Step 1.} Choice of parameters.
Let $\tilde\eta_0=\eta/2$, with $\eta$ the constant given by the finite exact shadowing property \ref{weak shadowing property} applied with $K$, $N=3$ and $\delta=\eta_1/2$.
Let $\sigma_0$ be the constant given by the uniform transitivity property applied with parameters $K$ and $\tilde\eta_0$.
Let $\sigma = 2\max(\sigma_0,\tau_K)$ and $S\geq\sigma$.
Let $0<\eta_0\le \tilde \eta_0$.

{\bf Step 2.} First use of transitivity.  Transitivity property~\ref{transitivite prop} gives us a point $w_{x_1,x_0}\in B(x_1,\tilde\eta_0)$ and $\ell(w_{x_1,x_0})\in [S/2-\tau_K,S/2+\tau_K]$ such that $\phi_{\ell(w_{x_1,x_0})}(w_{x_1,x_0})\in B(x_0,\tilde\eta_0)$.

{\bf Step 3.} Choice of a chord.
Let $E$ be a $E(x_0,y_0,\eta_0,T_0^-,T_0^+,\delta)$-set of maximal cardinality. For every $z\in E$, denote by $\ell(z)\in [T_0^-,T_0^+]$ a time such that $\phi_{\ell(z)}(z)\in B(y_0,\eta_0)$.

{\bf Step 4.} Second use of transitivity.  Lemma~\ref{transitivite prop} applied to $K$, $\tilde\eta_0$, $y_0$, $y_1$ and $S_2=S/2+(T_0^+-\ell(z))+(S/2+\tau_K-\ell(w_{x_1,x_0}))\geq S/2\geq \sigma_0$ gives a point $w_{y_0,y_1}$ and a length $\ell(w_{y_0,y_1})\in [S_2-\tau_K,S_2+\tau_K]$ such that
$w_{y_0,y_1}\in B(y_0,\tilde\eta_0)$ and $\phi_{\ell(w_{y_0,y_1})}(w_{y_0,y_1})\in B(y_1,\tilde\eta_0)$.
By construction, we have
\[
T_0^++S\leq \ell(z)+\ell(w_{x_1,x_0})+\ell(w_{y_0,y_1})\leq T_0^++S+2\tau_K.
\]

{\bf Step 5.} Concatenation.
By the choice of parameters in Step 1,  we have $d(\phi_{\ell(w_{x_1,x_0})}(w_{x_1,x_0}),z)<2\tilde\eta_0$ and $d(\phi_{\ell(z)}(z),w_{y_0,y_1})<2\tilde\eta_0$,
By the finite exact shadowing property~\ref{weak shadowing property} there exists a point $f(z)$ such that
\begin{itemize}
    \item for every $s\in[0,\ell(w_{x_1,x_0})]$,   $d(\phi_s(f(z)),\phi_s(w_{x_1,x_0}))<\eta_1/2$;
    \item for every $s\in[0,\ell(z)]$,    $d(\phi_{\ell(w_{x_1,x_0})+s}(f(z)),\phi_s(z))<\eta_1/2$;
    \item for every $s\in[0,\ell(w_{y_0,y_1})]$,  $d(\phi_{\ell(w_{x_1,x_0})+\ell(z)+s}(f(z)),\phi_s(w_{y_0,y_1}))<\eta_1/2$.
\end{itemize}
By construction,
\[
d(f(z),x_1)\leq d(f(z),w_{x_1,x_0})+d(w_{x_1,x_0},x_1)<\eta_1/2+\tilde\eta_0\leq \eta_1
\]
and
\begin{align*}
d(\phi_{\ell(w_{x_1,x_0})+\ell(z)+\ell(w_{y_0,y_1})}(f(z)),y_1)\leq&  d(\phi_{\ell(w_{x_1,x_0})+\ell(z)+\ell(w_{y_0,y_1})}(f(z)),\phi_{\ell(w_{y_0,y_1})}(w_{y_0,y_1}))\\&+d(\phi_{\ell(w_{y_0,y_1})}(w_{y_0,y_1}),y_1)\\
<&
\eta_1/2+\tilde\eta_0\leq \eta_1.
\end{align*}
Thus, the point $f(z)$ belongs to $\mathcal C(x_1,y_1,\eta_1,T_0^++S,T_0^++S+2\tau_K)$.
Therefore, we have just defined a map $f : E \to  \mathcal C(x_1,y_1,\eta_1,T_0^++S,T_0^++S+2\tau_K)$

{\bf Step 6.} Separation. Recall that  $E$ is a $E(x_0,y_0,\eta_0,T_0^-,T_0^+,\delta)$-set of maximal cardinality.
For every $z\in E$, we built in the preceding steps  a point $f(z)\in \mathcal C(x_1, y_1, \eta_1, T_0^+ + S, T_0^+ +S+2\tau_K)$. We prove now that $E'=f(E)$ is $(\frac{\delta}{2}, T_0^+ + S)$-separating and that $f$ is injective.

Let  $z_1,z_2\in E$ be such that $z_1\neq z_2$. As $E$ is $(\delta,T_0^-)$-separating, there exists $u\in[0,T_0^-]$ such that $d(\phi_u(z_1),\phi_u(z_2))\geq \delta$. Therefore, we have
\begin{align*}
d(\phi_{u+\ell(z_{x_1,x_0})}(f(z_1)),\phi_{u+\ell(z_{x_1,x_0})}(f(z_2)))\geq \,& \, d(\phi_u(z_1),\phi_u(z_2))-d(\phi_{u+\ell(z_{x_1,x_0})}(f(z_1)),\phi_u(z_1))\\
&\,\,-d(\phi_{u+\ell(z_{x_1,x_0})}(f(z_2)),\phi_u(z_2))
\\  >&\,\,\delta- \eta_1\geq \dfrac{\delta}{2}
\end{align*}
that is, since $T_0^-+\frac S 2 +\tau_K\leq T_0^++S$, $E'$ is a $(\frac{\delta}{2}, T_0^++S)$-separating set and $f$ is injective.
Thus $\#E\leq \#E'$.

{\bf Conclusion.} By Lemma \ref{E set point un et deux alors trois}, we obtain
\[
\mathcal{N}_{\mathcal{C}}(x_0,y_0,\eta_0,T_0^-,T_0^+,\delta)\leq \mathcal{N}_{\mathcal{C}}\left(x_1,y_1,\eta_1,T_0^++S,T_0^++S+2\tau_{\min}(K),\frac{\delta}{2}\right)\, ,
\]
as required.
\end{proof}


\subsection{Comparing chords and periodic orbits}\label{sec:chords-periodic}

We can now compare the number of chords with the number of periodic orbits of approximately the same length.
Note that the admissible lengths of the chords/periodic orbits are intervals of the same length $\tau$ but shifted by $\sigma$.
This is not critical to compare chord entropy and Gurevic entropy but will be crucial for upcoming statements.

\begin{prop}\label{chords et orbite periodique}
Let $K \subset M$  be a compact subset of $M$  with nonempty interior.
Let $\tau>2\tau_K$.
Fix $\delta>0$.
There exist constants $D=D(\delta)>0$, $T_{\rm min}>0$ and $\sigma>0$ such that for all $x,y\in K$, $T\geq T_\mathrm{min}$ and   $0<\eta<1$, we have
\[
\mathcal{N}_{\mathcal{C}}(x,y,\eta,T,T+\tau,\delta)\,\,\leq\,\, D\times(T+\sigma+\tau)\times \,\#\mathcal{P}_{K}(T+\sigma, T+\sigma+\tau).
\]
\end{prop}

\begin{proof} The idea of the  proof is quite simple. We start with a chord from $x$ to $y$ whose   length is in $[T,T+\tau]$ ,  we use transitivity to build an almost-closed pseudo orbit following the chord from $x$ to $y$ and coming back to $x$. The closing lemma~\ref{petal separe} allows us to close it into a closed orbit intersecting $K$. Then, we control the default of injectivity of the construction.

{\bf Step 1.} Choice of parameters. \\
Let $T_{\mathrm{min}}>0$ and $\sigma'>0$ be the constants given by Lemma \ref{petal separe} applied to $K\subset \overline{B(K,1)}$, $\delta/3$, $\nu=\min(\tau/2-\tau_K,1)$ and $N=1$.
Without loss of generality, one may assume $T_\mathrm{min}\ge 1$.
Let $\sigma = \sigma'-\tau_K-\nu$. One may assume $\sigma>0$.
Let $T\geq T_\mathrm{min}$.
Let $E$ be a $E(x,y,\eta,T, T+\tau, \delta)$-set of maximal cardinality.

{\bf Step 2.} Construction of a map from chords to periodic orbits.\\
We now define a map
\[
f :  E\to \mathcal{P}_K(T+\sigma, T+\sigma+\tau).
\]
Let $w\in E$.
Lemma ~\ref{petal separe} gives us a periodic orbit $\gamma=f(w)$ with period in \[
[T+\sigma'-\tau_K-\nu,T+\sigma'+\tau_K+\nu]\subset[T+\sigma,T+\sigma+\tau]
\]
that intersects $K$ and $\frac{\delta}{3}$-shadows the orbit of $w$ from $B(x,\eta)$ to $B(y,\eta)$.
More precisely, for $w\in E$, let $\ell(w)\in [T,T+\tau]$ a time such that $\phi_{\ell(w)}(w)\in B(y,\eta)$. Then, by construction, there exists an origin $s_0$ for the periodic orbit $\gamma$ such that for all $s\in [0,\ell(w)]$ we have
\begin{equation}\label{eq 1 first chords}
d(\gamma(s+s_0),\phi_s(w))<\dfrac{\delta}{3}.
\end{equation}

{\bf Step 3.} Control of the cardinality  of the preimages of $f$.\\
Let $\gamma\in\mathcal{P}_K(T+\sigma,T+\sigma+\tau)$ and let $w_1,w_2\in E$ be such that $f(w_1)=f(w_2)=\gamma$.
The construction $\gamma = f(w_1)$ gives us an origin $s_1$ and the construction $\gamma=f(w_2)$ an origin $s_2$.

For every $s\in[0,T]$ by \eqref{eq 1 first chords} and by property \eqref{eqn:minoration} we obtain
\begin{align*}
d(\phi_s(w_1),\phi_s(w_2))&\leq d(\phi_s(w_1),\gamma(s+s_1))+d(\gamma(s+s_1),\gamma(s+s_2))+d(\phi_s(w_2),\gamma(s+s_2))\\
&\leq \dfrac{\delta}{3}+\dfrac{\delta}{3}+b\vert s_2-s_1\vert \, .
\end{align*}
If $b\vert s_2-s_1\vert <\frac{\delta}{3}$, then $w_2\in B(w_1,\delta,T)$ and, as $E$ is $(\delta,T)$-separating, we obtain $w_1=w_2$. Therefore, if $w_1\neq w_2$ then $\vert s_2-s_1\vert \ge \delta/3b$ and
\begin{equation*}
\# f^{-1}(\gamma)\leq \left\lceil \dfrac{3b}{\delta}(T+\sigma+\tau)\right\rceil.
\end{equation*}
As $E$ is a $E(x,y,\eta,T,T+\tau,\delta)$-set of maximal cardinality, we conclude that
\[
\mathcal{N}_{\mathcal{C}}(x,y,\eta,T,T+\tau,\delta)\leq\left\lceil \frac{3b}{\delta}\times(T+\sigma+\tau) \right\rceil\times \,\#\mathcal{P}_K(T+\sigma,T+\sigma+\tau).
\]
This is the desired result, with $D=\left\lceil\frac{3b}{\delta}\right\rceil +1$.
\end{proof}

\begin{prop}\label{orbite periodique et chords} 
Let $K\subset M$ be a compact subset. Fix some $\tau>2\tau_K$. There exists $\epsilon_0>0$ and $D$ such that for all $0<\delta< \epsilon_0/4$ and $0<\eta<\epsilon_0/2$, there exists $\sigma>0$ such that for all $x,y\in K$ and $T>0$, we have
\[
\#\mathcal{P}_K(T, T+\tau)\leq D \, \mathcal{N}_{\mathcal{C}}(x,y,\eta,T+\sigma, T+\sigma+\tau, \delta)\, .
\]
\end{prop}

\begin{proof}
The strategy of the proof is, once again, to use transitivity property~\ref{transitivite prop bis} from $x$ to a periodic orbit $\gamma$ and from $\gamma$ to $y$, and then the finite exact shadowing property~\ref{weak shadowing property} to get a chord from $x$ to $y$ that starts close from $x$, goes to $\gamma$, follows it, and then finishes close to $y$.
One difficulty, as usual, is to control the (lack of) injectivity of the construction. See Figure~\ref{fig:comparaison_cordes_OP_2}.

\begin{figure}
    \centering
    \includegraphics{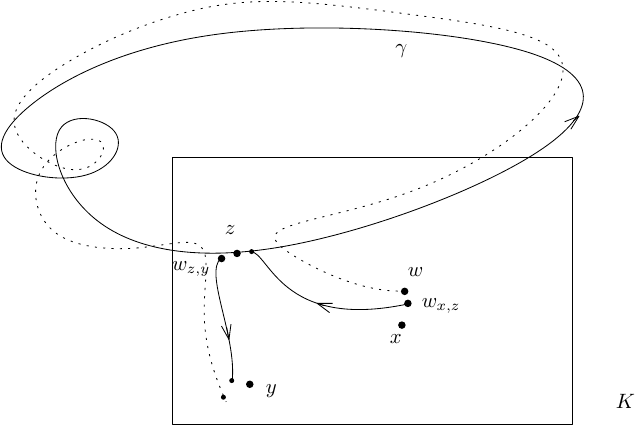}
    \caption{How to create a chord from a periodic orbit}
    \label{fig:comparaison_cordes_OP_2}
\end{figure}

{\bf Step 1.} Choice of appropriate parameters.
At the end of the construction, we will use the separation property \ref{lemma on same po} with $\nu=1$ and $\tau_1=\tau+2\tau_K$. This lemma gives us some constants $\tau_0$ and $\epsilon_0$.
Fix $0<\delta<\epsilon_0/4$ and $0<\eta<\epsilon_0/2$.
We will apply the finite exact shadowing property~\ref{weak shadowing property} on the $1$-neighbourhood $K'=\overline{B(K,1)}$ of $K$, with $N=3$ pieces of orbits that we want to glue to get a shadowing orbit at distance at most $\eta/2$ of the initial pieces.  Property~\ref{weak shadowing property} gives a constant $0<\zeta<\eta/2$ associated with $K'$, $N=3$ and $\eta/2$.
Now, transitivity property~\ref{transitivite prop bis} on $K'$ with precision $\zeta$ gives us a constant $\sigma_0>0$.
Observe that by definition, $\tau_{K'}\le \tau_K$.
Set $\sigma =2\sigma_0+\tau$.

{\bf Step 2.} A map from periodic orbits to chords.
Let $E$ be a $E(x,y,\eta ,T+\sigma, T+\sigma+\tau, \delta)$-set of maximal cardinality.
We define a map
\[
f :  \mathcal{P}_K(T,T+\tau)\to E\,,
\]
with controlled lack of injectivity as follows.

{\bf Step 2a.} Transitivity.
Start with a periodic orbit $\gamma\in \mathcal{P}_K(T,T+\tau)$.
Choose a starting point $z\in \gamma\cap K$ and a parametrization of $\gamma$ such that $z=\gamma(0)$.
By transitivity property~\ref{transitivite prop bis}, we find $w_{x,z}\in B(x, \zeta)$ and $\ell(w_{x,z})\in  [\sigma_0-\tau_{K'},\sigma_0+\tau_{K'}]\subset [\sigma_0-\tau_K,\sigma_0+\tau_K]$ such that $\phi_{\ell(w_{x,z})}(w_{x,z})\in B(z,\zeta)$.
In particular,
\[
\ell(w_{x,z})+\ell(\gamma)\in [\sigma_0+T-\tau_K,\sigma_0+T+\tau+\tau_K].
\]

{\bf Step 2b.} Transitivity with well chosen length.
For every $S\ge\sigma_0$, transitivity allows to find a chord from $B(z,\zeta)$ to $B(y,\zeta)$ with length in $[S-\tau_K,S+\tau_K]$.
Choose \[S=\sigma_0+(\sigma_0+\tau_K-\ell(w_{x,z}))+(T+\tau-\ell(\gamma))\geq \sigma_0\]
and let $w_{z,y}$ denote the initial point of the associated chord and $\ell(w_{z,y})$ its length.
By construction,
\[
\ell(z_{x,z})+\ell(\gamma)+
\ell(z_{z,y})\in [\sigma_0+(\sigma_0+\tau_K)+(T+\tau)-\tau_K,\sigma_0+(\sigma_0+\tau_K )+(T+\tau)+\tau_K]\,.
\]
Therefore (recall that $\sigma = 2\sigma_0+\tau$ and $\tau>2\tau_K$),
\[\ell(z_{x,z})+\ell(\gamma)+
\ell(z_{z,y})\in[T+\sigma,T+\sigma+2\tau_K]\subset [T+\sigma,T+\sigma+\tau].
\]

{\bf Step 2c.} Finite exact shadowing.
Recall that $w_{x,z}\in B(x,\zeta)$, $\phi_{\ell(w_{z,y})}(w_{z,y})\in B(y,\zeta)$ and $\zeta\leq\eta/2$.
As $\phi_{\ell(w_{x,z})}(w_{x,z})\in B(\gamma(0),\zeta)=B(z,\zeta)$ and  $w_{z,y}\in B(z,\zeta)=B(\gamma(\ell(\gamma)),\zeta)$, the finite exact shadowing property~\ref{weak shadowing property} gives a point $w\in  B(w_{x,z},\eta/2)\subset B(x,\eta)$ such that
\begin{itemize}
    \item for  $0\le s\le\ell(w_{x,z})$,  $d(\varphi_s(w), \varphi_s(w_{x,z}))<\frac{\eta}{2}$;
    \item for $0\le s\le\ell(\gamma)$,
    $d(\varphi_{\ell(w_{x,z})+s}(w),\gamma(s))<\frac{\eta}{2}$;
    \item for  $0\le s\le \ell(w_{z,y}) $,
    $d(\varphi_{\ell(w_{x,z})+\ell(\gamma)+s}(w),\varphi_s(w_{z,y}))<\frac{\eta}{2}$.
\end{itemize}
In particular, $\varphi_{\ell(w_{x,z})+\ell(\gamma)+\ell(w_{z,y})}(w)\in B(y,\eta)$ and $w \in \mathcal{C}(x,y,\eta,T+\sigma, T+\sigma+\tau)$.

{\bf Step 2d.} Construction of $f$. Since $E$ is a $E(x,y,\eta,T+\sigma,T+\sigma+\tau,\delta)$-set, there exists a point $p\in E$ such that $w\in B(p,\delta, T+\sigma)$. Set $f(\gamma)=p$. If $w$ belongs to more than one dynamical ball, just enumerate all the points in $E$ and choose the first one.

{\bf Step 3.} Bound the cardinality of the preimages of $f$.
Consider $\gamma_1,\gamma_2\in\mathcal{P}_K(T,T+\tau)$ such that $f(\gamma_1)=f(\gamma_2)=p\in E$.
Divide the interval $[T, T+ \tau]$ into intervals of length $\tau_0$, where $\tau_0$ is given by Lemma~\ref{lemma on same po} as explained at the beginning of the proof. We now prove that if $\gamma_1$ and $\gamma_2$ satisfy $f(\gamma_1)=f(\gamma_2)$ and $|\ell(\gamma_1)-\ell(\gamma_2)|\le \tau_0$, then $\gamma_1=\gamma_2$. This will imply the desired result with $D=\ceil{\tau/\tau_0}$.

Assume from now that $0\leq \left|\ell(\gamma_2)-\ell(\gamma_1)\right|\leq \tau_0$, We want to show that for $s\in [0,T-2\tau_K]$, we have $d(\gamma_1(s),\gamma_2(s))\le \epsilon_0$, and then use the separation property~\ref{lemma on same po}.

The construction of $f(\gamma_1)$ (resp. $f(\gamma_2)$) involves chords with initial points $w_{x,z_1}$ and $w_{z_1,y}$ (resp. $w_{x,z_2}$ and $w_{z_2,y}$), and produces a point $w_1$ (resp. $w_2$) in $ \mathcal{C}(x,y,\eta,T+\sigma,T+\sigma+\tau)$ such that $w_1\in B(p,\delta, T+\sigma)$ (resp. $w_2\in B(p,\delta, T+\sigma)$).
This proves  $w_1\in B(w_2,2\delta, T+\sigma)$.
Without loss of generality, we may assume that $\ell(w_{x,z_1})\leq \ell(w_{x,z_2})$.

By construction, for every $\ell(w_{x,z_1})\le s\le \ell(\gamma_1)+\ell(w_{x,z_1})$, $\phi_s(w_1)$ is $\eta/2$-close to $\gamma_1$.
Similarly, for every $\ell(w_{x,z_2})\le s\le \ell(\gamma_2)+\ell(w_{x,z_2})$, $\phi_s(w_2)$ is $\eta/2$-close to $\gamma_2$.
More precisely, for all $s\in [0,\ell(\gamma_1)]$,
\[d(\gamma_1(s), \phi_{s+\ell(w_{x,z_1})}(w_1))\leq \frac{\eta}{2}.\]
Therefore, for all $s\in[-(\ell(w_{x,z_2})-\ell(w_{x,z_1})), \ell(\gamma_1)-(\ell(w_{x,z_2})-\ell(w_{x,z_1}))]$
\[d(\gamma_1(s+(\ell(w_{x,z_2})-\ell(w_{x,z_1})), \phi_{s+\ell(w_{x,z_2})}(w_1))\leq \frac{\eta}{2}.\]
Symmetrically, for all  $s\in [0,\ell(\gamma_2)]$,
\[d(\gamma_2(s), \phi_{s+\ell(w_{x,z_2})}(w_2))\leq \frac{\eta}{2}.\]
Recall that $0\le \ell(w_{x,z_2})-\ell(w_{x,z_1}) \le 2\tau_K$, $T\le\ell(\gamma_1)\le T+\tau$ and $T\le \ell(\gamma_2)\le T+\tau$.
Therefore, for all $s\in[0, T-2\tau_K]$, we have
\begin{align*}
d(\gamma_1(s+\ell(w_{x,z_2})-\ell(w_{x,z_1})),\gamma_2(s))\leq& d(\gamma_1(s+\ell(w_{x,z_2})-\ell(w_{x,z_1})),\varphi_{s+\ell(w_{x,z_2})}(w_1))\\
&+d(\varphi_{s+\ell(w_{x,z_2})}(w_1),\varphi_{s+\ell(w_{x,z_2})}(w_2))\\
&+d(\varphi_{s+\ell(w_{x,z_2})}(w_2),\gamma_2(s))\\
<&\frac{\eta}{2}+2\delta +\frac{\eta}{2}<\epsilon_0.
\end{align*}
By Lemma \ref{lemma on same po}, we deduce that $\gamma_1=\gamma_2$. Thus, the cardinality of the preimage  of any point by $f$ is bounded by $D=\ceil{\tau/\tau_0}$ and
\[
\#\mathcal{P}_K(T,T+\tau)\leq D\, \# E=D\, \mathcal{N}_{\mathcal{C}}(x,y,\eta,T+\sigma,T+\sigma+\tau,\delta).
\]
as required.
\end{proof}

We end this section with a technical adaptation of Proposition~\ref{chords et orbite periodique} which will be useful to compare entropies at infinity in the proof of Theorem~\ref{difficult-technical}.
The idea is to compare chords and orbits contained in specified compact subsets of $M$.
This may be skipped on first reading.

Let $K$ be a compact subset of $M$.
Let $x,y\in K$ and let $K'$ be a compact set containing $K$.
Analogously to Definition~\ref{defi E-set}, we consider $E_{K'}(x,y,\eta,T^-, T^+,\delta)$-sets which are $E(x,y,\eta,T^-, T^+,\delta)$-sets made up of chords from $B(x,\eta)$ to $B(y,\eta)$ \emph{contained in $K'$}.
We will denote by $\mathcal N_{\mathcal C, K'}(x,y,\eta,T^-,T^+,\delta)$ the maximal cardinality of such sets.

\begin{prop}\label{cordes et orbite periodique version compacte}
Let $K \subset M$  be a compact subset of $M$  with nonempty interior.
Let $\tau>2\tau_K$.
Fix $0<\delta<3$.
There exist constants $R_\mathrm{min}>0$, $D=D(\delta)>0$, $T_{\rm min}>0$ and $\sigma>0$ such that for all $x,y\in K$, $R\ge R_\mathrm{min}$, $T\geq T_\mathrm{min}$ and all $0<\eta<1$, we have
\[
\mathcal{N}_{\mathcal{C},K_R}(x,y,\eta,T,T+\tau,\delta)\,\,\leq\,\, D\times(T+\sigma+\tau)\times \,\#\{\gamma\in \mathcal{P}_{K}(T+\sigma, T+\sigma+\tau), \gamma\subset K_{R+1}\}
\]
where $K_R=\overline{B(K,R)}$.
\end{prop}

\begin{proof}
The proof goes exactly as the proof of Proposition~\ref{chords et orbite periodique}.
One just have to choose $R_\mathrm{min}$ big enough so that $K_{R_{\min}}$ contains all chords connecting, by uniform transitivity, any couple of balls, among a finite family covering $K$, whose radius depends only on $\delta$ and $K$. 
As $\delta/3<1$, the periodic orbit $\gamma$ is then contained in $K_{R+1}$ if the original chord is contained in $K_R$.
\end{proof}

\subsection{Chord entropy}

In this section, we define a notion of entropy that counts chords with increasing length. We prove that for $H$-flows, it coincides with Gurevic entropy. This chord entropy will be easier to use than the standard Gurevic entropy.

Fix $x,y\in M$ and $\sigma>0$, $\eta>0$, $\delta>0$.
Let
\begin{equation}\label{def:hC_x_y_eta_delta}
h_{\mathcal C}(x,y,\eta,\delta)=
\limsup_{T\to \infty}\frac{1}{T}\log \mathcal N_{\mathcal C}(x,y,\eta, T,T+\sigma,\delta).
\end{equation}
Recall that $\mathcal N_{\mathcal C}(x,y,\eta, T,T+\sigma,\delta)$ counts chords and is defined in Definition~\ref{defi E-set}.
This does not depend on $\sigma$, as proved in the following lemma.

\begin{lemm}\label{lem:constant}
The quantity $h_{\mathcal C}(x,y,\eta,\delta)$ is non-decreasing in $\eta>0$, non-increasing in $\delta>0$ and does not depend on $C$.
\end{lemm}
\begin{proof} The first assertion comes from Fact~\ref{fact about monotonicity N}
and the second assertion from Fact~\ref{fact_about_monotonicity_eta}.

We now prove the last assertion. Choose two constants $0<\sigma_1<\sigma_2<\infty$.
Let $n$ be the smallest integer such that $\sigma_2\le n\sigma_1$.
For all choices of $x$, $y$, $\delta$, $\eta$ and $T>0$, we have
\[
\mathcal N_{\mathcal C} (x,y,\eta,T ,T+\sigma_1,\delta)
\le
\mathcal N_{\mathcal C} (x,y,\eta,T ,T+\sigma_2,\delta)\le \sum_{j=0}^{n-1}\mathcal N_{\mathcal C} (x,y,\eta,T+j\sigma_1 ,T+(j+1)\sigma_1,\delta).
\]
Indeed,
\[
\mathcal C(x,y,\eta,T,T+\sigma_2)
\subset
\bigcup_{j=0}^{n-1}\mathcal C(x,y,\eta,T + j\sigma_1, T+(j+1)\sigma_1)
\]
and if $E$ is a $E(x,y,\eta,T,T+\sigma_2,\delta)$-set of maximal cardinality, then $E_j = E\cap \mathcal C(x,y,\eta,T + j\sigma_1, T+(j+1)\sigma_1)$ is $(\delta, T)$-separating and therefore $(\delta,T+j\sigma_1)$-separating. By Fact \ref{E set point un et deux alors trois}, we deduce that
$$
\# E_j\leq \mathcal{N}_{\mathcal C} (x,y,\eta,T+j\sigma_1 ,T+(j+1)\sigma_1,\delta)\, ,
$$
and thus, since $E\subset \bigcup_{j=0}^{n-1}E_j$,
\[
\mathcal N_{\mathcal C} (x,y,\eta,T ,T+\sigma_2,\delta)
=
\# E\leq \sum_{j=0}^{n-1}\# E_j\leq \sum_{j=0}^{n-1}\mathcal{N}_{\mathcal C} (x,y,\eta,T+j\sigma_1 ,T+(j+1)\sigma_1,\delta)\, .
\]
Therefore,
\[
\mathcal N_{\mathcal C} (x,y,\eta,T ,T+\sigma_1,\delta)\le \mathcal N_{\mathcal C} (x,y,\eta,T ,T+\sigma_2,\delta)\le n\max_{j=0\dots,n-1}\mathcal N_{\mathcal C} (x,y,\eta,T+j\sigma_1 ,T+(j+1)\sigma_1,\delta)\,.
\]
Yet
\begin{align*}
\limsup_{T\to\infty}\frac{\log \mathcal N_{\mathcal C} (x,y,\eta,T ,T+\sigma_1,\delta)}{T} &= \limsup_{T\to\infty} \max_{j=0\dots,n-1}\frac{\log \mathcal N_{\mathcal C} (x,y,\eta,T+j\sigma_1 ,T+(j+1)\sigma_1,\delta) }{T+j\sigma_1} \\
&= \limsup_{T\to\infty}\frac{1}{T}\log \max_{j=0\dots,n-1} \mathcal N_{\mathcal C} (x,y,\eta,T+j\sigma_1 ,T+(j+1)\sigma_1,\delta)
\\
&= \limsup_{T\to\infty}\frac{1}{T}\log n \max_{j=0\dots,n-1} \mathcal N_{\mathcal C} (x,y,\eta,T+j\sigma_1 ,T+(j+1)\sigma_1,\delta).
\end{align*}
Therefore
\[\limsup_{T\to\infty}\frac{1}{T}\log \mathcal N_{\mathcal C} (x,y,\eta,T ,T+\sigma_1,\delta) = \limsup_{T\to\infty}\frac{1}{T}\log \mathcal N_{\mathcal C} (x,y,\eta,T ,T+\sigma_2,\delta).\]
\end{proof}

\begin{lemm}\label{lemme:changement_points_entropie_cordes}
Let $K\subset M$ be a compact subset.
Let $\delta>0$ and $\eta_1>0$ such that $\eta_1<\delta/2$.
Then there exists $\tilde \eta_0$ such that for all $x_0$, $y_0$, $x_1$, $y_1$ in $K$ and for all $0<\eta_0\le\tilde\eta_0$,
\[
h_{\mathcal C}(x_0,y_0,\eta_0,\delta)\le h_{\mathcal C}(x_1,y_1,\eta_1,\frac{\delta}{2}).
\]
In particular
\[
\lim_{\delta\to 0}\lim_{\eta\to 0} h_{\mathcal C}(x_0,y_0,\eta,\delta)= \lim_{\delta\to 0}\lim_{\eta\to 0} h_{\mathcal C}(x_1,y_1,\eta,\delta).
\]
\end{lemm}

\begin{proof}
By Proposition~\ref{prop_chagement_points_cordes}, there exists $\sigma$ and $\tilde\eta_0$ such that, for all $0<\eta_0\le\tilde\eta_0$ and all $T>0$
\[
\mathcal{N}_{\mathcal{C}}(x_0,y_0, \eta_0, T,T+2\tau_K,\delta)\leq\mathcal{N}_{\mathcal{C}}(x_1,y_1,\eta_1,T+\sigma,T+\sigma+2\tau_K,\delta/2).
\]
Therefore
\[
h_{\mathcal C}(x_0,y_0,\eta_0,\delta)\le h_{\mathcal C}(x_1,y_1,\eta_1,\frac{\delta}{2}).
\]
Considering the limit when $\eta_0\to 0$ and then  when $\eta_1\to 0$, we obtain
\[
\lim_{\eta\to 0} h_{\mathcal C}(x_0,y_0,\eta,\delta)\le \lim_{\eta\to 0} h_{\mathcal C}(x_1,y_1,\eta,\frac{\delta}{2}).
\]
We now consider the limit when  $\delta\to 0$ to obtain
\[
\lim_{\delta\to 0}\lim_{\eta\to 0} h_{\mathcal C}(x_0,y_0,\eta,\delta)\le \lim_{\delta\to 0}\lim_{\eta\to 0} h_{\mathcal C}(x_1,y_1,\eta,\delta).
\]
Therefore, by inverting the roles,
\[
\lim_{\delta\to 0}\lim_{\eta\to 0} h_{\mathcal C}(x_0,y_0,\eta,\delta)= \lim_{\delta\to 0}\lim_{\eta\to 0} h_{\mathcal C}(x_1,y_1,\eta,\delta).
\]
\end{proof}

\begin{defi}\label{def:chord-entropy}
We define the {\em chord entropy} as
\begin{equation*}
h_{\mathcal C} (\phi)= \lim_{\delta\to 0}\lim_{\eta\to 0}\limsup_{T\to \infty}\frac{1}{T}\log \mathcal N_{\mathcal C}(x,y,\eta, T,T+\sigma,\delta)\, .
\end{equation*}
\end{defi}

Observe that the chord entropy does not depend on the choice of $x$, $y$ and $\sigma$.

\begin{theo}\label{thm h cord=h gur}
Let $\varphi$ be a $H$-flow on $M$.
The Gurevic entropy coincides with the chord entropy:
\[
h_\mathrm{Gur}(\phi)=h_{\mathcal C}(\phi).
\]
Moreover, for every fixed compact set $K$ with nonempty interior, there exists $\alpha_0>0$ such that for all $x,y\in K$,  $\tau\ge 5\tau_K$, $0<\delta<\alpha_0/2$ and $0<\eta<\alpha_0/2$, the quantity
\[\frac{1}{T}\log \mathcal N_C(x,y,\eta, T,T+\tau,\delta)\]
converges towards $h_\mathrm{Gur}(\phi)$ when $T\to +\infty$.
Thus
\[
h_{\mathcal C}(\phi)=h_\mathrm{Gur}(\phi)=\lim_{T\to \infty}\frac{1}{T}\log \mathcal{N}_{\mathcal{C}}(x,y,\eta,T,T+\tau,\delta).
\]
\end{theo}
\begin{proof}
It follows immediately from Propositions~\ref{chords et orbite periodique} and \ref{orbite periodique et chords}, as we will see now.

Fix $\tau=5\tau_K$.
Fix $\epsilon_0$ as in Proposition~\ref{orbite periodique et chords}.
Let $\alpha_0 = \min(\epsilon_0/2,1)$.
Choose $K\subset M$ compact and $x,y\in K$.
Fix $0<\delta<\alpha_0/2$ and $0<\eta<\alpha_0/2$.
From Propositions~\ref{chords et orbite periodique} and \ref{orbite periodique et chords}, there exists $\sigma$, $\sigma'$, $D$ and $D'$ such that for all $T\gg 1$
\[
\frac{1}{D}\# \mathcal P_K(T,T+\tau)\le  \mathcal N_\mathcal{C} (x,y,\eta,T+\sigma, T+\sigma+\tau,\delta)\le D' (T+\sigma+\sigma'+\tau) \# \mathcal P_K(T+\sigma+\sigma',T+\sigma+\sigma'+\tau).
\]
As $h_\mathrm{Gur}$ is a true limit (Theorem~\ref{theo:Gurevic}), we obtain that the following limits exist and
\[ \lim_{T\to \infty}\frac{1}{T}\log \mathcal{P}_K(T,T+\tau)= \lim_{T\to \infty}\frac{1}{T}\log \mathcal{N}_{\mathcal{C}}(x,y,\eta,T,T+\tau,\delta)\]
Thus
\[
h_\mathrm{Gur}(\phi)=h_{\mathcal C}(\phi).
\]
\end{proof}

\subsection{Entropy at infinity through chords}\label{chords infinity}

This section is devoted to the notion of chord entropy at infinity, that we will later compare with $h_{\mathrm{Gur}}^\infty(\phi)$.
Fix a compact subset $K\subset M$, and two points $x,y\in\partial K$.
For $\eta>0$, we define the $\eta$-interior neighbourhood of $K$ as
\[
K_{-\eta}=K \setminus\cup_{x\in\partial K} B(x,\eta)\,.
\]
\label{p:eta_interior_compact_set}
We define a \emph{chord outside $K_{-\eta}$ from $B(x,\eta)$ to $B(y,\eta)$} as a path from a point of $B(x,\eta)$ to a point of $B(y,\eta)$ that does not intersect $K_{-\eta}$.
We now consider the set of chords outside $K_{-\eta}$ with controlled length. Define  $\mathcal{C}^{K^c}(x,y,\eta,T^-,T^+)$  as
\[
\{ z\in B(x,\eta), \exists \tau\in[T^-,T^+]\text{ such that }\varphi_{\tau}(z)\in B(y,\eta) \text{ and }\varphi_{[0,\tau]}(z)\cap K_{-\eta}=\emptyset\}.
\]
\label{p:cordes_exterieur_K}
Observe that some sets $\mathcal{C}^{K^c}(x,y,\eta,T^-,T^+)$ could be empty for all $T^-$ and $T^+$.
Following Definition~\ref{defi E-set}, we introduce the following notations.
\begin{defi}\label{defi E-set outside compact}
Let $\delta>0$.
A set $E$ is a $E^{K^c}(x,y,\eta, T^-,T^+,\delta)$-set if:
\begin{enumerate}
    \item\label{point i def E-set outside} $E\subset  \mathcal{C}^{K^c}(x,y,\eta,T^-,T^+)$;
    \item\label{point ii def E-set outside} $E$ is a $(\delta,T^-)$-separating set;
    \item\label{point iii def E-set outside} the set $ \mathcal{C}^{K^c}(x,y,\eta,T^-,T^+)$ is contained in the union of dynamical balls $\bigcup_{z\in E}B(z,\delta,T^-)$.
\end{enumerate}
We define the {\em number} $\mathcal{N}^{K^c}_{\mathcal{C}}(x,y,\eta,T^-,T^+,\delta)$ of chords from $x$ to $y$ outside $K_{-\eta}$ with length in $[T^-,T^+]$ as the maximal cardinality of a $E^{K^c}(x,y,\eta,T^-,T^+,\delta)$-set.
\end{defi}

The proof of the following fact is similar to the proof of Fact~\ref{E set point un et deux alors trois}.
\begin{fact}\label{E set point un et deux alors trois outside compact}
Let $E$ be a set satisfying points~\ref{point i def E-set outside} and \ref{point ii def E-set outside} of Definition~\ref{defi E-set outside compact}.
Then there exists a set $E'$ containing $E$ which is a $E^{K^c}(x,y,\eta,T^-, T^+,\delta)$-set.
\end{fact}
The following fact is analogous to Fact~\ref{fact about monotonicity N} and we omit its proof.
\begin{fact}\label{fact about monotonicity N outside compact}
Let $K$ be a compact set with nonempty interior and $x,y\in\partial K$. Let $\delta>0$ and $0<T^-<T^+$. The map
\[
\eta\in(0,+\infty)\mapsto \mathcal{C}^{K^c}(x,y,\eta,T^-,T^+,\delta)
\]
is non-decreasing for the inclusion. The map
\[
\eta\in(0,+\infty)\mapsto \mathcal{N}^{K^c}_{\mathcal{C}}(x,y,\eta,T^-,T^+,\delta)
\]
is non-decreasing.
\end{fact}
The following fact is similar to Fact~\ref{fact_about_monotonicity_eta}.
\begin{fact}\label{fact about monotonicity delta outside compact}
Let $K$ be a compact set with nonempty interior and $x,y\in\partial K$. Let $\eta>0$ and $0<T^-<T^+$. The map
\[
\delta\mapsto \mathcal{N}^{K^c}_{\mathcal{C}}(x,y,\eta,T^-,T^+,\delta)
\]
is non-increasing.
\end{fact}

Moreover, a proof similar to the proof of Lemma~\ref{lem:constant} gives the following result.

\begin{lemm}
The exponential growth rate \[\limsup_{T\to +\infty}\frac{1}{T}\log \mathcal{N}^{K^c}_{\mathcal{C}}(x,y,\eta,T,T+C,\delta)\]  does not depend on $C>0$.
\end{lemm}


We can now start defining the chord entropy outside a compact set. Set
\[
h_{\mathcal C}^{K^c}(x,y,\eta,\delta)=\limsup_{T\to \infty}\frac{1}{T}\log \mathcal{N}^{K^c}_{\mathcal{C}}(x,y,\eta,T ,T+C,\delta)
\]
and
\[
h_{\mathcal C}^{K^c}(\eta,\delta)=\sup_{x,y\in\partial K}\limsup_{T\to \infty}\frac{1}{T}\log \mathcal{N}^{K^c}_{\mathcal{C}}(x,y,\eta,T ,T +C,\delta) = \sup_{x,y\in\partial K}h_{\mathcal C}^{K^c}(x,y,\eta,\delta) \,.
\]
\label{p:chord_entropy_outside_K}
The \emph{chord entropy outside $K$} is defined as
\begin{equation}\label{def:chord-entropy-outside-K}
   h_{\mathcal C}^{K^c}(\phi )=\lim_{\delta\to 0}\lim_{\eta\to 0} h_{\mathcal C}^{K^c}(\eta,\delta )\,.
\end{equation}
This definition makes sense because the function $\eta\to h^{K^c}_{\mathcal{C}}(\eta,\delta)$ is non-decreasing, and the function $\delta\to h^{K^c}_{\mathcal{C}}(\eta,\delta)$ is non-increasing.

In the following proposition we prove that the above quantity is essentially non-increasing when $K$ grows.

\begin{prop}\label{prop croissance chord entropy} If $K_1\subset \inter{K}_2$ and $K_1$ has nonempty interior, then
\[ h_{\mathcal C}^{K_2^c}(\phi )\le h_{\mathcal C}^{K_1^c}(\phi )  \,.
\]
 \end{prop}

This proposition is proved below. It motivates the following definition.

\begin{defi}\label{definition_chord_entropy_infinity}
The \emph{chord entropy at infinity} is
\begin{equation}\label{def:chord-entropy-at-infinity}
h_{\mathcal C}^{\infty}(\phi)=\inf_K h^{K^c}_{\mathcal{C}}(\phi)=\inf_K \lim_{\delta\to 0}\lim_{\eta\to 0}\sup_{x,y\in\partial K}\limsup_{T\to \infty} \frac{1}{T}\log \mathcal N_{\mathcal C}^{K^c}(x,y,\eta,T ,T+C,\delta)\, ,
\end{equation}
the infimum being taken over all compact sets $K$ with nonempty interior.
\end{defi}

The following corollary is an immediate consequence of Proposition~\ref{prop croissance chord entropy} and of Definition~\ref{definition_chord_entropy_infinity}.

\begin{coro}[Chord entropy at infinity is invariant under compact perturbations]\label{coro:chord-entropy-at-infinity-perturbation} Let $\phi\colon M_1\to M_1$ and $\psi\colon M_2\to M_2$ be two $H$-flows such that there exist two compact sets $K_1\subset M_1$ and $K_2\subset M_2$ with $M_1\setminus K_1=M_2\setminus K_2$, $\phi(M_1\setminus K_1)=\psi(M_2\setminus K_2)$ and ${\phi}_{|M_1\setminus K_1}={\psi}_{|M_2\setminus K_2}$.
Then
\[
 h_{\mathcal C}^\infty(\phi)=
 h_{\mathcal C}^\infty(\psi).
\]
\end{coro}

\begin{proof}[Proof of Proposition~\ref{prop croissance chord entropy}]
The main idea of the proof is the following.
For $x,y\in\partial K_2$, find some points $x',y'\in\partial K_1$ such that the number of chords from $x$ to $y$ outside $K_2$ is bounded by the number of chords from $x'$ to $y'$ outside $K_1$.
As the chord entropy outside a compact set is defined by counting chords, the theorem is proved.
More precisely, the points $x'$ and $y'$ have the following property: there exist a chord from $x'$ to $B(x,\eta/2)$ contained in $M\setminus \inter{K_1}$ and a chord from $B(y,\eta/2)$ to $y'$ contained in $M\setminus \inter{K_1}$.
We can now concatenate these two chords and a chord from $x$ to $y$ outside $K_2$ to obtain a chord from $x'$ to $y'$ outside $K_1$. As this process is injective, we obtain the desired inequality between the number of chords and therefore a proof of the theorem.
See Figure~\ref{fig:comparaison_entropie_hors_compact}.

\begin{figure}
    \centering
    \includegraphics{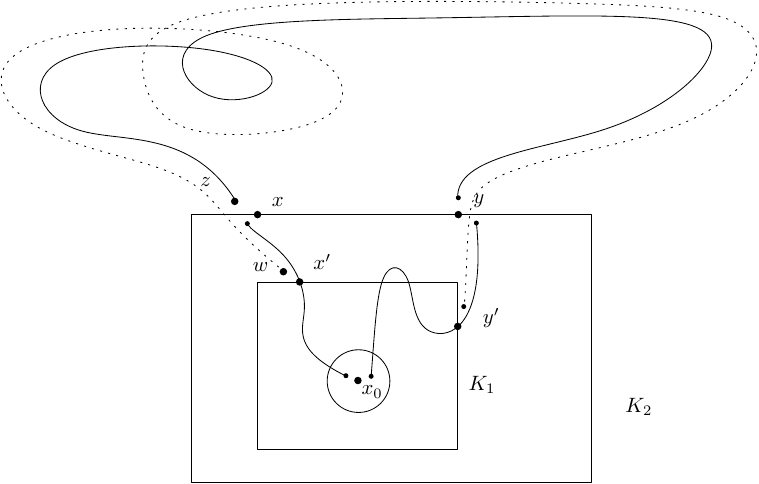}
    \caption{How to construct a chord connecting $x'$ to $y'$ from a chord connecting $x$ to $y$}
    \label{fig:comparaison_entropie_hors_compact}
\end{figure}
{\bf Step 1.} Setting the parameters.
Fix $0<\delta< \min(d(\partial K_1,\partial K_2),1)$ and $C>0$. Since the interior of $K_1$ is nonempty, we can fix $x_0$ and $\delta_0>0$ such that $B(x_0,\delta_0)\subset \inter{K_1}$. Fix $0<\alpha<\min(\delta/4,\delta_0)$. By the finite exact shadowing property, see Definition~\ref{weak shadowing property}, applied at the compact set $\overline{B(K_2,1)}$, $\delta=\alpha$ and $N=3$, we get some $0<\eta<\alpha$.

{\bf Step 2.} Comparing chords outside $K_1$ and $K_2$.
We now show that for every $x,y\in\partial K_2$ there exists $x',y'\in\partial K_1$ and $\ell(x'),\ell(y')>0$ such that
\[
\mathcal{N}_{\mathcal C}^{K_2^c}\left(x,y,\frac{\eta}{2}, T,T+C,\delta\right)\le \mathcal{N}_{\mathcal C}^{K_1^c}\left(x',y',\alpha,T+\ell(x')+\ell(y'),T+\ell(x')+\ell(y')+C,\frac{\delta}{2}\right)
\]
for any $T>0$.

{\bf Step 2a.} How to find $x'$ and $y'$?
Let $x,y\in\partial K_2$.
By Lemma~\ref{transitivite prop} applied at points $x,x_0\in K_2$ with $\delta=\frac{\eta}{2}>0$, there exists $t>0$ such that $\phi_{t}(B(x_0,\frac{\eta}{2}))\cap B(x,\frac{\eta}{2})\neq \emptyset$. Since $\frac{\eta}{2}<\delta_0$, we have $B\left(x_0,\frac{\eta}{2}\right)\subset \inter{K_1}$.
Therefore, there exist a point $x'\in\partial K_1$ and a time $0<\ell(x')\le t$ such that
\[
\phi_{\ell(x')}(x')\in B\left(x,\frac{\eta}{2}\right)\quad\text{and}\quad \phi_{[0,\ell(x')]}(x')\cap \inter{K_1}
=
\emptyset\, .
\]
By a similar argument, there exists a point $y'\in\partial K_1$ and a time $\ell(y')>0$ such that
\[
\phi_{-\ell(y')}(y')\in B\left(y,\frac{\eta}{2}\right)\quad\text{and}\quad \phi_{[-\ell(y'),0]}(y')\cap \inter{K_1}=\emptyset\, .
\]

{\bf Step 2b} Constructing a map $f$ from chords connecting $x$ to $y$, to chords connecting $x'$ to $y'$.
Let $E$ be a $E^{K_2^c}(x,y,\frac{\eta}{2},T,T+C,\delta)$-set of maximal cardinality.
We define a map
\[
f\colon E\to \mathcal{C}^{K_1^c}(x',y',\alpha, T+\ell(x')+\ell(y'),T+\ell(x')+\ell(y')+C)
\]
as follows. Let $z\in E$. In particular, there exists $\ell(z)\in [T,T+C]$ such that $z\in B(x,\frac{\eta}{2}), \phi_{\ell(z)}(z)\in B(y,\frac{\eta}{2})$ and $\phi_{[0,\ell(z)]}(z)\cap (K_2)_{-\frac{\eta}{2}}=\emptyset$. Observe that
\[
d(\phi_{\ell(x')}(x'),z)<\eta\quad\text{and}\quad d(\phi_{\ell(z)}(z),\phi_{-\ell(y')}(y'))<\eta\, .
\]
By the finite exact shadowing property, see Definition~\ref{weak shadowing property}, applied at $x', z, \phi_{-\ell(y')}(y')\in \overline{B(K_2,1)}$, we obtain a point $w$ such that
\begin{itemize}
    \item for all $s\in[0,\ell(x')]$, $d(\phi_s(w),\phi_s(x'))<\alpha$;
    \item for all $s\in[0,\ell(z)]$, $d(\phi_{\ell(x')+s}(w),\phi_s(z))<\alpha$;
    \item for all $s\in[0,\ell(y')]$, $d(\phi_{\ell(x')+\ell(z)+s}(w),\phi_{-\ell(y')+s}(y'))<\alpha$.
\end{itemize}
In particular, we have $w\in B(x',\alpha)$ and $\phi_{\ell(x')+\ell(z)+\ell(y')}(w)\in B(y',\alpha)$ with $\ell(x')+\ell(z)+\ell(y')\in[T+\ell(x')+\ell(y'),T+\ell(x')+\ell(y')+C]$.
Moreover, since $\phi_{[0,\ell(x')]}(x')$ and $\phi_{[-\ell(y'),0]}(y')$ do not intersect $\inter{K_1}$, both $\phi_{[0,\ell(x')]}(w)$ and $\phi_{[\ell(x')+\ell(z),\ell(x')+\ell(z)+\ell(y')]}(w)$ do not intersect $(K_1)_{-\alpha}$.
Additionally, since $\phi_{[0,\ell(z)]}(z)$ does not intersect $(K_2)_{-\frac{\eta}{2}}$ and $\eta/2<\delta< d(\partial K_1,\partial K_2)$, the arc $\phi_{[\ell(x'),\ell(x')+\ell(z)]}(w)$ does not intersect $(K_1)_{-\alpha}$.
This proves
\[w\in \mathcal{C}^{K_1^c}(x',y',\alpha, T+\ell(x')+\ell(y'),T+\ell(x')+\ell(y')+C).\]
Set $f(z)=w$.

{\bf Step 2c.} The map $f$ is injective.
Let now $z_1,z_2\in E$ and assume that $f(z_1)=f(z_2)=w$. By the construction of $f$ and since $\ell(z_1)\geq T$ and $\ell(z_2)\geq T$, for every $s\in [0,T]$, we have
\[
d(\phi_s(z_1),\phi_s(z_2))\leq d(\phi_s(z_1),\phi_{\ell(x')+s}(w))+d(\phi_s(z_2),\phi_{\ell(x')+s}(w))<2\alpha<\frac{\delta}{2}.
\]
Since the set $E$ is $(\delta, T)$-separating, we conclude that $z_1=z_2$, i.e., $f$ is injective.

{\bf Step 2d.} Consequences on the numbers of chords.
Observe that the set $f(E)$ is contained in $\mathcal{C}^{K_1^c}(x',y',\alpha,T+\ell(x')+\ell(y'),T+\ell(x')+\ell(y')+C)$.
Moreover, since $E$ is $(\delta,T)$-separating and $2\alpha<\frac{\delta}{2}$, the set $f(E)$ is $(\frac{\delta}{2}, T+\ell(x')+\ell(y'))$-separating.
Since $\# E=\# f(E)$ (as $f$ is injective) and by Fact~\ref{E set point un et deux alors trois outside compact}, we then conclude that
\[
\mathcal{N}_{\mathcal C}^{K_2^c}(x,y,\frac{\eta}{2}, T,T+C,\delta)\le \mathcal{N}_{\mathcal C}^{K_1^c}(x',y',\alpha,T+\ell(x')+\ell(y'),T+\ell(x')+\ell(y')+C,\frac{\delta}{2}).
\]

{\bf Step 3.} Conclusion.
The previous inequality implies that, for every $x,y\in \partial K_2$, there exists $x',y'\in\partial K_1$ such that
\[
h_{\mathcal{C}}^{K^c_2}(x,y,\frac{\eta}{2},\delta)\leq h_{\mathcal{C}}^{K^c_1}(x',y',\alpha,\frac{\delta}{2})\leq h^{K^c_1}_{\mathcal{C}}(\alpha,\frac{\delta}{2})\, ,
\]
where $\eta<\alpha<\delta$. Considering then the supremum over $x,y\in \partial K_2$, we obtain
\[
h_{\mathcal{C}}^{K^c_2}(\frac{\eta}{2},\delta)\leq h_{\mathcal{C}}^{K^c_1}(\alpha,\frac{\delta}{2})\, ;
\]
taking the limit as $\eta\to 0$ and then the limit as $\alpha\to 0$, we have
\[
\lim_{\eta\to 0} h_{\mathcal{C}}^{K^c_2}(\frac{\eta}{2},\delta)\leq \lim_{\alpha\to 0}h_{\mathcal{C}}^{K^c_1}(\alpha,\frac{\delta}{2})\,.
\]
We now let $\delta\to 0$ and obtain $h^{K_2^c}_{\mathcal{C}}(\phi)\leq h^{K_1^c}_{\mathcal{C}}(\phi)$, as required.
\end{proof}

\begin{rema}
The proof of Proposition~\ref{prop croissance chord entropy} relies only on the finite exact shadowing property and the transitivity of the flow.
\end{rema}


\subsection{Gurevic entropy at infinity and chords entropy at infinity coincide }\label{gur et chords infinity}

Our goal is now to show that counting the chords at infinity is the same as counting the periodic orbits at infinity, as presented in the following statement.

\begin{theo}\label{theorem gur infty egal chords infty} Let $\phi:M\to M$ be a $H$-flow. Then
 \[
 h_\mathrm{Gur}^\infty(\phi)=h_{\mathcal C}^\infty(\phi)\,.
 \]
\end{theo}

The following corollary is immediate from the above theorem and corollary \ref{coro:chord-entropy-at-infinity-perturbation}.

\begin{coro}[Gurevic entropy at infinity is invariant under compact perturbations]  Let $\phi\colon M_1\to M_1$ and $\psi\colon M_2\to M_2$ be two $H$-flows such that there exist two compact sets $K_1\subset M_1$ and $K_2\subset M_2$ with $M_1\setminus K_1=M_2\setminus K_2$, $\phi(M_1\setminus K_1)=\psi(M_2\setminus K_2)$ and ${\phi}_{|M_1\setminus K_1}={\psi}_{|M_2\setminus K_2}$.
Then
\[
 h_\mathrm{Gur}^\infty(\phi)=
 h_\mathrm{Gur}^\infty(\psi)\,.
\]
\end{coro}

The proof of the inequality
$h_{\mathcal C}^\infty(\varphi)\le h_\mathrm{Gur}^\infty(\phi)$ is easier and done in Proposition \ref{easy-ineq} below.
The hard inequality is $h_\mathrm{Gur}^\infty(\phi)\le h_{\mathcal C}^\infty(\phi)$. Indeed, we saw in section \ref{sec:chords-periodic} that chords and periodic orbits have  the same exponential growth rate. However, the Gurevic entropy at infinity  counts periodic orbits that spend a small proportion of  time in $K$, but an unbounded amount of time, whereas the chord entropy at infinity counts chords outside $K$, that can be closed into periodic orbits that spend a bounded amount of time in $K$. We follow the strategy developed in \cite{GST} and cut a periodic orbit that spends most of its time outside $K$ into successive excursions outside $K$. This is expressed in the technical Theorem~\ref{difficult-technical} whose Corollary~\ref{difficult-ineq} gives the desired inequality.

\begin{prop}\label{easy-ineq} Let $\phi:M\to M$ be a $H$-flow.
Then
\[
h_{\mathcal C}^\infty(\varphi)\le h_\mathrm{Gur}^\infty(\phi).
\]
\end{prop}

\begin{proof}

The main idea of the proof is the following: the chord entropy at infinity can be approximated by counting separated chords outside a big compact $K$ which start in a neighborhood of $x\in\partial K$ and end in neighborhood of $y\in\partial K$. These orbits can be closed to obtain different periodic orbits which intersect $K$ but stay a finite amount of time in $K$. As these orbits contribute to the Gurevic entropy at infinity we obtain that the chord entropy at infinity is smaller than the Gurevic entropy at infinity.
We now give a detailed proof of the proposition.

First note that if $h_\mathcal{C}^\infty (\phi) = \infty$, then $h_\mathcal{C}(\phi) = h_\mathrm{Gur}(\phi) = \infty$ and, from Lemma~\ref{prop_entropies_Gur_finies_simultanement}, $h_\mathrm{Gur}^\infty(\phi) = \infty$: the proposition is proved.
Therefore, in the remaining part of the proof, we may assume $h_\mathcal{C}^\infty (\phi) < \infty$.

{\bf Step 1.} We quantitatively approximate the chord entropy at infinity by  counting chords from $x$ to $y$ outside a compact set.
Fix $\epsilon>0$ and $\sigma_0>0$. By the definition of $h_{\mathcal{C}}^\infty(\phi)$ (Definition~\ref{definition_chord_entropy_infinity}) and Proposition~\ref{prop croissance chord entropy}, there exists a compact set $K_0$ with nonempty interior such that, for every compact set $K$ for which $K_0\subset \inter{K}$, we have
\[
h_{\mathcal{C}}^{K^c}(\phi)\in \left[h_{\mathcal{C}}^\infty(\phi),h_{\mathcal{C}}^\infty(\phi)+\frac{\epsilon}{8}\right]\, .
\]
Fix a compact set $K$ such that $K_0\subset \inter{K}$.
Without loss of generality, we may assume there exists $x_0\in\inter{K}$ such that $B(x_0,2\epsilon)\subset K$. There exists $0<\delta<\epsilon$ so that
\[
\lim_{\eta\to 0}\sup_{x,y\in\partial K}\limsup_{T\to\infty}\dfrac 1 T \log \left(\mathcal{N}_{\mathcal{C}}^{K^c}(x,y,\eta, T,T+\sigma_0,\delta)\right)\in[h_{\mathcal{C}}^\infty(\phi)-\epsilon/4,h_{\mathcal{C}}^\infty(\phi)+\epsilon/4]\, .
\]
Fix $0<\alpha<\frac{\delta}{3}$.
Apply the multiple closing lemma (Lemma \ref{petal unis}) to the compact $\overline{B(K,\epsilon)}, \delta=\alpha$, $\nu=1$ and $N=2$ to obtain a time $T_\mathrm{min}>0$ and a parameter $0<\eta<\alpha$.
Coming back to the chord entropy, we can assume that $\eta<\delta$ is small enough such that
\[
\sup_{x,y\in\partial K}\limsup_{T\to\infty}\dfrac 1 T \log \left(\mathcal{N}_{\mathcal{C}}^{K^c}(x,y,\frac{\eta}{2}, T,T+\sigma_0,\delta)\right)\in[h_{\mathcal{C}}^\infty(\phi)-\epsilon/2,h_{\mathcal{C}}^\infty(\phi)+\epsilon/2]\, .
\]
Consider then $x,y\in \partial K$ such that
\begin{equation}\label{approx h cordes infty}
\limsup_{T\to\infty}\dfrac 1 T \log \left(\mathcal{N}_{\mathcal{C}}^{K^c}(x,y,\frac{\eta}{2}, T,T+\sigma_0,\delta)\right)\in[h_{\mathcal{C}}^\infty(\phi)-\epsilon,h_{\mathcal{C}}^\infty(\phi)+\epsilon]\, .
\end{equation}
Apply the uniform transitivity (Lemma~\ref{transitivite prop bis}) to the compact set $K$, the compact set $K'=\overline{B(x_0,\epsilon)}$ and $\delta=\frac{\eta}{2}$ and obtain a time $\sigma>0$. Without loss of generality, we may assume $\sigma-\tau_{K}\geq 1$.

\begin{figure}
    \centering
    \includegraphics{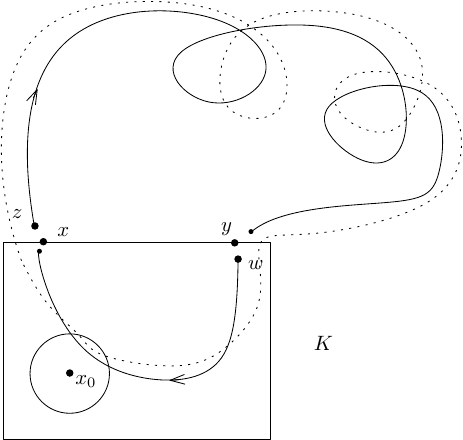}
    \caption{Construction of $f$}
    \label{fig:entropies_cordes_gur_infini}
\end{figure}

{\bf Step 2.} We construct a map $f$ from chords to periodic orbits (see Figure~\ref{fig:entropies_cordes_gur_infini})
Fix $\rho>0$.
Let
\begin{equation}\label{estimation T bonbon}T>\max\left(T_\mathrm{min}, \frac{\sigma+\tau_{K}+1}{\rho}, \sigma_0+\sigma+\tau_{K}+1\right)\, .\end{equation}
Consider a set $E$ that is a $E^{K^c}(x,y,\frac{\eta}{2} ,T ,T+\sigma_0 ,\delta)$-set of maximal cardinality, i.e.,
\[\#E = \mathcal{N}_{\mathcal{C}}^{K^c}(x, y, \frac{\eta}{2}, T, T+\sigma_0, \delta).\]
We now build a map
\[
f\colon E \to \mathcal{P}_{K_{-2\alpha}}^\rho(T, T+\sigma_0')
\]
where $\sigma_0'=\sigma_0+\sigma+\tau_{K}+1$ and $\mathcal{P}_{K_{-2\alpha}}^\rho(T, T+\sigma_0')$ is defined in Definition~\ref{definition_gurevic_infini}. We fix once for all a point $w\in B(y,\frac{\eta}{2})$ such that $\phi_{\ell(w)}(w)\in B(x,\frac{\eta}{2})$ for some $\ell(w)\in[\sigma-\tau_{K},\sigma+\tau_{K}]$ and $\phi_{[0,\ell(w)]}(w)\cap \overline{B(x_0,\epsilon)}\neq \emptyset$.
Such a point exists thanks to uniform transitivity (Lemma~\ref{transitivite prop bis}) applied at the points $x,y\in \partial K\subset K$.

Let $z\in E$.
In particular, $z\in B(x,\frac{\eta}{2})$ and $\phi_{\ell(z)}(z)\in B(y,\frac{\eta}{2})$ for some $\ell(z)\in [T,T+\sigma_0]$.
Moreover, $\phi_{[0,\ell(z)]}(z)\cap K_{-\frac{\eta}{2}}=\emptyset$ (this also comes from the definition of $E$, see Definition~\ref{defi E-set outside compact}).
By the multiple closing lemma (Lemma \ref{petal unis}) applied at the points $z$ and $w$ and times $\ell(z)$ and $\ell(w)$, we obtain a periodic orbit $\gamma$ of period $l(\gamma)\in [\ell(z)+\ell(w)-1,\ell(z)+\ell(w)+1]\subset [T,T+\sigma_0']$ such that
\begin{itemize}
    \item for all $s\in[0, \ell(z)]$, we have $d(\gamma(s), \phi_s(z))<\alpha$;
    \item for all $s\in[0, \ell(w)]$, we have $d(\gamma(\ell(z)+s), \phi_s(w))<\alpha$.
\end{itemize}
In particular, as $\phi_{[0,\ell(z)]}(z)\cap K_{-\frac{\eta}{2}}=\emptyset$, we have $\gamma([0,\ell(z)])\cap K_{-\frac{\eta}{2}-\alpha}=\emptyset$. As $\eta<\alpha$, we obtain
$\gamma([0,\ell(z)])\cap K_{-2\alpha}=\emptyset$.
Therefore
\[
\frac{\ell(\gamma\cap K_{-2\alpha})}{\ell(\gamma)}\leq \frac{\ell(\gamma)-\ell(z)}{T}\leq \frac{\ell(w)+1}{T}\leq \frac{\sigma+\tau_{K}+1}{T}<\rho
\]
(where the last inequality is satisfied as $T$ has been chosen big enough, according to \eqref{estimation T bonbon}).

We now prove that $\gamma$ intersects $K_{-2\alpha}$. As $\phi_{[\ell(z),\ell(z)+\ell(w)]}(w)\cap \overline{B(x_0,\epsilon)}\neq \emptyset$, we have that $\gamma([\ell(z),\ell(z)+\ell(w)])\cap B(x_0,\epsilon+\alpha)\neq \emptyset$. As $\alpha<\frac{\delta}{3}<\frac{\epsilon}{3}$ and $B(x_0,2\epsilon)\subset K$, we obtain $B(x_0,\epsilon+\alpha)\subset K_{-2\alpha}$ and $\gamma\cap K_{-2\alpha}\neq \emptyset$.
This proves $\gamma\in \mathcal{P}_{K_{-2\alpha}}^\rho(T, T+\sigma_0')$.
Let $f(z)=\gamma$.

{\bf Step 3.} The map $f$ is almost injective.
We now control the cardinality of the preimage by $f$ of every periodic orbit.
Let $z_1,z_2\in E$ be such that $f(z_1)=f(z_2)=\gamma$. Let $s_1$, resp. $s_2$, be the origin of $\gamma$ that comes with the construction of $f(z_1)$, resp. $f(z_2)$.
Without loss of generality me may assume $0\le s_1\leq s_2<l(\gamma)/2$.

We first prove $|s_2-s_1|\leq \sigma_0'$.
By construction, $\gamma([s_1,s_1+T])\cap K_{-2\alpha} = \emptyset $ and  $\gamma([s_2,s_2+T])\cap K_{-2\alpha} = \emptyset $.
Moreover $\gamma([s_1+T,s_1+\ell(\gamma)])\cap K_{-2\alpha}\neq \emptyset$ and  $\gamma([s_2+T,s_2+l(\gamma)])\cap K_{-2\alpha}\neq \emptyset$.
As $T> \sigma_0'$, again by \eqref{estimation T bonbon}, we have $T> \ell(\gamma)/2$.
Therefore, $[s_2+T,s_2+\ell(\gamma)]\subset [s_1, s_1+\ell(\gamma)+T]$ and, to ensure $\gamma([s_1+T,s_1+\ell(\gamma)])\cap K_{-2\alpha}\neq \emptyset$ and $\gamma([s_2+T,s_2+\ell(\gamma)])\cap K_{-2\alpha}\neq \emptyset$, we must have $[s_1+T,s_1+\ell(\gamma)]\cap [s_2+T,s_2+\ell(\gamma)]\neq \emptyset$, since $\gamma([s_1,s_1+T])=\gamma([s_1+\ell(\gamma),s_1+\ell(\gamma)+T])$ cannot intersect $K_{-2\alpha}$. Therefore $s_2-s_1\leq  \ell(\gamma)-T$.
Thus $0\leq s_2-s_1\leq \sigma_0'$.

Then, for all $\tau\in[0,T]$ we have, using \eqref{eqn:minoration},
\begin{align*}
d(\phi_{\tau}(z_1),\phi_{\tau}(z_2)) &\leq d(\phi_{\tau}(z_1),\gamma(s_1+\tau))+d(\gamma(s_1+\tau),\gamma(s_2+\tau))+d(\phi_{\tau}(z_2),\gamma(s_2+\tau)) \\
&<\alpha + b|s_2-s_1|+\alpha\, .
\end{align*}
If $|s_2-s_2|\leq \alpha/b$, then
\[ d(\phi_{\tau}(z_1),\phi_{\tau}(z_2)) \leq 3\alpha < \delta\, .\]
Since $E$ is a $(\delta, T)$-separating set, we conclude that $z_1=z_2$.
Therefore
\[ \# f^{-1}(\gamma)\leq \left\lceil\frac{b \sigma_0'}{\alpha}\right\rceil.\]

{\bf Step 4.} Conclusion.
We just proved that
\[
\mathcal{N}_{\mathcal{C}}^{K^c}(x,y,\frac{\eta}{2},T,T+\sigma_0,\delta)\leq \left\lceil\frac{b \sigma_0'}{\alpha}\right\rceil \#\mathcal{P}_{K_{-2\alpha}}^\rho(T,T+\sigma_0')\, .
\]
By taking the exponential growth rate in the previous inequality when $T\to +\infty$ we get
\[ \limsup_{T\to\infty}\dfrac 1 T \log \left(\mathcal{N}_{\mathcal{C}}^{K^c}(x,y,\frac{\eta}{2}, T,T+\sigma_0,\delta)\right)\leq \limsup_{T\to\infty}\dfrac 1 T \log \left(\#\mathcal{P}_{K_{-2\alpha}}^\rho(T,T+\sigma_0')\right)=h_{\mathrm{Gur}}^{K_{-2\alpha},\rho}(\phi)\]
and therefore, by \eqref{approx h cordes infty},
\[
h^\infty_{\mathcal{C}}(\phi)\leq \epsilon +h_{\mathrm{Gur}}^{K_{-2\alpha},\rho}(\phi).
\]

We now consider the limit $\rho\to 0$ to obtain
\[
h^\infty_{\mathcal{C}}(\phi)\leq \epsilon +\lim_{\rho\to 0} h_{\mathrm{Gur}}^{K_{-2\alpha},\rho}(\phi).
\]
We can now take the infimum over $K$ (as $K\mapsto \lim_{\rho\to 0} h_{\mathrm{Gur}}^{K,\rho}(\phi)$ is non-increasing with respect to inclusion (Fact~\ref{fait:h_Gur_K_epsilon}), taking the infimum means taking big compact sets $K$ and therefore is compatible with the conditions on $K$).
By the arbitrariness of $\epsilon$, we conclude that $h^\infty_{\mathcal{C}}(\phi)\leq h^{\infty}_{\mathrm{Gur}}(\phi)$.
\end{proof}

As said above, the other inequality is more difficult, but the proof follows closely the proof of \cite[Theorem 5.1]{GST}.

Given two compacts sets $K_1\subset K_2$, $\sigma_0>0$ and $0<\alpha\le 1$, we introduce for all $T>0$ the set of periodic orbits with length roughly $T$ that intersect $K_1$ and spend a small amount of time in $K_2$\,:
\[
\mathcal P(K_1,K_2,\alpha,T,T+\sigma_0)=\{\gamma\in\mathcal P_{K_1}(T,T+\sigma_0),\,\ell(\gamma\cap K_2)\le \alpha \ell(\gamma)\,\}\,.
\]
\label{p:P_K1_K2}
As in \cite{GST}, we need the following lemma.

\begin{lemm}\label{lemme:comparaison_vairante_P_K_T}
Let $\phi\colon M\to M$ be a $H$-flow.  Given any compact sets $K_1\subset K_2$ with $\inter{K_1}\neq\emptyset$, $\sigma_0>0$ and $\alpha>0$, there exists $0<\delta<1$ such that
\begin{eqnarray*}
\limsup_{T\to +\infty} \frac{1}{T}\log
\#\mathcal{P}_{K_2}^\alpha(T,T+\sigma_0)
&\leq&
\limsup_{T\to +\infty} \frac{1}{T}\log
\#\mathcal{P}(K_1,(K_2)_{-\delta},2\alpha,T,T+\sigma_0)\\
&\leq&
  \limsup_{T\to +\infty} \frac{1}{T}\log
\#\mathcal P_{(K_2)_{-\delta}}^{2\alpha}(T,T+\sigma_0)\, ,
\end{eqnarray*}
where $(K_2)_{-\delta}=K_2\setminus \bigcup_{x\in\partial K_2}B(x,\delta)$.
\end{lemm}

\begin{proof}

The second inequality is easily obtained, thanks to the inclusion
\[
\mathcal{P}(K_1,(K_2)_{-\delta},2\alpha,T,T+\sigma_0)\subset \mathcal{P}_{(K_2)_{-\delta}}^{2\alpha}(T,T+\sigma_0).
\]

We now prove the first inequality.
The proof goes as follows. If $\gamma$ is a periodic orbit intersecting $K_2$ then, using transitivity and the multiple closing lemma, we can make it do a small detour to intersect $K_1$. As the detour is controlled, we still control the time spent in $K_2$  (in fact $(K_2)_{-\delta}$) as well as the period of the new periodic orbit. This leads to an inequality between the number of periodic orbits in $\mathcal{P}_{K_2}^\alpha(T,T+\sigma_0)$ and $\mathcal{P}(K_1,(K_2)_{-\delta},2\alpha,T,T+\sigma_0+\sigma_0')$, for some suitable $\sigma_0'>0$.
This proves
\begin{eqnarray*}
\limsup_{T\to +\infty} \frac{1}{T}\log
\#\mathcal{P}_{K_2}^\alpha(T,T+\sigma_0)
&\leq&
\limsup_{T\to +\infty} \frac{1}{T}\log
\#\mathcal{P}(K_1,(K_2)_{-\delta},2\alpha,T,T+\sigma_0+\sigma_0')\, .
\end{eqnarray*}
We then prove that the limit superior does not depend on $\sigma_0'$. See Figure~\ref{fig:entropies_gur_cordes_infini_1}.
We now give a detailed proof.

{\bf Step 1.} Setting the parameters.
Since $\inter{K_1}\neq \emptyset$, we can fix a point $x_0\in \inter{K_1}$ such that $B(x_0,\delta_0)\subset \inter{K_1}$ for some $\delta_0>0$.
Lemma~\ref{lemma on same po} (i.e. the separation of orbits) applied with $\nu=1$ and $\tau_1=1$ gives us $\tau_0>0$ and $\epsilon_1>0$. Fix $0<\delta<\frac{\epsilon_1}{3}$ small enough such that $d(\partial K_1,\partial K_2)>2\delta$ and $B(x_0,\delta_0+\delta)\subset \inter{K_1}$.
From Lemma~\ref{petal unis}  (i.e. the multiple closing lemma) applied at the compact set $\overline{B(K_2,1)}$, with $\delta>0$, $\nu=1$ and $N=2$, we obtain a time $T_\mathrm{min}>0$ and $\eta>0$.
From unifom transitivity, i.e., Lemma~\ref{transitivite prop bis}, applied at the compact sets $\overline{B(x_0,\delta_0)}\subset \overline{B(K_2,1)}$ with $\frac{\eta}{2}>0$, we obtain a time $\sigma>0$.

We now cover the compact set $\overline{B(K_2,1)}$ with $N$ balls $B(x_i, \frac{\eta}{2})$.
For all $1\leq i\leq N$, by transitivity applied at the center of the ball $B(x_i,\frac{\eta}{2})$, there exists a point $z_i\in B(x_i,\frac{\eta}{2})$ and a time $\ell(z_i) \in[\sigma+1,\sigma+2\tau_{K_2}+1]$ such that $\phi_{\ell(z_i)}(z_i)\in B(x_i,\frac{\eta}{2})$ and $\phi_{[0,\ell(z_i)]}(z_i)\cap \overline{B(x_0,\delta_0)}\neq \emptyset$.
Let $\sigma_0'=\sigma+2(\tau_{K_2}+1)$.

\begin{figure}
    \centering
    \includegraphics{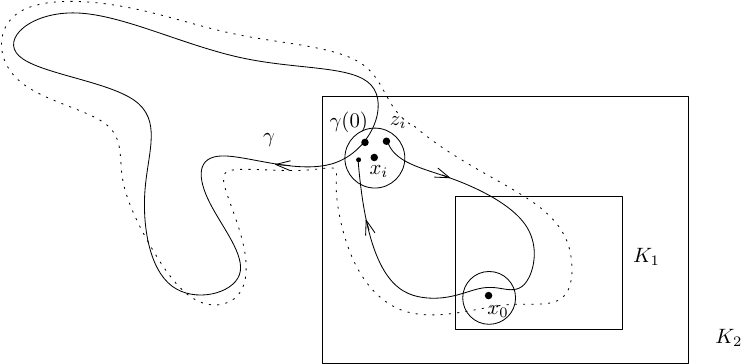 }
    \caption{Construction of a periodic orbit intersecting $K_1$.}
    \label{fig:entropies_gur_cordes_infini_1}
\end{figure}

{\bf Step 2.} We construct a map between our two sets of periodic orbits.
Fix
\begin{equation}\label{quel T pour pruve 5.25}
T \geq \max\left(T_\mathrm{min},\frac{\sigma_0'+\alpha \sigma_0}{\alpha}\right)\, ;
\end{equation}
we now construct a map
\[
f\colon \mathcal{P}_{K_2}^\alpha (T,T+\sigma_0)\to \mathcal{P}(K_1,(K_2)_{-\delta},2\alpha, T,T+\sigma_0+\sigma_0')\, .
\]
Let $\gamma\in \mathcal{P}_{K_2}^\alpha (T,T+\sigma_0)$ and assume, without loss of generality, that $\gamma(0)\in K_2$. Let $1\leq i\leq N$ be such that $\gamma(0)\in B(x_i,\frac{\eta}{2})$.
Notice that
\[
d(\phi_{\ell(z_i)}(z_i),\gamma(0))<\eta\qquad\text{and}\qquad d(\gamma(0),z_i)=d(\phi_{\ell(\gamma)}(\gamma(0)),z_i)<\eta.
\]
Therefore, also by \eqref{quel T pour pruve 5.25}, we can use Lemma~\ref{petal unis} with points $x_1=z_i$, $x_2=\gamma(0)$, and times $\ell(z_i)\in[\sigma+1,\sigma+2\tau_{K_2}+1]$, $\ell(\gamma)\in[T,T+\sigma_0]$ to obtain a periodic orbit $\beta$ of period $\ell(\beta)\in[\ell(\gamma)+\ell(z_i)-1,\ell(\gamma)+\ell(z_i)+1]\subset [T, T+\sigma_0+\sigma_0']$ which is $\delta$-close first to $\phi_{[0,\ell(z_i)]}(z_i)$ and then to $\gamma$.

The orbit $\beta$ intersects $K_1$ because $\phi_{[0,\ell(z_i)]}(z_i)$ meets $\overline{B(x_0,\delta_0)}$ and $B(x_0,\delta_0+\delta)\subset \inter{K_1}$.
Moreover, since $\ell(\gamma\cap K_2)\leq \alpha \ell(\gamma)$ and $\beta$ is $\delta$-close to $\gamma$ on $[\ell(z_i), \ell(z_i)+\ell(\gamma)]$, we have
\[ \ell(\beta([\ell(z_i), \ell(z_i)+\ell(\gamma)]))
\cap (K_2)_{-\delta})\leq \alpha(T+\sigma_0)\]
and therefore
\[
\ell(\beta\cap (K_2)_{-\delta})\leq \alpha(T+\sigma_0)+\ell(\beta)-\ell(\gamma)\leq \alpha(T+\sigma_0)+\sigma_0'\, .
\]
So
\[\frac{\ell(\beta\cap (K_2)_{-\delta})}{\ell(\beta)}\leq \frac{\alpha(T+\sigma_0)+\sigma_0'}{T}\leq 2\alpha\, ,\]
as $T$ has been chosen big enough, according to \eqref{quel T pour pruve 5.25}.
Thus, $\beta\in \mathcal{P}(K_1,(K_2)_{-\delta},2\alpha, T, T+\sigma_0+\sigma_0')$.

Set $f(\gamma)=\beta$.

{\bf Step 3.} The map $f$ is almost injective.
We want now to control the loss of injectivity of the function $f$ defined above. Let $\gamma_1,\gamma_2\in\mathcal{P}^\alpha_{K_2}(T,T+C)$ be such that $f(\gamma_1)=f(\gamma_2)=\beta$.
Assume that $0\leq \ell(\gamma_1)-\ell(\gamma_2)\leq \tau_0$ (where $\tau_0$ is given by Lemma~\ref{lemma on same po} as explained above).
The point $\gamma_1(0)$ belongs to $B(x_i, \frac{\eta}{2})$, for some $1\leq i\leq N$; additionally, let us assume that also $\gamma_2(0)$ belongs to the same ball $B(x_i, \frac{\eta}{2})$. Actually, for this second choice we have $N$ possibilities, and we will count them when considering the cardinality of the preimage of a periodic orbit.

The construction of $f(\gamma_i)$ provides an origin of the orbit $\beta$, that we denote $s_j$, for $j=1,2$. We can assume without loss of generality that $s_1=0$.
For any $s\in[0,\ell(\gamma_2)]$, one has
\begin{align*}
d(\gamma_1(s),\gamma_2(s))&\leq d(\gamma_1(s),\beta(\ell(z_i)+s))+d(\beta(\ell(z_i)+s),\beta(\ell(z_i)+s_2+s))+d(\beta(\ell(z_i)+s_2+s),\gamma_2(s))\\
&\le \delta +b\vert s_2\vert +\delta\, ,
\end{align*}
where $b$ comes from \eqref{eqn:minoration}.
If $|s_2|\le \frac{\delta}{b}$, then
\[d(\gamma_1(s),\gamma_2(s)) \le 3\delta <\epsilon_1\, . \]
We can then conclude, by Lemma \ref{lemma on same po}, that $\gamma_1=\gamma_2$. It follows that the cardinal of the preimage of $\beta$ by $f$ is bounded by $N\left\lceil \frac{\sigma_0}{\tau_0} \right\rceil \left\lceil \frac{(T+\sigma_0+\sigma_0')b}{\delta}\right\rceil$.
Consequently
\[
\#\mathcal{P}^\alpha_{K_2}(T,T+\sigma_0)\le \,N\left\lceil \frac{\sigma_0}{\tau_0} \right\rceil \left\lceil \frac{(T+\sigma_0+\sigma_0')b}{\delta}\right\rceil \#\mathcal{P}(K_1,(K_2)_{-\delta},2\alpha, T, T+\sigma_0+\sigma_0')\, .
\]
This proves
\begin{eqnarray}\label{final inequality with C et C'}
\limsup_{T\to +\infty} \frac{1}{T}\log
\#\mathcal{P}_{K_2}^\alpha(T,T+\sigma_0)
&\leq&
\limsup_{T\to +\infty} \frac{1}{T}\log
\#\mathcal{P}(K_1,(K_2)_{-\delta},2\alpha,T,T+\sigma_0+\sigma_0')\, .
\end{eqnarray}

{\bf Step 4.} The limit superior $\limsup_{T\to +\infty} \frac{1}{T}\log
\#\mathcal{P}(K_1,K_2,\alpha,T,T+\sigma_0)$ does not depend on $\sigma_0$.
Let $0<\sigma_1<\sigma_2$. Let $n = \floor{\sigma_2/\sigma_1}$.
Then
\[
\mathcal P(K_1,K_2,\alpha,T,T+\sigma_2)\subset \cup_{i=0}^n \mathcal P(K_1,K_2,\alpha,T + i\sigma_1,T+(i+1)\sigma_1)
\]
and
\begin{align*}
\#\mathcal P(K_1,K_2,\alpha,T,T+\sigma_1) \le & \#\mathcal P(K_1,K_2,\alpha,T,T+\sigma_2) \\
\le & (n+1)\max_{i=0\dots n}\# \mathcal P(K_1,K_2,\alpha,T + i\sigma_1,T+(i+1)\sigma_1)\, .
\end{align*}
Therefore,
\[\limsup_{T\to +\infty} \frac{1}{T}\log
\#\mathcal{P}(K_1,K_2,\alpha,T,T+\sigma_1) =
\limsup_{T\to +\infty} \frac{1}{T}\log
\#\mathcal{P}(K_1,K_2,\alpha,T,T+\sigma_2)\, .\]
Thus, from this equality and from \eqref{final inequality with C et C'}, we conclude
\begin{eqnarray*}
\limsup_{T\to +\infty} \frac{1}{T}\log
\#\mathcal{P}_{K_2}^\alpha(T,T+\sigma_0)
&\leq&
\limsup_{T\to +\infty} \frac{1}{T}\log
\#\mathcal{P}(K_1,(K_2)_{-\delta},2\alpha,T,T+\sigma_0)
\end{eqnarray*}
as required.
\end{proof}

We are now able to state and prove the analogue of \cite[thm 5.1]{GST} in our context.

\begin{theo}\label{difficult-technical} Let $\phi\colon M\to M$ be a $H$-flow. Let $K\subset M$ be a compact set with $\inter{K}\neq\emptyset$. Let $\epsilon>0$.
There exist a map $\alpha\in(0,1)\to\psi(\alpha)\in(0,\infty)$ converging to $0$ when $\alpha\to 0$ and $R\ge 1$ such that for all $0<\alpha <1$ and $S>0$
\[
\limsup_{T\to \infty}\frac{1}{T}\log \#\mathcal{P}(K,K_R,\alpha,T,T+S)\le  h_{\mathcal C}^{K^c}(\phi)\left(1+2\dfrac{b\alpha}{R}\right)+\epsilon+\psi(\alpha). \]
where $K_R$ is the $R$-neighbourhood of $K$.
\end{theo}

From Theorem \ref{difficult-technical}, we will deduce the following result.

\begin{coro}\label{difficult-ineq} Let $\phi\colon M\to M$ be a $H$-flow.
Then $h_\mathrm{Gur}^\infty(\phi)\le h_{\mathcal C}^\infty(\phi)$.
\end{coro}

\begin{proof}
If $h_{\mathcal{C}}^\infty(\phi)=\infty$, then there is nothing to prove. Assume now $h_{\mathcal{C}}^\infty(\phi)<\infty$.
Let $\epsilon>0$.
Fix $K$ a compact subset with nonempty interior such that
\[h_\mathcal C^{K^c}(\phi)\le h_\mathcal C^\infty(\phi)+\epsilon\, .\]
Use Theorem~\ref{difficult-technical} to obtain $\psi$ and $R$.
As $\lim_{\alpha\to 0} h^{K_{R+1},\alpha}_\mathrm{Gur}(\phi)\ge h_\mathrm{Gur}^\infty(\phi)$ and since, by Fact~\ref{fait:h_Gur_K_epsilon}, the function $\alpha\mapsto h_{\mathrm{Gur}}^{K_{R+1},\alpha}(\phi)$ is non increasing, then for all $\alpha>0$ one has
\[h^{K_{R+1},\alpha}_\mathrm{Gur}(\phi)\ge h_\mathrm{Gur}^\infty(\phi).\]
Choose $\alpha$ such that $\alpha<1/2$, $2b\alpha/R<\epsilon$ and $\psi(2\alpha)<\epsilon$.
From Lemma~\ref{lemme:comparaison_vairante_P_K_T} with parameters $K\subset K_{R+1}$ and $\alpha$ we have
\begin{eqnarray*}
h_{\mathrm{Gur}}^{K_{R+1},\alpha}(\phi)=\limsup_{T\to +\infty} \frac{1}{T}\log
\#\mathcal{P}_{K_{R+1}}^\alpha(T,T+S)
&\leq&
\limsup_{T\to +\infty} \frac{1}{T}\log
\#\mathcal{P}(K,K_R,2\alpha,T,T+S)\, .
\end{eqnarray*}
We now use Theorem~\ref{difficult-technical} to obtain
\[
\limsup_{T\to \infty}\frac{1}{T}\log \#\mathcal{P}(K,K_R,2\alpha,T,T+S)\le  h_{\mathcal C}^{K^c}(\phi)(1+\epsilon)+\epsilon+\psi(2\alpha). \]
Therefore
\[
h_\mathrm{Gur}^\infty(\phi) \leq  h_{\mathcal C}^{\infty}(\phi)(1+\epsilon)+3\epsilon\, . \]
As $\epsilon$ can be choose arbitrarily small, we obtain  $h_\mathrm{Gur}^\infty(\phi)\le h_{\mathcal C}^\infty(\phi)$.
\end{proof}

We will prove Theorem~\ref{difficult-technical} by adapting the arguments of \cite{GST} to our context. We provide details when necessary, and refer to \cite{GST} for complements.

\begin{proof}[Proof of Theorem~\ref{difficult-technical}  ]

If $h_{\mathcal{C}}^{K^c}(\phi)=\infty$, then there is nothing to prove. Assume now that $h_{\mathcal{C}}^{K^c}(\phi)<\infty$. We start a given compact set $K\subset M$ with $\inter{K}\neq \emptyset$ and a given $\epsilon>0$.

The heuristic of the proof is the following.
We cut every periodic orbit $\gamma\in \mathcal{P}(K,K_R,\alpha,T,T+S)$ into chords joining   points of $\partial K$, in such a way that  a chord is either a \textit{large excursion } outside $K_R$, i.e. a connected component of $\gamma\setminus K$ that intersects $(K_R)^c$, or a chord between points of $\partial K$ staying inside $K_R$, between two large excursions.
This decomposition will give a bound on the number of periodic orbits in terms of number of chords.
More precisely, the number of large excursions, i.e. the chords going outside $K_R$ will be controlled by the chord entropy at infinity, while the number of other chords, i.e. those remaining inside $K_R$, will be controlled by the Gurevic entropy.

The set of useful chords is quite involved and bounding its cardinal will require some work.
This will be the main technical part of the proof.

{\bf Step 1.} We quantitatively approximate chord entropy at infinity.

Apply Lemma~\ref{lemma on same po} with $\nu=1$ and $\tau_1=S$: we obtain $\bar\tau_0>0$ and $\bar\epsilon>0$. Recall that $lip(\phi)\ge 1$ is the Lipschitz constant of the map $\phi_{\tau}$, for all $\tau\in[-1,1]$ (see point 1 of Definition~\ref{def:H-flow}).
By the definition of $h_{\mathcal{C}}^{K^c}(\phi)$ and since this quantity is finite, there exists $\overline{\delta}>0$ such that $\overline{\delta}<\overline{\epsilon}/(3lip(\phi))$, $\overline{\delta}\leq b/2$ and, for all $0<\delta\le\overline{\delta}$,
\[
\lim_{\eta \to 0} \sup_{x,y\in\partial K}\limsup_{T\to +\infty}\dfrac{\log \mathcal{N}_{\mathcal{C}}^{K^c}(x,y,\eta,T,T+S,\delta)}{T}\in \left(h_{\mathcal{C}}^{K^c}(\varphi)-\frac{\epsilon}{4},h_{\mathcal{C}}^{K^c}(\varphi)+\frac{\epsilon}{4}\right).
\]
Choose $\delta=\overline{\delta}$. There exists $\overline{\eta}<\overline{\delta}$ such that, for all $0<\eta\le\overline{\eta}$,
\begin{equation}\label{h corde K senza delta e eta}
\sup_{x,y\in\partial K}\limsup_{T\to +\infty}\dfrac{\log \mathcal{N}_{\mathcal{C}}^{K^c}(x,y,\eta,T,T+S,\overline{\delta})}{T}\in \left(h_{\mathcal{C}}^{K^c}(\varphi)-\frac{\epsilon}{2},h_{\mathcal{C}}^{K^c}(\varphi)+\frac{\epsilon}{2}\right)\, .
\end{equation}
Choose $0<\eta\leq\overline{\eta}$. As above, by the same argument as in the proof of Theorem \ref{theo:Gurevic}, the above limsup does not depend on $S$. Therefore, the quantifiers   $\eta$ and $\delta$ do not depend either on $S$.

{\bf Step 2.} We divide a periodic orbit in excursions.

Choose $M=M(\eta)\in\mathbb{N}$ a (minimal) number of balls  $B(x_i,\eta)$ centered at points  $x_i\in\partial K$ for $i=1,\dots ,M$ so that
\[
\partial K\subset \bigcup_{i=1}^MB(x_i,\eta).
\]
Fix an arbitrary $R\geq 1$, whose value will be determined later on the proof, and some $\alpha\in(0,1)$.
We will divide each periodic orbit of $\mathcal{P}(K,K_R,\alpha,T,T+S)$ into suitable chords.

\begin{figure}
    \centering
    \includegraphics{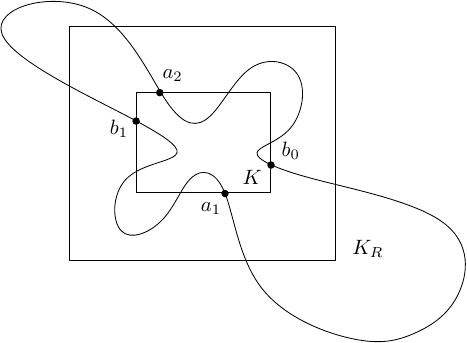}
    \caption{larges excursions}
    \label{fig:excursions}
\end{figure}

{\bf Step 2.a.} Large excursions.

Let $\gamma\in \mathcal{P}(K,K_R,\alpha,T,T+S)$. A \textit{large excursion} of $\gamma$  outside $K_R$ is a connected component of $\gamma\cap\overline{K^c}$ starting from $\partial K$, staying outside $K$ except at the endpoints of the interval, and intersecting $(K_R)^c$. We will divide $\gamma$ into $2N$ connected components, with $N$ large excursions separated by $N$ pieces of orbit that stay inside $K_R$.
Choose a point $b_0=b_{N}$ on $\gamma$ at the beginning of such a large excursion. Following \cite{GST}, we denote by $(a_i)_{1\le i\le N}$  on $\gamma\cap \partial K$ the endpoints of the large excursions, and starting points of components inside $K_R$, and $(b_i)_{0\le i\le N-1}$  on $\gamma\cap \partial K$, the starting points of the large excursions. With these notations, every large excursion goes from $b_{i-1}$ to $a_{i}$ for some $1\le i\le N$, and each complement goes from $a_i$ to $b_i$, for some $1\le i\le N$.
Parametrize $\gamma=(\gamma(t))_{t\in [0,\ell(\gamma)]}$   so that $b_0=b_N=\gamma(0)$, $b_i=\gamma(\tau_i)$, with $\tau_0=0$ and $\tau_N=\ell(\gamma)$, and $a_i=\gamma(\sigma_i)$.
We get a decomposition of $[0,\ell(\gamma)]$ into $2N$ intervals corresponding to $N$ large excursions $(\tau_i,\sigma_{i+1})$ for $i=0,\dots, N-1$, and their $N$ complements $(\sigma_i,\tau_i)$ for $i=1,\dots, N$.

{\bf Step 2.b.} Elementary observations on excursions.
As in \cite{GST}, let us do the following elementary but crucial observations.
\begin{enumerate}
\item\label{punto 1} By definition of  $ \mathcal{P}(K,K_R,\alpha,T,T+S)$, as $\ell(\gamma)\in [T,T+S]$,  \[
\ell(\gamma\cap K)\le \sum_{i=1}^N (\tau_i-\sigma_i)\le \ell(\gamma\cap K_R)< \alpha\ell(\gamma)\le \alpha T+\alpha S\,.
\]
\item\label{punto 2} For $1\le i\le N$, $\gamma([\tau_{i-1}, \sigma_{i}])$ lies outside $ K$, except at points $\gamma(\tau_{i-1})$ and $\gamma(\sigma_i)$ which belong to $\partial K$.
Moreover, by the right hand side in \eqref{eqn:minoration}, each large excursion spends a time at least $\frac{2R}{b}$ inside $K_R\setminus K$, where $b$ is the constant in \eqref{eqn:minoration}.
Thus, since there are $N$ large excursions, we get $\ell (\gamma\cap (K_R\setminus K))\ge \frac{2R}{b} N$. By definition, as $\gamma\in\mathcal{P}(K,K_R,\alpha,T,T+S)$, we know that $\ell(\gamma\cap K_R^c)\geq (1-\alpha)T$. Thus, we deduce that
\begin{equation}\label{dans thm diff bound sur N}
(1-\alpha)T+\frac{2R}{b}N\le \ell(\gamma\cap K_R^c)+\ell(\gamma\cap(K_R\setminus K))= \ell(\gamma\cap K^c)\le
\ell(\gamma)\le  T+S\,.
\end{equation}
Set
\[
\hat\nu = \hat{\nu}_{\alpha,R,T}= \frac{b\alpha}{R}T.
\]
From \eqref{dans thm diff bound sur N}, we get $N\le \frac{b(\alpha T+S)}{2R}$.
Moreover, when $T\ge \frac{S}{\alpha}$, then $\frac{b(\alpha T+S)}{2R} \le \hat\nu $, so that
\[
N\le \hat\nu \,.
\]
\end{enumerate}
We set $t_0=\tau_0=0$ and for $i=1,\dots, N$, $t_i=\lfloor \tau_i\rfloor$ and $s_i=\lfloor \sigma_i\rfloor$.
\begin{enumerate}[resume]
\item\label{punto 3}
When $T\ge\max(\frac{S}{\alpha},\frac{ RS}{\alpha b})$, we get  $\hat\nu \ge  S\ge \alpha S$ and by point \ref{punto 2},   $\hat\nu \ge N$ . Therefore,  by points \ref{punto 1} and \ref{punto 2}, we have
\[
 \sum_{i=1}^N(t_i-s_i)\le \sum_{i=1}^N (\tau_i-\sigma_i)+N\le \alpha T+\alpha S+N\le \alpha T+2\hat\nu \,.
\]
\item\label{punto 4}
For $T\ge\max(\frac{S}{\alpha},\frac{ RS}{\alpha b})$, we deduce from point \ref{punto 2}  that
\begin{equation*}
\sum_{i=1}^{N}(s_i-t_{i-1}) \le \sum_{i=1}^N (\sigma_i-\tau_{i-1})+N
\le  T+S+N\le T+ 2\hat\nu  \, .
\end{equation*}
\item\label{punto 5}
As $\tau_N = \ell(\gamma)$, we have
\[ T-1 \leq t_N \leq T+S.\]
\end{enumerate}

{\bf Step 2.c.} Construction of a map from periodic orbits to a set of excursions.
We first define our set of excursions.
Let $\mathcal E$ be defined as

\begin{align*}
\bigcup_{N=1}^{\hat\nu }\bigcup_{\substack{(t_1,\dots,t_N)\in\mathbb{N}^N \\ (s_1,\dots,s_N)\in\mathbb{N}^N\\ \vert t_N -T\vert\le S+1 \\ 0<s_1\le t_1<s_2\le \dots <s_N\le t_N}}\bigcup_{\substack{(j_i)_{i=1}^N\in\{1,\dots,M\}^N \\ (l_i)_{i=1}^N\in\{1,\dots,M\}^N}} \prod_{i=1}^N & E_{K_R}\left(x_{j_i},x_{l_i},\eta, t_i-s_i-1,t_i-s_i+1,\delta \right) \\ &\times E^{K^c}\left(x_{l_i},x_{j_{i+1}},\eta, s_i-t_{i-1}-1,s_i-t_{i-1}+1,\delta \right)
\end{align*}
where each $E_{K_R}\left(x_{j_i},x_{l_i},\eta, t_i-s_i-1,t_i-s_i+1,\delta \right)$ is a $E(x_{j_i},x_{l_i},\eta,t_i-s_i-1,t_i-s_i+1,\delta)$-set of chords contained in $K_R$ of maximal cardinality, i.e.,
\[\# E_{K_R}\left(x_{j_i},x_{l_i},\eta, t_i-s_i-1,t_i-s_i+1,\delta \right) = \mathcal{N}_{\mathcal{C},K_R}\left(x_{j_i},x_{l_i},t_i-s_i-1,t_i-s_i+1,\delta\right)\, ,\]
while each $E^{K^c}\left(x_{l_i},x_{j_{i+1}},\eta, s_i-t_{i-1}-1,s_i-t_{i-1}+1,\delta \right)$ is a $E^{K^c}(x_{l_i},x_{j_{i+1}},\eta, s_i-t_{i-1}-1,s_i-t_{i-1}+1,\delta)$-set of maximal cardinality, i.e.
\[\# E^{K^c}\left(x_{l_i},x_{j_{i+1}},\eta, s_i-t_{i-1}-1,s_i-t_{i-1}+1,\delta \right)=\mathcal{N}^{K^c}_{\mathcal{C}}\left(x_{l_i},x_{j_{i+1}},\eta, s_i-t_{i-1}-1,s_i-t_{i-1}+1,\delta\right).\]
\color{black}

We construct $f : \mathcal{P}(K,K_R,\alpha,T,T+S)\to\mathcal E$ as follows.
With the above decomposition, we associate to the orbit $\gamma$ with an arbitrary choice of origin such that $\gamma(0)\in\partial K$ the following family of $2N$ chords: the
$N$ chords of respective lengths $\tau_i-\sigma_i\in \left[t_i-s_i-1, t_i-s_i+1\right]$ from $a_i\in\partial K$ to $b_i\in\partial K$ and the $N$ chords of respective lengths $\sigma_{i}-\tau_{i-1}\in\left[s_i-t_{i-1}-1, s_i-t_{i-1}+1\right]$ from $b_i$ to $a_{i+1}$ outside $K$.
Each $a_i$ belongs to some ball $B(x_{j_i},\eta)$, while each $b_i$ belongs to some ball $B(x_{l_i},\eta)$. In particular, each $\gamma([\sigma_i,\tau_i])$ defines a point in $E_{K_R}\left(x_{j_i},x_{l_i},\eta, t_i-s_i-1,t_i-s_i+1,\delta \right)$, while each $\gamma([\tau_{i},\sigma_{i+1}])$ defines a point in $E^{K^c}\left(x_{l_i},x_{j_{i+1}},\eta, s_i-t_{i-1}-1,s_i-t_{i-1}+1,\delta\right)$.
Note that, if $R\geq 3b/2$, for every $1\le i\le N$, we have $s_i-t_{i-1}\geq  \sigma_i-\tau_{i-1}-1 \geq 1$. This shows that the image of $f$ is indeed contained in $\mathcal E$.

{\bf Step 2.d.} The map $f$ is almost injective.

For every $\gamma_0\in \mathcal{P}(K,K_R,\alpha,T,T+S)$, we need to bound the cardinality of  $f^{-1}(f(\gamma_0))$.

Assume that $\gamma,\tilde\gamma$ are in $f^{-1}(f(\gamma_0))$, i.e. they are orbits in $\mathcal{P}(K,K_R,\alpha,T,T+S)$ such that $f(\gamma)=f(\tilde\gamma)$.
By definition of $f$, these orbits lead to   the same $2N$-tuple of integers $(s_1,t_1,\dots,s_N,t_N)$, and the same $2N$-tuple of balls $B(x_{j_i},\eta)$ and $B(x_{l_i},\eta)$, and the same $2N$ tuple of chords.

We will need refined informations. Therefore, we partition $f^{-1}(\gamma_0)$ accordingly to the precise length of $\gamma$ and to the precise length of the chords.  In other words, we assume that $|\ell(\gamma)-\ell(\tilde\gamma)|\le \bar\tau_0$, where $\bar\tau_0$ is given by Lemma \ref{lemma on same po}. Denote by $(\sigma_i)_{i=1,\dots, N}, (\tau_i)_{i=0,\dots, N-1}$, resp.   $(\tilde\sigma_i)_{i=1,\dots, N}, (\tilde\tau_i)_{i=0,\dots, N-1}$, the
lengths of chords in the construction  of $f(\gamma)$ (resp. $f(\tilde \gamma)$).
We also assume that for every $1\le i\le N$,  we have $\left\vert \sigma_i-\tilde\sigma_i\right\vert <\frac{\delta}{b}$
and for every  $0\le i\le  N-1$, we have $\left\vert \tau_i-\tilde\tau_i\right\vert <\frac{\delta}{b}$, where  $b$ is the constant in \eqref{eqn:minoration}
Recall that $\tau_0=\tilde\tau_0=0$.

We will show that for every $s\in[0,T]$, $d(\gamma(s),\tilde\gamma(s))\leq \overline{\epsilon}$. Thanks to lemma \ref{lemma on same po}, we will deduce the desired result $\gamma=\tilde\gamma$.
To do so, we divide $[0,T]$ into sub-intervals where we control the distance between $\gamma$ and $\tilde\gamma$.
For $i=1,\dots, N$, set
\[
S_i = [\sigma_i,\tau_i]\cap [\tilde\sigma_i,\tilde\tau_i] = [\max(\sigma_i,\tilde\sigma_i),\min(\tau_i,\tilde\tau_i)]
\]
and for   $i=0, \dots, N-1$, set
\[
B_i = [\tau_i,\sigma_{i+1}]\cap [\tilde\tau_i,\tilde\sigma_{i+1}] = [\max(\tau_i,\tilde\tau_i),\min(\sigma_{i+1},\tilde\sigma_{i+1})].
\]
As every large excursion has  length $\tau_i-\sigma_i$ (resp $\tilde \tau_i-\tilde \sigma_i$ at least $2R/b$, for $R> \delta/2$, we know that $B_i\neq\emptyset$. Nonetheless, $S_i$ may be empty if $\sigma_i\leq\tau_i<\tilde\sigma_i \leq\tilde\tau_i$ or $\tilde\sigma_i\leq\tilde\tau_i<\sigma_i \leq\tau_i$.
In this case, we have $\max(\tau_i,\tilde\tau_i)-\min(\sigma_i,\tilde\sigma_i )\leq 2\frac{\delta}{b}$.
Let $(R_j)_{j\in J}$ be the connected components of
\[ [0,T]\setminus\left(\bigcup_{i=1}^N S_i \cup \bigcup_{i=0}^{N-1} B_i \right).\]
By construction, $R_j$ can be of the following forms
\begin{align*}
    R_j&=(\min(\tau_i,\tilde\tau_i), \max(\tau_i,\tilde\tau_i) )\\
    R_j&=(\min(\sigma_{i},\tilde\sigma_{i}),\max(\sigma_i,\tilde\sigma_i))\\
    R_j &= (\min(\sigma_{i},\tilde\sigma_{i}),\max(\tau_i,\tilde\tau_i) )
\end{align*}
and the last case can only happen if $S_i=\emptyset$. We then always have $\mathrm{Leb}(R_j)\leq 2\frac{\delta}{b} \leq 1$.

Let $s\in[0,T]$, then
\begin{itemize}
    \item if $s\in S_i$, as $f(\gamma)=f(\tilde \gamma)$, the points $\gamma(\sigma_i)$ and $\tilde \gamma (\tilde \sigma_i)$ belong to the same ball $B(x_{j_i},\eta)$ and lead to the same chord, so that  $d(\gamma(s),\tilde\gamma(
    s))\leq 2\delta + b\vert \sigma_i-\tilde\sigma_i\vert <3\delta\leq\overline{\epsilon}$
    \item if $s\in B_i$, by the same argument, we have
    $d(\gamma(s),\tilde\gamma(s))\leq 2\delta + b\vert \tau_i-\tilde\tau_i\vert <3\delta\leq\overline{\epsilon}$
    \item if $s\in R_j = (r_j^-,r_j^+)$, as $r_j^+-r_j^-\leq 1$, we have $d(\gamma(
    s),\tilde\gamma(s))\leq lip(\phi) (\gamma(
    r_j^-),\tilde\gamma(r_j^+))< lip(\phi) 3\delta \leq\overline{\epsilon}$.
\end{itemize}

By Lemma~\ref{lemma on same po}, we conclude that $\gamma=\tilde\gamma$.
Therefore, we proved that  $f(\gamma)=f(\gamma')$, $|\ell(\gamma)-\ell(\gamma')|\leq\bar\tau_0$, $\left\vert \sigma_i-\tilde\sigma_i\right\vert <\frac{\delta}{b}$ and $\left\vert \tau_i-\tilde\tau_i\right\vert <\frac{\delta}{b}$ for every $i$ implies $\gamma=\gamma'$.
As two elements $\gamma$ and $\gamma'$ in $\mathcal{P}(K,K_R,\alpha,T,T+S)$ with $f(\gamma)=f(\gamma')$ are such that $|\ell(\gamma)-\ell(\gamma')|\leq S$, $\left\vert \sigma_i-\tilde\sigma_i\right\vert \leq 2$ and $\left\vert \tau_i-\tilde\tau_i\right\vert \leq 2$ for all $i$, we have at most $ \left\lceil\frac{S}{\bar\tau_0}\, \left(\frac{2b}{\delta}\right)^{2 N}\right\rceil$ elements in $f^{-1}(f(\gamma))$.
We deduce that
\[ \# \mathcal{P}(K,K_R,\alpha,T,T+S)\le \left\lceil\frac{S}{\bar\tau_0}\, \left(\frac{2b}{\delta}\right)^{2\hat\nu }\right\rceil \#\mathcal E\, .\]

{\bf Step 3.} Bound on $\#\mathcal E$.

{\bf Step 3.a.} Bound on the number of chords.
By the initial choice of $\eta,\delta$ and by \eqref{h corde K senza delta e eta}, for every $x_j,x_l$ there exists $T_{j,l}>0$ such that for every $T\geq T_{j,l}$ we have
\[
\mathcal{N}^{K^c}_{\mathcal{C}}(x_j,x_l,\eta, T,T+2,\delta)\leq e^{(h_{\mathcal{C}}^{K^c}(\varphi)+\epsilon)T}\, .
\]
Let $T_{\max}=\max_{j,l}T_{j,l}+1>0$.
Note that $T_\mathrm{max}$ does not depend on $\alpha$.

Choose $R\ge \max\left(\frac{ T_{\max}b}{2}, R_{\min},\delta/2, 3b/2\right)$, where $R_{\min}$ is given by Proposition~\ref{cordes et orbite periodique version compacte} applied with $K$, $\delta$.
Since the length of every large excursion satisfies $s_i-t_{i-1}\geq \frac{2R}{b}$ for every $i$, we have   $s_i-t_{i-1}-1\geq T_{\max}$ for every $i=1,\dots, N$.
Thus, for every $j,l=1,\dots, M$, we obtain

\[
\mathcal{N}_{\mathcal{C}}^{K^c}(x_l,x_j,\eta, s_i-t_{i-1}-1,s_i-t_{i-1}+1,\delta)\leq e^{(h_{\mathcal{C}}^{K^c}(\varphi)+\epsilon)(s_i-t_{i-1})}\, .
\]
Let
\[h =
\limsup_{T\to\infty}\frac{1}{T}\log \#\{\gamma\in\mathcal P_{K}(T,T+S), \gamma\subset K_{R+1}\} <\infty\, ,\]
where we use Lemma \ref{lemme:entropie_gur_compact} to guarantee that $h$ is finite.
By the same argument as used in the proof of Theorem \ref{theo:Gurevic}, the above limsup does not depend on $S$. Therefore,  with $S=3\tau_K$, we also have
\[
h=\limsup_{T\to\infty}\frac 1 T \log\left( (T+3\tau_K)\#\{ \gamma\in\mathcal P_{K}(T,T+3\tau_K), \gamma\subset K_{R+1} \}\right)<\infty\, .
\]
Thus there exists $D'>0$ such that  for every $T>0$
\begin{equation}\label{aaa brutta}
(T + 3\tau_K)\#\{\gamma\in\mathcal P_{K}(T,T+3\tau_K), \gamma\subset K_{R+1}\} \leq D'e^{(h+\epsilon)T}\, .
\end{equation}

By Proposition~\ref{cordes et orbite periodique version compacte}, there exists $D>0$ and $\sigma>0$ such that when  $t_i-s_i$  is large enough
\begin{align*}
\mathcal{N}_{\mathcal{C},K_R}(x_{j_i},x_{l_i},\eta,t_i-s_i-1,t_i-s_i+1,\delta)
\leq& \mathcal{N}_{\mathcal{C},K_R}(x_{j_i},x_{l_i},\eta,t_i-s_i-1,t_i-s_i-1+3\tau_K,\delta)\\
\leq& D (t_i-s_i + \sigma + 3\tau_K)\\ &\#\{\gamma\in\mathcal P_{K}(t_i-s_i+\sigma,s_i+\sigma+3\tau_K), \gamma\subset K_{R+1}\}\\
\leq& D D' e^{(t_i-s_i+\sigma)(h+\epsilon)}.
\end{align*}
Up to increasing $D$, we may   assume that the second inequality is also satisfied for small $t_i-s_i$
Then we obtain, for all $t_i-s_i$ 
\[
\mathcal{N}_{\mathcal{C},K_R}(x_{j_i},x_{l_i},\eta,t_i-s_i-1,t_i-s_i+1,\delta)
\leq D D' e^{(t_i-s_i+\sigma)(h+\epsilon)} \leq D_1 e^{(t_i-s_i)(h+\epsilon)}.
\]

{\bf Step 3.b.} Bound on $\#\mathcal E$.
From the previous step and by points \ref{punto 3} and \ref{punto 4} in Step 2.b, we have
\begin{align*}
\# \prod_{i=1}^N E_{K_R}(x_{j_i},x_{l_i},\eta, t_i-s_i-1,& t_i-s_i+1,\delta )\times E^{K^c}(x_{l_i},x_{j_{i+1}},\eta, s_i-t_{i-1}-1,s_i-t_{i-1}+1,\delta ) \\
\le
& D_1^N \exp\left((h_{\mathcal{C}}^{K^c}(\phi)+\epsilon)\sum_{i=1}^N(s_i-t_{i-1}) + (h+\epsilon)\sum_{i=1}^N (t_i-s_i)\right)\\
\le & D_1^N  \exp\left(h_{\mathcal{C}}^{K^c}(\varphi)(T+2\hat\nu  ) + h(\alpha T+2\hat\nu ) + \epsilon (T(1+\alpha)+4\hat\nu )\right).
\end{align*}
Therefore $\# \mathcal E$ is bounded by
\begin{equation*}\label{bound degolas}
 \sum_{N=1}^{\hat\nu }  \sum_{\substack{(t_1,\dots, t_N)\in\mathbb{N}^N \\ (s_1,\dots,s_N)\in\mathbb{N}^N\\ \vert t_N-T\vert\le S+1 \\ 0<s_1\le t_1<s_2\le\dots<s_N\le t_N}}  M^{2N}D_1^N \exp\left(h_{\mathcal{C}}^{K^c}(\varphi)(T+2\hat\nu ) + h(\alpha T+2\hat\nu ) + \epsilon (T(1+\alpha)+4\hat\nu )\right)\, .
\end{equation*}

As in \cite[Lemma 5.5]{GST}, the number of terms in the second sum in the previous equation is bounded from above by the number of ordered integer decompositions of $T+S+1$ of length $2N$, ie $\binom{T+S+1}{2N}$.
As $N\leq \hat\nu $ (as soon as $T\ge S/\alpha$) and $1+S\leq \hat\nu $ (as soon as $T\geq R(S+1)/b\alpha$)
we have
\[
\binom{T+S+1}{2N}\le \binom{T+\hat\nu }{ 2N}.
\]\color{black}

Therefore we obtain
\[\#\mathcal E\leq (M^{2}D_1)^{\hat\nu }\hat\nu \binom{T+\hat\nu }{2\hat\nu }\, \exp\left(h_{\mathcal{C}}^{K^c}(\varphi)(T+2\hat\nu ) + h(\alpha T+2\hat\nu ) + \epsilon (T(1+\alpha)+4\hat\nu )\right)\, .\]

{\bf Step 4.} Conclusion.
From the previous step (recall that $\hat\nu  = b\alpha T/R$), we obtain
\begin{align*}
\limsup_{T\to\infty}\dfrac 1 T \log \#\mathcal{P}(K,K_R,\alpha,T,T+S)\le& h_{\mathcal{C}}^{K^c}(\phi)\left(1+2\dfrac{b\alpha}{R}\right) + h\left(\alpha +2\dfrac{b\alpha}{R}\right) + \epsilon \left(1+\alpha +4\dfrac{b\alpha}{R}\right)\\
&+ \dfrac{b\alpha}{R}\left[\log(M^2D_1)+2\log \left(\frac{b}{\delta}\right)\right] \\
&+ \limsup_{T\to\infty}\frac 1 T \log\left(\binom{T+\hat\nu }{2\hat\nu }\right)\, .
\end{align*}
We then get
\[
\limsup_{T\to\infty}\dfrac 1 T \log \#\mathcal{P}(K,K_R,\alpha,T,T+S)\leq h_{\mathcal{C}}^{K^c}(\phi)\left(1+2\dfrac{b\alpha}{R}\right)+\epsilon  + \psi(\alpha)\, ,
\]
with
\[
\psi(\alpha)=\alpha h + \alpha\epsilon+ \frac{b\alpha}{R}\left[4\epsilon + 2h+  \log(M^2D_1)+2\log \left(\frac{2b}{\delta}\right)\right]+\limsup_{T\to\infty}\frac 1 T \log\left(\binom{T+\hat\nu }{2\hat\nu } \right)\, .
\]
As $\binom{n}{k}\le \left( \frac{n e }{k}\right)^k$, we have
\[
\frac 1 T \log\left(\binom{T+\hat\nu}{2\hat\nu}\right) \le 2 b \frac{\alpha}{R}\log\left(\frac{e(1+b\frac{\alpha}{R})}{2b \frac{\alpha}{R}}\right)\, ,
\]
so $\limsup_{T\to\infty}\frac 1 T \log\left(\binom{T+\hat\nu }{2\hat\nu } \right)$ converges to $0$ as $\alpha\to 0$. This, together with the fact that $h$ is finite, proves that $\psi$ converges to $0$ as $\alpha\to 0$.
\end{proof}


\section{Subadditivity properties}\label{sec:subadditivity}

In Section~\ref{section_periodic_measures}, we construct the probability measure $m_\mathrm{max}$ that is the candidate to satisfy the conclusion of Theorem \ref{theo:main}. The existence of such a measure is implied by the $h_{\mathrm{Gur}}$-strongly positive recurrent hypothesis (see Definition~\ref{def:SPR})\,: $h_{\mathrm{Gur}}^\infty(\phi)<h_{\mathrm{Gur}}(\phi)$.
The SPR assumption is a sufficient but \textit{a priori} not necessary condition for the existence of such a measure.

We go on with subadditivity statements, Propositions~\ref{prop cirm 1} and \ref{prop cirm 2}, that will be crucial in the proof of Theorem~\ref{theo:main}.
Proposition~\ref{prop cirm 1}  does not require either the construction of the measure nor the SPR assumption $h_{\mathrm{Gur}}^\infty(\phi)<h_{\mathrm{Gur}}(\phi)$, and could have been proven in Section~\ref{sec:entropies}. However, both statements are more relevant together: this is the reason why it is stated and proven here.  Proposition \ref{prop cirm 2} is much more subtle, and requires the existence of the measure $m_{\mathrm{max}}$. The end of the section is devoted to the proof of Proposition~\ref{prop cirm 2}.


\subsection{Construction of the measures \texorpdfstring{$m_\infty$}{TEXT} and \texorpdfstring{$m_{\mathrm{max}}$}{TEXT} through periodic measures}\label{construction-mesure}
\label{defi SPR}\label{section_periodic_measures}

For every periodic orbit
$\gamma=((\phi_t(x))_{t\in\mathbb{R}},T)\in\mathrm{Per}(\phi_t)$, let $\mu_\gamma$ be the $\phi$-invariant probability measure obtained by push-forward of the normalized Lebesgue-measure of the circle.

Assume that the $H$-flow $\phi$ is $h_{\mathrm{Gur}}$-strongly positively recurrent, i.e., $h_\mathrm{Gur}^\infty(\phi)<h_\mathrm{Gur}(\phi)$. By Definition~\ref{definition_gurevic_infini}, for every small $\alpha>0$, we can find  a  large compact set $K_0$, and $\epsilon>0$ small enough   so that
\[
\left|\limsup_{L\to\infty}\frac{1}{L}\log \#\mathcal P^\epsilon_{K_0}(L,L+5\tau_{K_0})
-h_{\mathrm Gur}^\infty(\phi)\right|\le \alpha\,.
\]
By
Theorem~\ref{theo:Gurevic}, we have that
\[
h_{\mathrm Gur}(\phi)=\lim_{L\to \infty}\frac{1}{L}\log\#\mathcal{P}_{K_0}(L,L+5\tau_{K_0})\, .
\]
By Definition~\ref{def:SPR}, choose $0<\alpha<h_{\mathrm{Gur}}(\phi)-h_{\mathrm{Gur}}^\infty(\phi)$.
It follows that
\begin{eqnarray*}
\limsup_{L\to\infty}\frac{1}{L}\log \#\mathcal P^\epsilon_{K_0}(L,L+5\tau_{K_0}) \,
&\le&
h_\mathrm{Gur}^\infty(\phi)+\alpha\nonumber\\ &{\bf <}& h_\mathrm{Gur}(\phi)\nonumber\\ &=& \lim_{L\to\infty}\frac{1}{L}\log \#\mathcal P_{K_0}(L,L+5{\tau_{K_0}})\,.\quad
\end{eqnarray*}
We get a strict inequality
\begin{equation}\label{ineg-stricte}
\limsup_{L\to\infty}\frac{1}{L}\log \#\{\gamma\in\mathcal P_{K_0}(L,L+5{\tau_{K_0}}),\ell(\gamma\cap K_0)<\varepsilon \ell(\gamma)\}
\,<\,
\lim_{L\to\infty}\frac{1}{L}\log \#\mathcal P_{K_0}(L,L+5{\tau_{K_0}})\, .
\end{equation}
Let
\begin{equation}\label{eqn:m_{K,L}}
m_{K_0,L}
=
\frac{1}{\# \mathcal{P}_{K_0}(L,L+5{\tau_{K_0}})}\sum_{\gamma\in \mathcal{P}_{K_0}(L,L+5{\tau_{K_0}})}\mu_\gamma
\end{equation}
be the invariant probability measure supported on the periodic orbits of $\mathcal P_{K_0}(L,L+5{\tau_{K_0}})$.

A sequence $(\mu_n)_{n\in\N}$ of finite Borel measures on $M$ converges to $\mu_\infty$ in the vague topology if, for every continuous map $f\colon M\to\mathbb{R}$ with compact support, one has $\int f\,d\mu_n\to\int f\,d\mu$, as $n\to+\infty$.
Recall that the set $\mathcal{M}^{\le 1}(\phi)$ of $\phi$-invariant measures $\mu$ such that $\mu(M)\leq 1$ is compact for the vague topology. Indeed, let $(\mu_n)_n$ be a sequence$(M,d)$ such that $\mu_n(M)\leq 1$.
As $M$ is a locally compact metric space, there exists a countable family of functions $(\phi_k)_{k\in\mathbb{N}}$ that are dense in the set $C_c(M)$ of continuous functions with compact support. Denote by $\Vert\phi_k\Vert$ the maximum over $M$ of the function $\phi_k$. By compactness of $[-1,1]$, one can find, for every fixed $k\in \N$, using a recursive definition, a strictly increasing map $\psi_k\colon\mathbb{N}\to\mathbb{N}$ such that, considering the subsequence $(\psi_1\circ \psi_2\circ \dots \psi_k(n))_{n\in\mathbb{N}}$, the quantity $\int_M \frac{\phi_k}{\Vert\phi_k\Vert}\,d\mu_{\psi_1\circ \psi_2\circ \dots \psi_k(n)}$ converges to some limit denoted by $\int_M\frac{\phi_k}{\Vert\phi_k\Vert}\,d\mu_\infty$. By Cantor's diagonal argument, defining $\psi(n):=\psi_1\circ \psi_{2}\circ \cdots\circ \psi_n(n)$, we get a subsequence $(\mu_{\psi(n)})_{n\in\mathbb{N}}$ such that for every $k\in\N$, $\int_M\frac{\phi_k}{\Vert\phi_k\Vert}\,d\mu_{\psi(n)} \to \int_M\frac{\phi_k}{\Vert\phi_k\Vert}\,d\mu_\infty$. By a standard density argument, for every $\phi\in C_c(M)$, we show that $\int_M\frac{\phi}{\Vert\phi\Vert}\,d\mu_{\psi(n)}$ is a Cauchy sequence, and therefore converges towards a limit, denoted by $\int_M\frac{\phi}{\Vert\phi\Vert}\,d\mu_\infty$. Thus, for every $\phi\in C_c(M)$, we obtain that
\[
\int_M \phi\, d\mu_{\psi(n)}=\Vert\phi\Vert \int_M \frac{\phi}{\Vert\phi\Vert}\, d\mu_{\psi(n)}\to \Vert\phi\Vert \int_M \frac{\phi}{\Vert\phi\Vert}\, d\mu_{\infty}=\int_M \phi\, d\mu_{\infty}\, .\]
If every $\mu_n$ is $\phi$-invariant, then  the limit $\mu_{\infty}$ of the converging subsequence  will be $\phi$-invariant. Moreover, we can check that $\mu_\infty(M)\le 1$.
Therefore, $(m_{K_0,L})_{L>0}$ has at least one accumulation point, and all its  accumulation points are $\phi$-invariant and of mass at most $1$.

\begin{prop}\label{limit measure SPR}
Let  $\phi:M\to M$ be a H-flow being $h_{\mathrm{Gur}}$-strongly positive recurrent. Then for every compact set $K_0$ big enough, any accumulation point $m_\infty$ of the family $(m_{K_0,L})_{L>0}$ when $L\to +\infty$ is a nonzero finite invariant measure.
\end{prop}

\begin{proof} Fix a compact set $K_0$ and $\epsilon>0$ such that \eqref{ineg-stricte} is satisfied. That is
\[
\limsup_{L\to \infty} \dfrac 1 L \log \#\mathcal{P}^\epsilon_{K_0}(L,L+5\tau_{K_0})<\lim_{L\to \infty}\dfrac 1 L \log\# \mathcal{P}_{K_0}(L,L+5\tau_{K_0})\,.
\]
Let $L_n\to \infty$ be a sequence such that $(m_{K_0,L_n})_n$ converges to $m_\infty$ in the vague topology. Let us show that
\[
\liminf_{L\to \infty} m_{K_0,L}(K_0)
\ge
\epsilon\,.
\]
Split $\mathcal P_{K_0}(L,L+5{\tau_{K_0}})$ into
\[
\mathcal P_{K_0}^{good}(L,L+5{\tau_{K_0}})=\{\gamma\in\mathcal P_{K_0}(L,L+5{\tau_{K_0}}),\,\,\ell(\gamma\cap K_0)\ge \epsilon \ell(\gamma)\}
\]
 and
\[
\mathcal P_{K_0}^{bad}(L,L+5{\tau_{K_0}})=
\mathcal P_{K_0}(L,L+5{\tau_{K_0}})\setminus \mathcal P^{good}_{K_0}(L,L+5{\tau_{K_0}})=\mathcal{P}_{K_0}^\epsilon(L,L+5\tau_{K_0})\,.
\]
Inequality (\ref{ineg-stricte}) ensures that  $\limsup_{L\to\infty}\frac{1}{L}\log\left(\frac{\#\mathcal P_{K_0}^{bad}(L,L+5\tau_{K_0})}{\#\mathcal P_{K_0} (L,L+5\tau_{K_0})}\right)\,<\,0$, so that
\begin{equation}\label{only-place-SPR}
\lim_{L\to+\infty}\frac{\#\mathcal  P_{K_0}^{bad}(L,L+5{\tau_{K_0}})}{\#\mathcal P_{K_0}(L,L+5\tau_{K_0})}
=0\,.
\end{equation}
We deduce easily that
\begin{eqnarray*}
m_{K_0,L}({K_0})
&=&
\frac{1}{\#\mathcal P_{K_0}(L,L+5{\tau_{K_0}})}\left(\sum_{\gamma\in\mathcal P_{K_0}^{good}(L,L+5{\tau_{K_0}})} \frac{\ell(\gamma\cap K_0)}{\ell(\gamma)}+\sum_{\gamma\in \mathcal P_{K_0}^{bad}(L,L+5{\tau_{K_0}})} \frac{\ell(\gamma\cap K_0)}{\ell(\gamma)}\right)\\
&\ge& \varepsilon \frac{\#\mathcal P_{K_0}^{good}(L,L+5{\tau_{K_0}})}{\#\mathcal P_{K_0}(L,L+5{\tau_{K_0}})} \, ;
\end{eqnarray*}
thus, we obtain
\[
\liminf_{L\to +\infty} m_{K_0,L}(K_0)
\ge
\varepsilon\,.
\]
For every continuous map with compact support $\phi\ge 1_{K_0}$, by definition of vague convergence, we get $\int\varphi dm_\infty=\lim_{n\to \infty} \int \varphi dm_{K_0,L_n}\ge \liminf_{n\to \infty} m_{K_0,L_n}(K_0)\ge \varepsilon$. Choosing a decreasing sequence $\phi_k$ of such maps, with $\lim_{k\to \infty}\phi_k=1_{K_0}$, and using the decreasing version of the monotone convergence theorem we obtain $m_\infty(K)=\lim \int\phi_k dm_\infty\ge \varepsilon$.
The proposition follows.
\end{proof}

\begin{notation}[Measures]\label{def:measures}
From now, we denote by   $m_\infty$   an arbitrary fixed accumulation point of the family $(m_{K_0,L})_{L>0}$, and by $m_\mathrm{max}$  the invariant probability measure obtained by renormalizing $m_\infty$.
\end{notation}

\begin{rema}\rm The $h_{\mathrm{Gur}}$-strongly positive recurrent assumption $h_{\mathrm{Gur}}^\infty(\phi)<h_{\mathrm Gur}(\phi)$ is natural, and leads to many natural examples, as shown in \cite{ST19, GST} in the case of geodesic flows, or in \cite{FSV2}. Moreover, it is stable under perturbations on compact sets.
\end{rema}


\subsection{The subadditivity statements}

We will need rigorous versions of the following heuristic statements.

First, given two chords of respective lengths $T_0$ and $T_1$, we can concatenate them, and use the closing lemma to get a periodic orbit of length roughly $T_0+T_1$.
This process is essentially injective: we obtain a lower bound on the number of periodic orbits of length $T_0+T_1$ in terms of the number of chords of lengths $T_0$ and $T_1$. This lower bound is clearly still valid when we add a linear term $T_0+T_1$ on the right. This is properly stated in Proposition \ref{prop cirm 1}.

Second, we wish to say that a periodic orbit of length  $T_0+T_1$  can be divided in two pieces of orbits of respective lengths $T_0$ and $T_1$, and therefore, bound the number of periodic orbits of length $T_0+T_1$ in terms of the numbers of chords of lengths $T_0$ and $T_1$, respectively. This is much more subtle, for two deep reasons. First, due to the lack of compactness, there is absolutely no reason that  an arbitrary periodic orbit  return in a compact set at the time $T_0$. The measure constructed in definition \ref{def:measures} in section \ref{section_periodic_measures} will be crucial in the argument, and allow us to say that ``most'', or more precisely a positive proportion, of periodic orbits of length $T_0+T_1$ come back to $K_0$. The second difficulty is that we need a strong subadditivity, with a linear term on the left\,: $T_0+T_1$ times the number of periodic orbits of length $T_0+T_1$ should be smaller than the number of chords of length $T_0$ times the number of chords of length $T_1$. This requires precise statements on the number of returns of a typical periodic orbits, see Lemmas~\ref{lemme_pre_coding_1} and \ref{lemme_pre_coding_2} and Proposition~\ref{prop cirm 2}.


\begin{prop}[Easy subadditivity]\label{prop cirm 1} Let $\phi\colon M\to M$ be a $H$-flow. Let $K_0$ be a compact set with nonempty interior.
For every compact set $K\supset K_0$, every $0<\delta<1$ and every $0<\eta<1$, there exist $T_{\mathrm{ min}}>0$, $S_1>0$ and $D_1>0$ such that for every $S\ge S_1$,  all $x_0,y_0,z_0$ in $K$, and all $T_0,T_1\ge T_{\mathrm{min}}$, we have\,:
\begin{align*}
    \mathcal{N}_C(x_0,y_0,\eta,T_0,T_0+5\tau_{K_0},\delta)\,\times \,& \mathcal{N}_C(y_0,z_0,\eta,T_1,T_1+5{\tau_{K_0}},\delta)\leq \\ & D_1\,\times\,(T_0+T_1+S)\,\times\, \#\mathcal{P}_{K_0}(T_0+T_1+S,T_0+T_1+S+5{\tau_{K_0}} )\, .
\end{align*}
\end{prop}



The following proposition is more difficult. It is stated under the assumption of existence of a nonzero measure $m_\mathrm{max}$ as in definition \ref{def:measures}. This assumption is satisfied as soon as the flow is $h_{\mathrm{Gur}}$-strongly positive recurrent, i.e., $h_\mathrm{Gur}^\infty(\phi)<h_\mathrm{Gur}(\phi)$, by Proposition \ref{limit measure SPR}.

\begin{prop}[Hard subadditivity]\label{prop cirm 2}Let $\phi\colon M\to M$ be a $H$-flow that is $h_{\mathrm{Gur}}$-strongly positive recurrent. Let $K_0$ be a compact set with nonempty interior that satisfies inequality (\ref{only-place-SPR}).
Let $L_n\to +\infty$ be a sequence such that the sequence of measures $(m_{K_0,L_n})_n$ converges in the vague topology to a non zero measure $m_\infty$.
Let $K$ be a compact subset such that $\overset{\circ}{K}\supset K_0$ and   $m_\infty(\inter{K})\geq \frac 3 4 m_\infty(M)$.
There exists $\epsilon_1>0$ such that, for all $0<\eta<\delta<\frac{\varepsilon_1}{4}$
there exist constants
$S_2>0$ and $D_2>0$ such that for every $T > 5\tau_{K_0}$, there exists $k_0\in \mathbb{N}$, such that for every integer $n\ge k_0$, for all quadruples of points $x_0$, $y_0$, $x_1$ and $y_1$ in $K$, and every $S\ge S_2$,
we have
\begin{align*}
   L_n\,\times\,\#\mathcal{P}_{K_0}(L_n, L_n+5{\tau_{K_0}})\,\leq  &D_2\,\times\, \mathcal{N}_C(x_0,y_0,\eta, T +S, T +S+5{\tau_{K_0}},\delta)\\ &\times\mathcal{N}_C(x_1,y_1,\eta, L_n-T +S, L_n-T +S+5{\tau_{K_0}},\delta)\,.
\end{align*}
\end{prop}

\begin{rema}\rm We could have considered periodic orbits of periods $T_0$ and $T_1$, and stated a variant of the above propositions involving respectively $\#\mathcal P _{K_0}(T_0,T_0+5{\tau_{K_0}})$ and $\#\mathcal P _{K_0}(T_1,T_1+5{\tau_{K_0}})$ instead of the numbers of chords on the left side of the   inequality of Proposition \ref{prop cirm 1}, and $\mathcal{P}_{K_0}(  T_0+S, T_0+S+5{\tau_{K_0}} )$, $\mathcal{P}_{K_0}(  L_n-T_0+S, L_n-T_0+S+5{\tau_{K_0}} )$  on the right side of the   inequality of Proposition \ref{prop cirm 1}.
\end{rema}


\subsection{Proof of Proposition~\ref{prop cirm 1}}

{\bf Step 1.} Choice of appropriate parameters and notations.

Let $K$, $K_0$, $\delta$ and $\eta$ be as in the statement of the Proposition.
Let $K_1=\overline{B(K,1)}$. The proof will follow from the application of
Lemma \ref{petal separe} with parameters $K_0\subset K_1$, with $\nu=\tau_{K_0}/2$, $\delta/4$ and  $N=2$. This lemma gives constants $\sigma>0$ and $T_{\mathrm{min}}>0$.
Set $S_1=2\sigma+10\tau_{K_0}$.
Let $x_0$, $y_0$ and $z_0$ be in $K$. Let $T_0,T_1\geq T_{\mathrm{min}}$.
Let $E_0$ and $E_1$ be respectively  a $E(x_0,y_0,\eta,T_0,T_0+5\tau_{K_0},\delta)$-set  and  a $E(y_0,z_0,\eta,T_1,T_1+5\tau_{K_0},\delta)$-set  of maximal cardinality. In particular
\[
\# E_0 = \mathcal{N}_{\mathcal{C}}(x_0,y_0,\eta,T_0,T_0+5\tau_{K_0},\delta)\quad\text{and}\quad \# E_1 =\mathcal{N}_{\mathcal{C}}(y_0,z_0,\eta, T_1,T_1+5\tau_{K_0},\delta)\, .
\]
Throughout the proof, we will see the elements in $E_0$ and $E_1$ as points in $B(x_0,\eta)$ and $B(y_0,\eta)$ or chords with their initial points in $B(x_0,\eta)$ and $B(y_0,\eta)$. In particular, the initial and final points of every chord is in $K_1$, since $\eta<1$.

{\bf Step 2.} Construction of a map $f$ from chords to periodic orbits.

As $T_0\ge T_{\mathrm{min}}$ and $T_1\ge T_{\mathrm{min}}$, for every pair of chords $\beta_0\in E_0$ and $\beta_1\in E_1$ with lengths $\ell(\beta_0)\ge T_0 $ and   $\ell(\beta_1)\ge T_1 $, and every $\hat S\ge \sigma$, Lemma~\ref{petal separe} provides a periodic orbit $\gamma$ with length in
\[
\left[\ell(\beta_0)+\ell(\beta_1)+2\hat S-\tau_{K_1}-\frac{\tau_{K_0}}{2},\ell(\beta_0)+\ell(\beta_1)+2\hat S+\tau_{K_1}+\frac{\tau_{K_0}}{2}\right]
\] that intersects the interior of $K_0$.
Recall that, since $K_0\subset K_1$, we have $\tau_{K_1}\le\tau_{K_0}$. Thus, the periodic orbit $\gamma$ has length in
\begin{equation*}
\left[\ell(\beta_0)+\ell(\beta_1)+2\hat S-\frac 3 2 \tau_{K_0},\ell(\beta_0)+\ell(\beta_1)+2\hat S+\frac 3 2 \tau_{K_0}\right]\, .
\end{equation*}
In particular, for
\[
\hat S=\frac 1 2 \left( S  + (T_0-\ell(\beta_0))+(T_1-\ell(\beta_1))+\frac 3 2 \tau_{K_0}\right)
\]
(notice that $\hat S \geq \sigma$ as $S\geq S_1=2\sigma +10\tau_{K_0}$, $T_0-\ell(\beta_0)\geq -5\tau_{K_0}$  and $T_1-\ell(\beta_1)\geq -5\tau_{K_0}$)
we have
\begin{equation*}
T_0+T_1+S\le \ell(\beta_0)+\ell(\beta_1)+2\hat S-\frac 3 2\tau_{K_0}
\end{equation*}
and
\begin{equation*}
\ell(\beta_0)+\ell(\beta_1)+2\hat S+\frac 3 2\tau_{K_0}\le  T_0+T_1+S+5\tau_{K_0}\,.
\end{equation*}
Then $\gamma\in \mathcal P_{K_0}(T_0+T_1+S, T_0+T_1+S+5\tau_{K_0})$.

Thus for every $S\ge S_1$, the above construction defines a map
\[
f:E_0\times E_1\to \mathcal{P}_{K_0}(T_0+T_1+S, T_0+T_1+S+5\tau_{K_0})\,.
\]
Observe that the above construction gives a parametrization of $\gamma$ with an origin $s_0$.
More precisely, by Lemma \ref{petal separe}, there exists $s_0\in \mathbb{R}$ and $\tau\in[\hat S-\tau_{K_1},\hat S+\tau_{K_1}]$ such that,
\begin{itemize}
    \item for all $s\in[0,\ell(\beta_0)]$, we have $d(\gamma(s_0+s),\phi_s(\beta_0))<\frac{\delta}{4}$,
    \item for all $s\in[0,\ell(\beta_1)]$, we have $d(\gamma(s_0+s+\ell(\beta_0)+\tau),\phi_s(\beta_1))<\frac{\delta}{4}$.
\end{itemize}

{\bf Step 3.} Bound on $f^{-1}(\gamma)$.

Consider $\beta_0,\beta'_0\in E_0, \beta_1,\beta'_1\in E_1$ such that $(\beta_0,\beta_1)\neq (\beta'_0,\beta'_1)$   and $f(\beta_0,\beta_1)=f(\beta'_0,\beta'_1) =\gamma$.
Let $s_0$ (resp. $s'_0$) be the origin of $\gamma$ from the construction $\gamma = f(\beta_0,\beta_1)$ (resp. $\gamma = f(\beta'_0,\beta'_1)$).
Assume as a first case that $\beta_0\neq \beta'_0$. Since they belong to a $E(x_0,y_0,\eta, T_0,T_0+5\tau_{K_0},\delta)$-set, there exist $0\leq u\leq T_0\le \min(\ell(\beta_0),\ell(\beta_1))$, such that \[
d(\phi_u(\beta_0),\phi_u(\beta'_0))
\geq
\delta.
\]
However, by construction, $\gamma$ satisfies
\[
d(\gamma(s_0+u),\phi_u(\beta_0))\leq \delta/4
\quad \mbox{and}\quad
d(\gamma(s'_0+u),\phi_u(\beta'_0))\leq \delta/4\, .
\]
Therefore
\[
d(\gamma(s_0+u),\gamma(s'_0+u))\geq \delta/2\, .
\]
As the flow satisfies~(\ref{eqn:minoration}), we have
\[
d(\gamma(s_0+u),\gamma(s'_0+u))\leq b|s'_0-s_0|\, ,
\]
so that
\[
|s'_0-s_0|\geq \frac{\delta}{2b}\,.
\]
As the length of $\gamma$ is at most $T_0+T_1+S+5\tau_{K_0}$, it follows that
\[
\# \{ (\beta'_0,\beta'_1) ,  f(\beta'_0,\beta'_1) = f(\beta_0,\beta_1) \text{ and } \beta'_0\neq \beta_0\}\leq \,(T_0+T_1+S+5\tau_{K_0})\times \frac{2b}{\delta}\, .
\]
Assume now that $\beta_0=\beta'_0$ and $\beta_1\neq \beta'_1$. The construction of the periodic orbit $\gamma$ gives us two constants $\tau$ and $\tau'$ in $[\hat S-\tau_{K_0},\hat S+\tau_{K_0}]$ and $[\hat S'-\tau_{K_0},\hat S'+\tau_{K_0}]$. Observe that, as $|\hat S-\hat S'|=\frac 1 2 \vert \ell(\beta_1)-\ell(\beta'_1)\vert \leq 5\tau_{K_0}/2$, we have $\vert \tau'-\tau\vert \leq 5\tau_{K_0}$.

Assume that $\vert \tau'-\tau\vert \leq \frac{\delta}{4 b}$. Since $E_1$ is a $E(y_0,z_0,\eta, T_1,T_1+5\tau_{K_0},\delta)$-set, there exists $0\leq u\leq T_1$ such that
\[
d(\phi_u(\beta_1),\phi_u(\beta'_1))\geq \delta\, .
\]
Arguing as in the first case, we obtain that
\[
d(\gamma(s_0+\ell(\beta_0)+\tau+u),\gamma(s_0'+\ell(\beta_0)+\tau'+u))\geq \frac{\delta}{2}\, ,
\]
and so, using~\eqref{eqn:minoration}, we get
\[
d(\gamma(s_0+\ell(\beta_0)+\tau+u),\gamma(s_0'+\ell(\beta_0)+\tau+u))\geq \frac{\delta}{2}-b\vert \tau'-\tau\vert\geq \frac{\delta}{4}\, .
\]
Using again~\eqref{eqn:minoration}, we deduce that
\[
\vert s_0'-s_0\vert \geq \frac{\delta}{4b}\,.
\]
We then conclude that
\[
\# \{ (\beta'_0,\beta'_1) ,  f(\beta'_0,\beta'_1) = f(\beta_0,\beta_1) \text{ and } \beta'_1\neq \beta_1\}\leq \,(T_0+T_1+S+5\tau_{K_0})\times \frac{4b}{\delta}\times 5\tau_{K_0} \times \frac{4b}{\delta}.
\]
Thus
\[
\# f^{-1}(\gamma)\leq \,(T_0+T_1+S+5\tau_{K_0})\times \left(\frac{2b}{\delta}+5\tau_{K_0}\left(\frac{4b}{\delta}\right)^2\right)\, .
\]
By choosing $D_1>0$ such that
\[(T_0+T_1+S+5\tau_{K_0})\times \left(\frac{2b}{\delta}+5\tau_{K_0}\left(\frac{4b}{\delta}\right)^2\right)\le D_1 (T_0+T_1+S)\, ,\]
this concludes the proof of Proposition~\ref{prop cirm 1}.

\subsection{Proof of Proposition \ref{prop cirm 2}}

\subsubsection{Strategy of the proof}

As usual, we want to construct an almost injective map, specifically a map from periodic orbits of period $L_n$ intersecting $K_0$ to chords of length (almost) $T$ and endpoints $x_0, y_0$ and chords of length (almost) $L_n-T$ and endpoints $x_1,y_1$.

For every periodic orbit of length $L$ that intersects $K_0$, the naive idea is to consider a piece of orbit of length $T$ starting and ending in $K_0$ and to cut the periodic orbit at the beginning and the end of this piece of orbit.
We then obtain two arcs, one of length $T$ and one of length $L-T$.
Using transitivity and the shadowing property, we obtain arcs from $x_0$ to $y_0$ and from $x_1$ to $y_1$.

Unfortunately, there is no guarantee that such arcs exist. The difficulty of the proof consists in finding enough periodic orbits on which enough points of $K_0$ return to $K_0$ after a time $T$.
To find these arcs, we will choose a compact set $K\supset K_0$ with large measure $m_\mathrm{max}(K)>3/4$ so that for every $T>0$, $m_\mathrm{max}(K\cap\phi_{-T}(K))>1/2$. Set $A=K\cap \phi_{-T}(K)$. We want to find enough intersections between $A$ and $\mathcal{P}_{K_0}(L,L+5\tau_{K_0})$.

As $m_\infty(A)>0$, for $L$ large enough, a positive proportion of periodic orbits in $\mathcal P_{K_0}(L,L+5\tau_{K_0})$ spend a positive proportion of their time in $A$.
This is proved in Lemma~\ref{lemme_pre_coding_1}.

In Lemma~\ref{lemme_pre_coding_2}, we prove that if $\mu_\gamma(A)\geq\alpha'>0$ then we control the number of points in $A$ for some appropriate discretization of $\gamma$ of length $T_0$.

Lemmas~\ref{lemme_pre_coding_1} and \ref{lemme_pre_coding_2} give us Proposition~\ref{prop_pre_coding} where we prove that there are indeed enough periodic orbits with enough point in $K_0$.

\subsubsection{How often a typical periodic orbit comes back}


The existence of a nonzero measure $m_\mathrm{max}$ leads to the first important lemma.

\begin{lemm}\label{lemme_pre_coding_1}
Let $\phi\colon M\to M$ be a flow.
Let $K_0$ be a compact set with nonempty interior.
Let $A$ be a Borel set, $L>0$ and $\alpha>0$
such that  $m_{K_0,L}(A)\ge \alpha>0$.
For every $0<\alpha'<\alpha$,
we have
\[
\#\,\left\{\gamma\in \mathcal P_{K_0}(L,L+5\tau_{K_0}): \ \mu_\gamma(A)\geq \alpha' \right\}
\,\geq\,
\frac{\alpha-\alpha'}{1-\alpha'}\times \,\#\mathcal P_{K_0}(L,L+5\tau_{K_0})\, .
\]
\end{lemm}

\begin{proof}
Let $n =\#\left\{\gamma\in \mathcal P_{K_0}(L,L+5\tau_{K_0}), \mu_\gamma(A)< \alpha' \right\}$.
Then
\[
0< \alpha
\leq
m_{K_0,L}(A)
\leq
\frac{\alpha' n + \# \mathcal P_{K_0}(L,L+5\tau_{K_0})-n}{\# \mathcal P_{K_0}(L,L+5\tau_{K_0})}\,.
\]
Therefore
\[
\alpha\# \mathcal P_{K_0}(L,L+5\tau_{K_0})
\leq
\# \mathcal P_{K_0}(L,L+5\tau_{K_0})-(1-\alpha')n\,.
\]
so that $n\leq\frac{1-\alpha}{1-\alpha'}\,\# \mathcal P_{K_0}(L,L+5\tau_{K_0})$. The result follows.
\end{proof}

\begin{lemm}\label{lemme_pre_coding_2}
Let $\phi\colon M\to M$ be a flow.
Let $K_0$ be a compact set with nonempty interior.
Let $A$ be a Borel set and $\alpha'>0$.  For all $L>T_0>5\tau_{K_0}>0$, for every periodic orbit  $\gamma\in \mathcal{P}_{K_0}(L,L+5\tau_{K_0})$   such that  $\mu_\gamma(A)\geq \alpha'$, for any parametrization $\gamma:[0,\ell(\gamma)]\to M$ of $\gamma$ there exists a real number $s\in [0,T_0]$ such that
\[
\#\left\{i\in\left\{0,\dots,\left\lfloor \frac{L}{T_0}\right\rfloor-1\right\},\gamma(s+iT_0)\in A\right\}
\geq
\left\lfloor  \alpha' \left\lfloor \frac{L}{T_0}\right\rfloor\right\rfloor-3 \, .
\]
\end{lemm}

\begin{proof} Set $k=\left\lfloor \frac{L}{T_0}\right\rfloor$ and $k'=\left\lfloor \alpha'\left\lfloor \frac{L}{T_0}\right\rfloor\right\rfloor-3$.
Set $L = kT_0+r$ with $0\leq r< T_0$. By assumption, since $\mu_\gamma(A)\geq \alpha'$, we have
$\ell(\gamma\cap A)\ge \alpha'\ell(\gamma)$. As $L\le \ell(\gamma)\le L+5\tau_{K_0}$, we deduce that
\[
\mathrm{Leb}\left(\{u\in [0,kT_0] , \gamma(u)\in A\}\right)\ge \alpha'\ell(\gamma)-r-5\tau_{K_0}\ge \alpha'kT_0-T_0-
5\tau_{K_0}\,.
\]
For every $s\in [0,T_0]$, let $J(s)=\{0\le i\le \left\lfloor \frac{L}{T_0}\right\rfloor-1 , \gamma(s+iT_0)\in A\}$.
Suppose that for every $s\in [0,T_0]$, $\# J(s)\le l$, for some integer $0\le l\le k$.

We now prove that
\begin{equation}\label{eq:fubini}
\mathrm{Leb}\left(\{u\in [0,kT_0] , \gamma(u)\in A\}\right) \le lT_0\,.
\end{equation}
Write $u=s+iT_0$, with $s\in [0,T_0]$ and $0\le i\le k-1$.
We have
\begin{align*}
\{u\in [0,kT_0] , \gamma(u)\in A\}
&=
\bigsqcup_{i=0}^{k-1}\{s+iT_0,s\in[0,T_0] ,  \gamma(s+iT_0)\in A\} \\
&=
\bigsqcup_{J\in\mathcal P(\{0,\cdots,k-1\})}\left(\{s\in[0,T_0],  J(s)=J\}+JT_0\right)\,.
\end{align*}
As, for all $J\in \mathcal P(\{0,\cdots,k-1\})$, we have
\[\mathrm{Leb}\left(\{s\in[0,T_0], J(s)=J\}+JT_0\right)=\#J\times \mathrm{Leb}\left(\{s\in[0,T_0], J(s)=J\}\right)\, ,\]
as the subsets $\{s\in[0,T_0], J(s)=J\}$ form a measurable partition of $[0,T_0]$ and as $\# J(s)\leq l$, the inequality (\ref{eq:fubini}) follows.

Therefore, $\alpha'kT_0-T_0-5\tau_{K_0}\le lT_0$, and, since $5\tau_{K_0}<T_0$, we have $l\ge \alpha'k-2$. For $k'=
\left\lfloor \alpha'k\right\rfloor-3$, we have   $k'<\alpha'k-2$. Therefore, $J(s)$ is not bounded by $k'$ for all $s$ and there exists $s\in\mathbb R$ such that
\[
\#\left\{i\in\{0,\dots,k-1\},\ \gamma(s+iT_0)\in A\right\}\ge k'\,.
\]
\end{proof}

In Subsection~\ref{subsection def orbit periodique}, we have defined $\mathcal P_{K_0}(L,L+5\tau_{K_0})$ as the set of periodic orbits $((\phi_t(x))_{t\in\mathbb{R}},T)$, where $(x,T)$ is a periodic point (that is $\phi_T(x)=x$). Let $\mathcal P'_{K_0}(L,L+5\tau_{K_0})$ be the set of primitive periodic orbits in $\mathcal P_{K_0}(L,L+5\tau_{K_0})$, i.e., the set of periodic orbits $((\phi_t(x))_{t\in\mathbb{R}},T)$, where $(x,T)$ is a periodic point and $\phi_t(x)\neq x$ for every $t\in(0,T)$.
The following statement will follow from the fact that:
\begin{itemize}
    \item almost all periodic orbits are simple;
    \item if $m_{K_0,L}(A)\ge \alpha$, then a positive proportion of periodic orbits $\gamma$ satisfy $\mu_\gamma(A)\geq \alpha'$ (Lemma~\ref{lemme_pre_coding_1});
    \item if $\mu_\gamma(A)\geq \alpha'$, then there exists an appropriate discretization of $\gamma$ with step $T_0$ having a positive proportion of points in $A$ (Lemma~\ref{lemme_pre_coding_2}).
    \end{itemize}

\begin{prop}\label{prop_pre_coding}
Let $\phi\colon M\to M$ be a $H$-flow such that $h_{\mathrm{Gur}}(\phi)>0$.
Let $K_0$ be a compact set with nonempty interior.
Let $A$ be a Borel set, $0<\alpha'<\alpha$ and $L>0$ such that  $m_{K_0,L}(A)\ge \alpha>0$.  There exists $N>0$ such that for all $L>T_0>5\tau_{K_0}$, with $\frac{L}{T_0}\ge N$, we have
\[
\#\left\{\gamma\in \mathcal P'_{K_0}(L,L+5\tau_{K_0}),\,
\exists s\in\mathbb [0,T_0],\,
\#\Big\{i\in\{0,\dots, \lfloor L/T_0\rfloor-1\},  \gamma(s+iT_0)\in A\Big\}\geq \frac{\alpha'}{2}\frac{L}{T_0}\right\}\]
\[
\geq 0.99\times\frac{\alpha-\alpha'}{1-\alpha'}\times\,
\#\mathcal P_{K_0}(L,L+5\tau_{K_0})\, .
\]
\end{prop}

\begin{proof} Choose $N$ large enough so that, for $\frac{L}{T_0}\ge N$, we have $k'=\left\lfloor \alpha'\left\lfloor \frac{L}{T_0}\right\rfloor\right\rfloor-3\ge \frac{\alpha'}{2}\frac{L}{T_0}$.
First, observe that a non primitive periodic orbit $\gamma$ is a multiple of a primitive periodic one, that has length at most $\ell(\gamma)/2$.
It follows that the number of non-simple periodic orbits in $\mathcal{P}_{K_0}(L,L+5\tau_{K_0})$ is bounded above by $\mathcal{P}_{K_0}\left(\frac{L+5\tau_{K_0}}{2}\right)$.
By Corollary~\ref{corollaire_limite_nulle_quotient_P_K} (whose hypothesis are satisfied because $\phi$ is a $H$-flow with $h_{\mathrm{Gur}}(\phi)>0$), it follows that
\[
\lim_{L\to+\infty}\dfrac{\#\mathcal P'_{K_0}(L,L+5\tau_{K_0})}{\#\mathcal P_{K_0}(L,L+5\tau_{K_0})}=1\, .
\]
Up to choose a bigger $N$, we can assume that,
if $L> NT_0>N\, 5\tau_{K_0}$, we have
\[
\dfrac{\#\mathcal P'_{K_0}(L,L+5\tau_{K_0})}{\#\mathcal P_{K_0}(L,L+5\tau_{K_0})}\ge 1-0.01\times\frac{\alpha-\alpha'}{1-\alpha'}\, .
\]
Let $n_\gamma(s)=\#\{i\in\{0,\dots,\lfloor L/T_0\rfloor -1\} ,\gamma(s+iT_0)\in A\}$.
Now, by lemmas \ref{lemme_pre_coding_1} and \ref{lemme_pre_coding_2}, we have the inequalities
\begin{align*}
\#\left\{\gamma\in \mathcal{P}'_{K_0}(L,L+5\tau_{K_0}),\exists s\in\R, n_\gamma(s)\ge\frac{\alpha'}{2}\frac{L}{T_0}\right\}
\ge&
\#\left\{\gamma\in \mathcal{P}'_{K_0}(L,L+5\tau_{K_0}),\exists s\in\R, n_\gamma(s)\ge k'\right\} \\
\ge&
\#\left\{\gamma\in \mathcal{P}'_{K_0}(L,L+5\tau_{K_0}),\,\mu_\gamma(A)\ge \alpha'\right\}\\
\ge&  \#\left\{\gamma\in \mathcal{P}_{K_0}(L,L+5\tau_{K_0}),\,\mu_\gamma(A)\ge \alpha'\right\}\\
&- 0.01\times\frac{\alpha-\alpha'}{1-\alpha'}\,\times\,\#\mathcal{P}_{K_0}(L,L+5\tau_{K_0})\\
\ge&  \frac{\alpha-\alpha '}{1-\alpha'}\,\times\,\#\mathcal{P}_{K_0}(L,L+5\tau_{K_0})\\
&- 0.01\times\frac{\alpha-\alpha'}{1-\alpha'}\,\times\,\#\mathcal{P}_{K_0}(L,L+5\tau_{K_0})\\
=& 0.99\times\frac{\alpha-\alpha'}{1-\alpha'}\,\times\,\#\mathcal{P}_{K_0}(L,L+5\tau_{K_0})\, .
\end{align*}
\end{proof}

\subsubsection{Proof of Proposition~\ref{prop cirm 2}}

Recall that $m_\infty$ is the limit, in the vague topology, of a sequence $(m_{K_0,L_n})_{n\in \mathbb{N}}$ when $L_n\to\infty$, and $m_\mathrm{max}$ is the probability measure obtained by renormalization of $m_\infty$.

{\bf Step 1.} Setting the parameters.
Choose some large compact set $K$ such that $ \overset{\circ}{K}\supset K_0$ and $m_\mathrm{max}(\inter{K})>3/4$. In particular, for every $T >0$, we obtain that $m_\mathrm{max}(K\cap \phi_{-T }(K))>1/2$.
Fix $T> 5\tau_{K_0}$.
Let $A = K\cap \phi_{-T}(K)$. Let $\alpha=\frac{m_\infty(A)}{2}>0$.
In particular, there exists $k_0$ such that for all $n\ge k_0$, we have
$m_{K_0,L_n}(A)\ge \alpha$. Choose $0<\alpha'<\alpha$.
Lemma \ref{lemma on same po} applied with $\nu=1$ and $\tau_1 = 5\tau_{K_0}$ provides us with constants $\tau_0$ and $\epsilon_1$ which will be used below.
Fix $\eta$ and $\delta$ such that $0<\eta<\delta<\epsilon_1/4$.
Fix $T_0= 6\tau_{K_0}$.

The finite exact shadowing, i.e., Proposition \ref{weak shadowing property}, applied with the compact set $\overline{B(K,1)}$, with $\delta=\frac{\eta}{2}$ 
and $N=3$, gives us a constant $0<\rho<\frac{\eta}{2}$. The uniform transitivity, i.e., Lemma \ref{transitivite prop bis}, applied with the compact set $K$ and with $\delta=\rho$,
provides a constant $\sigma\geq 0$.
Let $S_2 = 2\sigma+2\tau_K+10\tau_{K_0}$.
Fix some $S\geq S_2$.

Let $x_0,y_0,x_1,y_1$ be four arbitrary points in $K$.
Let $E_0$ be a $E(x_0,y_0,\eta, T +S,T+S+5\tau_{K_0},\delta)$-set of maximal cardinality and $E_1$ be a $E(x_1, y_1, \eta,L_n-T+S,L_n-T+S+5\tau_{K_0},\delta)$-set of maximal cardinality. In particular, we have
\[
\# E_0=\mathcal{N}_C(x_0,y_0,\eta,T+S,T+S+5\tau_{K_0},\delta)\quad\text{and}\quad \# E_1=\mathcal{N}_C(x_1,y_1,\eta,L_n-T+S,L_n-T+S+5\tau_{K_0},\delta)\, .
\]

{\bf Step 2.} Some preliminary estimates on the number of cutting points.

Denote by
$\widetilde P_{K_0}(L_n)$ the set
\[
\left\{\gamma\in \mathcal P'_{K_0}(L_n,L_n+5{\tau}_{K_0})\,,\,\exists s\in [0,T_0], \#\{ i\in[0,\dots, \floor{L/T_0}-1],\gamma(s+iT_0)\in A\} \geq \frac{\alpha'}{2} \frac{L}{T_0}\right\}\,.
\]
From Proposition~\ref{prop_pre_coding}, since $m_{K_0, L_n}(A)\ge \alpha >0$, we know that
\[\#\widetilde P_{K_0}(L_n)\geq  0.99\times\frac{\alpha-\alpha'}{1-\alpha'}\times\,
\#\mathcal P_{K_0}(L_n,L_n+5\tau_{K_0})\, . \]
Without loss of generality, until the end of the proof, we reparametrize each $\gamma\in \widetilde{\mathcal P}_{K_0}(L_n)$ so that $s=0$. Let $\mathcal D_\gamma$ be the set of times of the form $iT_0$, such that $\gamma(iT_0)\in A$.
By the definition of $\widetilde{\mathcal P}_{K_0}(L_n)$, we have
\[
\#\mathcal D_\gamma
\geq \dfrac{\alpha'}{2}\times \dfrac{L_n}{T_0}=
\frac{\alpha'}{12\tau_{K_0}}\times L_n\, .
\]
Therefore
\begin{equation}\label{first eq to conclude prop cirm 2}
\#\bigsqcup_{\gamma\in \widetilde P_{K_0}(L_n)}\{\gamma\}\times D_\gamma \geq 0.99\times\frac{\alpha-\alpha'}{1-\alpha'}\times\frac{\alpha'}{12\tau_{K_0}}\times L_n\times
\#\mathcal P_{K_0}(L_n,L_n+5\tau_{K_0})\, .
\end{equation}

{\bf Step 3.} A map $F$ from periodic orbits to chords.

The intuitive idea is the following.
For  each point $\gamma(iT_0)\in A$, as $A=K\cap\phi_{-T}(K)$, we know that $\phi_T(\gamma(iT_0))\in K$.
Cutting the orbit $\gamma$ at these two points we get two chords of respective lengths $T$ and $\ell(\gamma)-T\in [L_n-T,L_n-T+5\tau_{K_0}]$.
Using transitivity, we add some small arcs to obtain chords from $x_0$ to $y_0$ and from $x_1$ to $y_1$.

Let $\gamma\in\tilde{\mathcal P}_{K_0}(L_n)$ and $i\in D_\gamma$.
Using transitivity and finite exact shadowing, we can build a chord $\beta^{0}$ with length $\ell(\beta^0)\in[T+S,T+S+4\tau_K]$  that goes from $B(x_0,\eta)$ at time $0$, to $B(\gamma(iT_0),\eta/2)$ at time  $t_{x_0}\in [ \frac{S}{2},\frac{S}{2}+2\tau_K]$, then follows exactly $\gamma$ at a distance at most $\eta/2$ during a time $T$, and then goes from $B(\gamma(iT_0+T),\eta/2)$ to $B(y_0,\eta)$ at time $\ell(\beta^0)$.
For the remaining part of the proof, we denote by $\beta^0_{x_0}$ the restriction of $\beta^0$ to the interval $[0,t_{x_0}]$, $\beta^0_{\gamma}$ the restriction of $\beta^0$ to the interval $[t_{x_0},t_{x_0}+T]$, and $\beta^0_{y_0}$ the restriction of $\beta^0$ to the interval $[t_{x_0}+T,\ell(\beta^0)]$.

Similarly, we get a chord $\beta^1$ from $B(x_1,\eta)$ to $B(y_1,\eta)$ of length  $\ell(\beta^1)\in [L_n-T+S, L_n-T+S+4\tau_K]$ that goes from $B(x_1,\eta)$ at time $0$, to $B(\gamma(iT_0+T),\eta/2)$ at some time $t_{x_1}\in [\frac{S}{2},\frac{S}{2}+2\tau_K]$, then follows exactly $\gamma$ at a distance at most $\eta/2$  during a time  $ \ell(\gamma)-T$, and then goes from $B(\gamma(iT_0),\eta/2)$ to $B(y_1,\eta)$ at a time $\ell(\beta^1)\in [L_n-T+S, L_n-T+S+4\tau_K]$. Observe that, in such a construction, the last part of the chord is obtained from a path built by transitivity whose length depends on the lengths of the chords of the first part of the construction, similarly to what is done in the proof of Proposition~\ref{orbite periodique et chords}: this enables to have a more precise control on the length of the final chord. As above, we denote by $\beta^1_{x_1}$ the restriction of $\beta^1$ to $[0,t_{x_1}]$, $\beta^1_{\gamma}$ the restriction of $\beta^1$ to $[t_{x_1},t_{x_1}+\ell(\gamma)-T]$ and $\beta^1_{y_1}$ the restriction of $\beta^1$ to $[t_{x_1}+\ell(\gamma)-T,\ell(\beta^1)]$.
Now, in order to define the image of $\gamma$ by $F$, we consider
the pair of chords $(\bar{\beta}_0,\bar{\beta}_1)\in E_0\times E_1$ that is the closest to the pair $(\beta^0,\beta^1)$.

Therefore, we obtain a map
\[
F\colon \bigcup_{\gamma\in \tilde{\mathcal{P}}_{K_0}(L_n)}\{\gamma\}\times\mathcal{D}_\gamma \to E_0\times E_1\, .
\]

{\bf Step 4.} The map $F$ is almost injective.

Assume that $(\gamma,iT_0)$ and $(\tilde\gamma,\tilde{i}T_0)$ lead to the same pair $(\bar{\beta}_0,\bar{\beta}_1)\in E_0\times E_1$, where $\gamma,\tilde\gamma\in \widetilde{\mathcal{P}}_{K_0}(L_n)$ and $i\in\mathcal D_\gamma, \tilde{i}\in \mathcal D_{\tilde\gamma}$.
First, assume that
 \begin{equation}\label{eqn:bound-a-priori1}
|\ell(\gamma)-\ell(\tilde \gamma)|
 \le
 \tau_0\,,
 \end{equation}
where $\tau_0$ is defined in Step 1.

The above construction associates to  $(\gamma,iT_0)$ (resp. $(\tilde \gamma,\tilde i T_0)$) two chords divided in three parts $\beta^0=(\beta^0_{x_0},\beta^0_\gamma,\beta^0_{y_0})$ and $\beta^1=(\beta^1_{x_1},\beta^1_\gamma,\beta^1_{y_1})$ (resp. $\tilde\beta^0=(\tilde\beta^0_{x_0},\tilde\beta^0_{\tilde\gamma},\tilde\beta^0_{y_0})$ and $\tilde\beta^1=(\tilde\beta^1_{x_1},\tilde\beta^1_{\tilde\gamma},\tilde\beta^1_{y_1})$) that are $\delta$-close to the chord $\bar{\beta}_0$, during a time at least $T+S$, and to the chord $\bar{\beta}_1$, during a time at least $L_n-T+S$, respectively.
Assume also  that
\begin{equation}\label{eqn:bound-a-priori}
\max\left\{
\left|\ell(\tilde \beta^0_{x_0})-\ell(\beta^0_{x_0}) \right|,\,
\left|\ell(\tilde\beta^0_{y_0})-\ell(\beta^0_{y_0}) \right|,\,
\left|\ell(\tilde \beta^1_{x_1})-\ell(\beta^1_{x_1}) \right|,\,
\left|\ell(\tilde \beta^1_{y_1})-\ell(\beta^1_{y_1}) \right|\right\}\le \,\frac{\delta}{b}\,.
\end{equation}
By construction, for every $u\in [0,T]$, we have
\[
d(\bar\beta_0(\ell(\beta^0_{x_0})+u), \beta^0(\ell(\beta^0_{x_0})+u))
\leq
\delta
\]
and
\[
d(\beta^0(\ell(\beta^0_{x_0})+u), \gamma(iT_0+u))
\leq
\eta/2.
\]
Therefore, for every $u\in[0,T]$,
\[
d(\overline\beta_0(\ell(\beta^0_{x_0})+u), \gamma(iT_0+u))
\leq
\delta + \eta/2.
\]
Similarly, for every $u\in[0,T]$,
\[
d(\overline\beta_0(l(\tilde \beta^0_{x_0})+u), \tilde \gamma(\tilde iT_0+u))
\leq
\delta + \eta/2.
\]
As, by \eqref{eqn:minoration},
\[
d(\overline\beta_0(\ell(\beta^0_{x_0})+u),\overline\beta_0(\ell(\tilde\beta^0_{x_0})+u))\leq b\left|\ell(\tilde \beta^0_{x_0})-\ell(\beta^0_{x_0}) \right|\leq \delta\, ,
\]
for every $u\in [0,T]$, we get, since $\eta<\delta<\frac{\epsilon_1}{4}$,
\[
d( \gamma(iT_0+u), \tilde \gamma(\tilde iT_0+u) ) \leq 4\delta <\epsilon_1\, .
\]
On the other part  of the orbit,  the same reasoning gives
for every $u\in[0,L_n-T]$
\[
d( \gamma(iT_0+T+u), \tilde \gamma(\tilde iT_0+T+u) )
\leq
4\delta <\epsilon_1\,.
\]
Therefore $\gamma$ and $\tilde\gamma$ are $\epsilon_1$-close on an interval of length $L_n$.
Lemma~\ref{lemma on same po} implies that $\gamma=\tilde \gamma$ and $i=\tilde i$.

This proves that for every choice of length of $\ell(\gamma)\in [L_n,L_n+5\tau_{K_0}]$ up to $\tau_0$ as in (\ref{eqn:bound-a-priori1}), and every choice respectively of $\ell(\beta_{x_0}^0), \ell(\beta_{y_0}^0),\ell(\beta_{x_1}^1),\ell(\beta^1_{y_1})$ in $ [\frac{S}{2},\frac{S}{2}+2\tau_K]$ up to $\delta/b$ as in (\ref{eqn:bound-a-priori}), there is at most one pair $(\gamma, iT_0)$ leading to $(\bar{\beta}^0,\bar{\beta}^1)$.
As a consequence, there are at most
$\left(\Big\lfloor\frac{5\tau_{K_0}}{\tau_0}\Big\rfloor+1\right)\times \left(\Big\lfloor\frac{2\tau_K\, b}{\delta}\Big\rfloor +1\right)^4$ pairs $(\gamma,iT_0)$ with $\gamma$ a periodic orbit in $\tilde{\mathcal{P}}_{K_0}(L_n)$ and $iT_0\in \mathcal{D}_\gamma$ that lead to the pair of chords $(\bar{\beta}_0,\bar{\beta}_1)$.
Therefore
\begin{equation}\label{second equation to conclude prop cirm 2}\# \bigcup_{\gamma\in \tilde{\mathcal{P}}_{K_0}(L_n)}\left(\{\gamma\}\times\mathcal{D}_\gamma\right)  \leq \left(\Big\lfloor\frac{5\tau_{K_0}}{\tau_0}\Big\rfloor+1\right)\times \left(\Big\lfloor\frac{2\tau_K\, b}{\delta}\Big\rfloor +1\right)^4\times \# E_0\times  \# E_1\, .\end{equation}

{\bf Step 5.} Conclusion.

This allows to  conclude the proof of Proposition \ref{prop cirm 2}.
Indeed, by \eqref{first eq to conclude prop cirm 2} and \eqref{second equation to conclude prop cirm 2},
\begin{align*}
    L_n\times \#\mathcal P_{K_0}(L_n,L_n+5\tau_{K_0})
    \le &\,
    \frac{1}{0.99}\frac{1-\alpha'}{\alpha-\alpha'}\,\times\,
    \frac{12\tau_{K_0}}{\alpha'}\times
    \# \bigcup_{\gamma\in \tilde{\mathcal{P}}_{K_0}(L_n)}\left(\{\gamma\}\times\mathcal{D}_\gamma\right) \\
    \le & \,  \frac{1}{0.99}\frac{1-\alpha'}{\alpha-\alpha'}\times \frac{12\tau_{K_0}}{\alpha'}\times\left(\Big\lfloor\frac{5\tau_{K_0}}{\tau_0}\Big\rfloor+1\right)\times \left(\Big\lfloor\frac{2\tau_K b}{\delta}\Big\rfloor +1\right)^4\times  \# E_0\times \#E_1 \\
    =& D_2\times \mathcal N_C(x_0,y_0,\eta,T+S,T+S+5\tau_{K_0},\delta)\\
     & \times  N_C(x_1,y_1,\eta,L_n-T+S,L_n-T+S+5\tau_{K_0},\delta)
\end{align*}
with
\[
D_2
=
\frac{1}{0.99}\frac{1-\alpha'}{\alpha-\alpha'}\,\times\,\frac{12\tau_{K_0}}{\alpha'}\,\times\,\left(\Big\lfloor\frac{5\tau_{K_0}}{\tau_0}\Big\rfloor+1\right)\times \left(\Big\lfloor\frac{2\tau_K\, b}{\delta}\Big\rfloor +1\right)^4\,.
\]


\section{The measure maximizes the entropy}\label{sec:finale}
In this section,  $K_0$ is  a compact set with nonempty interior as in section \ref{defi SPR} and $K\supset K_0$ is a larger compact set, such that $K_0\subset \inter{K}$.

In the first section, we prove uniform estimates on $m_{K_0,L}(B(x,T,\epsilon))$ for $x\in K$. In the second section, we take the limit when $L\to +\infty$ and obtain uniform estimates on the measure $m_\mathrm{max}(B(x,T,\epsilon))$, for every $x\in K$ (for every compact set $K\subset M$).
In the last section, we finally prove our main theorem.

\subsection{Relation between \texorpdfstring{$m_{K_0,L}$}{TEXT} and the number of chords}

The heuristics of this section is the following. Recall that the measure $m_{K_0,L}$ is defined in (\ref{eqn:m_{K,L}}) as the average of the periodic invariant probability measures $\mu_\gamma$, where $\gamma$ varies over all periodic orbits in $\mathcal P_{K_0}(L,L+5\tau_{K_0})$, i.e., those intersecting $K_0$, with length in $[L,L+5\tau_K]$.
Given some dynamical ball $B(x,T,\epsilon)$ for $x\in M$, the measure $m_{K_0,L}(B(x,T,\epsilon))$ satisfies therefore
\begin{eqnarray*}
m_{K_0,L}(B(x,T,\epsilon))&=&\frac{1}{\#\mathcal P_{K_0}(L,L+5\tau_{K_0})}\sum_{\gamma\in\mathcal P_{K_0}(L,L+5\tau_{K_0})}\mu_\gamma(B(x,T,\epsilon))\\
&\simeq&\frac{1}{L\times \#\mathcal P_{K_0}(L,L+5\tau_{K_0})}\sum_{\gamma\in\mathcal P_{K_0}(L,L+5\tau_{K_0})} \ell(\gamma\cap B(x,T,\epsilon))\, .
\end{eqnarray*}

Now, the heart of the argument is the proof that the last sum is comparable to  the number $\mathcal{N}_{\mathcal C}(\phi_T(x),x,\eta,L-T,\delta)$ of chords from $\phi_T(x)$ to $x$, so that
\[
m_{K_0,L}(B(x,T,\epsilon))
\simeq
\frac{1}{L\times \#\mathcal P_{K_0}(L,L+5\tau_{K_0})} \times \mathcal N_{\mathcal C}(\phi_T(x),x,L-T,\delta)\,.
\]

The idea of the proof consists in the following remark. Given a chord starting at a point $z$ of length roughly $L-T$ from a neighbourhood of $\phi_T(x)$ to a neighbourhood of $x$, and using the closing lemma as in Lemma \ref{petal separe}, we can build a periodic orbit in $\mathcal P_{K_0}(L,L+5\tau_K)$ following first $(\phi_s(x))_{0\le s\le T}$ and afterwards the chord $(\phi_s(z))_{0\le s\le L-T}$, and intersecting $B(x,\epsilon,T)$. Conversely, given a periodic orbit $\gamma\in \mathcal P_{K_0}(L,L+5\tau_K)$ with an origin $w\in B(x,\epsilon,T)$, we can cut $\gamma$ at $0$ and $T$, and get a chord of length roughly $L-T$ from a neighbourhood of $\phi_T(x)$ to a neighbourhood of $x$. The difficulty of the argument is to show that the above constructions are almost one-to-one, or more precisely, that each preimage of the corresponding maps between chords and periodic orbits has bounded cardinality. The rigorous details corresponding to the above heuristics are provided in Lemmas~\ref{ineq Anna} (lower bound) and \ref{ineq Anne} (upper bound). To this purpose, we will need to notice in Lemma~\ref{lemma Delta max} that each return of $\gamma$ in $B(x,T,\epsilon)$ has a bounded length and in Lemma~\ref{lemma Delta''} that these returns are not too close one from another.

We start with an easy but useful observation. Recall that $b$ is defined in Equation~(\ref{eqn:minoration}).

\begin{lemm}\label{lemma temps perm}
For all $\epsilon>0$, $T>0$, $x\in M$, and $y\in B(x,\frac{\epsilon}{2},T)$, for every $0\le s<\frac{\epsilon}{2b}$, we have  $\phi_{s}(y)\in B(x,\epsilon,T)$.
\end{lemm}

\begin{proof}
For every $s\in [0, \frac{\epsilon}{2b}[$, and every $\tau\in[0,T]$, we have
\begin{align*}
d(\phi_{\tau+s}(y),\phi_{\tau}(x))\leq d(\phi_{\tau+s}(y),\phi_{\tau}(y))+d(\phi_{\tau}(y),\phi_{\tau}(x))
\leq |s|b+\frac{\epsilon}{2}< \frac{\epsilon}{2}+\frac{\epsilon}{2}=\epsilon
\end{align*}
so that $\phi_s(y)\in B(x,\epsilon, T)$, as desired.
\end{proof}

In Lemmas~\ref{lemma Delta max} and \ref{lemma Delta''}, given a set $C$, and a point $y\in C$, we consider the connected component of $0$ in $\{s\in \R, \phi_s(y)\in C\}$. It is an interval that we denote by $J_C(y)=(J_C^{\rm min}(y),J_C^{\rm max}(y))$. In Lemma~\ref{lemma Delta max}, we prove that for a dynamical ball $B(x,\epsilon,T)$, and any $y\in B(x,\epsilon,T)$, the size of the interval $J_{B(x,\epsilon,T)}(y)$ is bounded.

\begin{lemm}\label{lemma Delta max}
Let $K\subset M$ be compact and $\Delta >0$.  There exist $\epsilon_{K,\Delta}>0$  such that for every $0<\epsilon\leq\epsilon_{K,\Delta}$, $x\in K$, $T\ge1$ and $y\in B(x,\epsilon,T)$, we have
\[
\mathrm{Leb}\left(J_{B(x,\epsilon,T)}(y)\right)=J^{\rm max}_{B(x,\epsilon,T)}(y)-J^{\rm min}_{B(x,\epsilon,T)}(y)\le \mathrm{Leb}\left(J_{B(x,\epsilon,1)}(y)\right) \le  \Delta\,.
\]
\end{lemm}

\begin{proof}
First note that for all $T \ge 1$, $x\in M$ and $\epsilon>0$, we have $B(x,\epsilon,T)\subset B(x,\epsilon,1)$ and therefore $J_{B(x,\epsilon,T)}(y)\subset J_{B(x,\epsilon,1)}(y)$.
For every $x\in K$, there exists $\tau_0=\tau_0(x)>0$, $r=r(x)>0$ and a flow-bow $\Omega$ centered at $x$ diffeomorphic to $B(0,r)\times (-\tau_0,\tau_0)$  (where $B(0,r)$ is a ball of radius $r$ in $\mathbb R^{\dim(M)-1}$).
As $\Omega$ is a neighborhood of $x$, there exists $\epsilon(x)$ such that $B(x,\epsilon(x),1)\subset \Omega$.
For every $0<\epsilon\leq\epsilon(x)$ and every $y\in B(x,\epsilon,1)$, we have $B(x,\epsilon,1)\subset B(x,\epsilon(x),1)\subset \Omega$.
Therefore, $\mathrm{Leb}\left(J_{B(x,\epsilon,1)}(y)\right)\le 2\tau_0$.
We may reduce $\tau_0$ so that $\tau_0\leq \Delta/2$.
As $K$ is compact, one can choose $\tau_0$ and $\epsilon$ uniformly in $x$.
This concludes the proof of the lemma.
\end{proof}

We are now able to prove the first key lemma of this section.

\begin{lemm}\label{ineq Anna} Let $K_0$ be a compact set with nonempty interior as in Section~\ref{defi SPR}, and $K\supset K_0$ be a larger compact set.
For every $\delta>0$, there exists $\epsilon_{K,\delta}>0$ and $T'_\mathrm{min}>1$ such that for every $0<\epsilon<\epsilon_{K,\delta}$, there exist $\sigma=\sigma_{K,\delta, \epsilon}>0$ and $\eta_{K,\delta,\epsilon}>0$ such that for every $0<\eta<\eta_{K,\delta,\epsilon}$, all $L,T$ such that $L\ge T+T'_\mathrm{min}$ and $T\ge T'_\mathrm{min}$ and $x\in K$ such that $\phi_T(x)\in K$, we have
\[
\frac{\epsilon}{4b}\times \frac{\mathcal{N}_{\mathcal{C}}(\phi_T(x),x,\eta,L-T-\sigma_{K,\delta,\epsilon},L-T-\sigma_{K,\delta,\epsilon}+5\tau_{K_0},\delta)}{L\#\mathcal{P}_{K_0}(L,L+5\tau_{K_0})} \,\leq m_{K_0,L}(B(x,\epsilon,T)).
\]
\end{lemm}

\begin{proof} The heuristics of the proof is the following. Thanks to Lemma~\ref{petales reunis}, we can glue any chord of length roughly $L-T$ starting close to  $\phi_T(x)$ and arriving close to $x$ to the orbit $(\phi_t(x))_{0\le t\le T}$ to get a periodic orbit of length $L$ intersecting $B(x,\epsilon,T)$. Doing it with enough care will guarantee that this intersection point lies inside $B(x,\epsilon/2,T)$, and therefore the orbit spends a time at least $\frac{\epsilon}{2b}$ inside $B(x,\epsilon,T)$. This will allow to get the desired bound.
Let us start the rigorous argument.

{\bf Step 1.} Choice of constants.
Choose $T'_\mathrm{min}$ so that for $L\ge T'_\mathrm{min}$,
we have $\frac{1}{L+5\tau_{K_0}}\ge \frac{1}{2L}$.
Fix $\delta>0$. Then, Lemma~\ref{lemma Delta max} applied on $K$ with $\Delta=\delta/2b$ gives us a constant $\epsilon_{K,\delta}>0$.

Fix $0<\epsilon<\epsilon_{K,\delta}$ and set $\delta'=\min(\delta/4,\epsilon/2,1)$.
Lemma~\ref{petales reunis} applied with $K_0\subset \overline{B(K,1)}$, $\delta=\delta'>0$, $\nu=\tau_{K_0}$ and $N=2$, gives us $\eta_{K_0,\delta,\epsilon}\le 1$, $\sigma>0$ and $T_{\rm min}>0$ such that the following holds. Let $\sigma_{K,\delta, \epsilon}:=\sigma +3\tau_{K_0}$.
First, increase $T'_\mathrm{min}$ so that $T'_\mathrm{min}\geq T_\mathrm{min}$ and $T'_\mathrm{min}\geq \sigma_{K,\delta, \epsilon}$.
Then, for every $0<\eta<\eta_{K_0,\delta,\epsilon}$, $L,T$ such that $L\geq T + T'_\mathrm{min}$ and $T \geq  T'_\mathrm{min}$ and any chord $z$ from $B(\phi_T(x),\eta)$ to $B(x,\eta)$ with length $\ell(z)$ such that
\[L-T-\sigma_{K,\delta, \epsilon}\le\ell(z)\le L-T-\sigma_{K,\delta, \epsilon}+5\tau_{K_0}\, ,\]
let $$S=\sigma_{K,\delta,\epsilon}-\tau_{K_0}+(L-T-\sigma_{K,\delta,\epsilon}+5\tau_{K_0}-\ell(z))=L-T-\ell(z)+2\tau_{K_0}\ge \sigma\, ;$$
By Lemma~\ref{petales reunis}, there exists a periodic orbit $\gamma$ with length
\[
\ell(\gamma)\in [T+\ell(z)+S-2\tau_{K_0},T+\ell(z)+S+2\tau_{K_0}]=[L,L+4\tau_{K_0}]\subset[L,L+5\tau_{K_0}]
\]
that intersects $\inter{K_0}$,  $\delta'$-shadows first $(\phi_{t}(x))_{0\le t\le T}$ and then $(\phi_t(z))_{0\le t\le \ell(z)}$. In particular, $\gamma(0)\in B(x,\delta')$. Observe that $\gamma\in\mathcal{P}_{K_0}(L,L+5\tau_{K_0})$.

{\bf Step 2.}  A first lower bound for $m_{K_0,L}(B(x,\epsilon,T))$.
Denote by $\mathcal{J}(L+5\tau_{K_0})$ the set of intervals included in $[0,L+5\tau_{K_0}]$.
For each periodic orbit $\gamma\in \mathcal{P}_{K_0}(L,L+5\tau_{K_0})$, choose a parametrization such that $\gamma(0)\notin B(x,\epsilon,T)$.
Denote by $\Theta_{K_0,L}(x,\epsilon,T)$ the set of pairs $(\gamma,I)  \in \mathcal{P}_{K_0}(L,L+5\tau_{K_0})\times \mathcal J(L+5\tau_{K_0})$ such that for every $\gamma\in \mathcal{P}_{K_0}(L,L+5\tau_{K_0})$ (with its parametrization and associated origin), and every $s\in I$, $\gamma(s)\in B(x,\epsilon,T)$.
Moreover, assume $I$ is maximal for this property.
In a similar way, let $\Theta'_{K_0,L}(x,\epsilon,T)$ be set of pairs $(\gamma,I)\in \Theta_{K_0,L}(x,\epsilon,T)$ such that $\gamma(I)\cap B(x,\epsilon/2,T)\neq\emptyset$.

Observe that for $L\ge T'_\mathrm{min}$, we have
\begin{align*}
m_{K_0,L}(B(x,\epsilon,T))&= \frac{1}{\#\mathcal{P}_{K_0}(L,L+5\tau_{K_0})}\sum_{\gamma\in\mathcal{P}_{K_0}(L,L+5\tau_{K_0})}\mu_\gamma(B(x,\epsilon,T)) \\
&= \frac{1}{\#\mathcal{P}_{K_0}(L,L+5\tau_{K_0})}\sum_{(\gamma,I)\in \Theta_{K_0,L}(x,\epsilon,T)} \frac{\mathrm{Leb}(I)}{\ell(\gamma)}\\
&\ge \frac{1}{\#\mathcal{P}_{K_0}(L,L+5\tau_{K_0})}\sum_{(\gamma,I)\in \Theta'_{K_0,L}(x,\epsilon,T)} \frac{\mathrm{Leb}(I)}{\ell(\gamma)}\\
&\ge \frac{1}{\#\mathcal{P}_{K_0}(L,L+5\tau_{K_0})}\sum_{(\gamma,I)\in \Theta'_{K_0,L}(x,\epsilon,T)} \frac{\mathrm{Leb}(I)}{2L}\\
&\ge\frac{1}{\#\mathcal{P}_{K_0}(L,L+5\tau_{K_0})}\,\times\,\frac{\epsilon}{4bL}\,\times \,\#\Theta'_{K_0,L}(x,\epsilon,T)\,.
\end{align*}
where the second lower bound follows from the inequalities $L\le \ell(\gamma)\le L+5\tau_{K_0}$ and $\frac{1}{L+5\tau_{K_0}}\ge \frac{1}{2L}$, and the third lower bound comes from Lemma~\ref{lemma temps perm}.

{\bf Step 3.} Going from chords to periodic orbits. Let $E$ be a $E(\phi_T(x),x,\eta,L-T-\sigma_{K,\delta,\epsilon}, L-T-\sigma_{K,\delta,\epsilon}+5\tau_{K_0},\delta)$-set of maximal cardinality, i.e., so that
\[
\#E =\mathcal{N}_{\mathcal C}(\phi_T(x),x,\eta,L-T-\sigma_{K,\delta,\epsilon},L-T-\sigma_{K,\delta,\epsilon}+5\tau_{K_0},\delta)\,.
\]
Applying Lemma~\ref{petales reunis} as described above, we can associate to each $z$ in $E$ a periodic orbit $\gamma$ in $\mathcal{P}_{K_0}(L,L+5\tau_{K_0})$. Moreover, we know that the origin $s_0$ given by the construction is such that $\gamma(s_0)\in B(x,\delta',T)\subset B(x,\epsilon/2,T)$.
This point in $B(x,\epsilon/2,T)$ gives us an interval $J$ with $(\gamma,J)\in \Theta'_{K_0,L}(x,\epsilon,T)$
and therefore a map
\[
\theta:E\to \Theta'_{K_0,L}(x,\epsilon,T)\,.
\]

{\bf Step 4.} The map $\theta$ is injective. Assume that two chords $z_1,z_2$ in $E$  have same image $(\gamma,J)$.
For $\gamma$ we use the parametrization from the definition of $ \Theta_{K_0,L}(x,\epsilon,T)$. The origins $s_1$ and $s_2$ associated to the construction of $\theta(z_1)$ and $\theta(z_2)$ satisfy $s_1,s_2 \in J$ and $\gamma(s_1),\gamma(s_2) \in B(x,\epsilon/2,T)$.
By Lemma~\ref{lemma Delta max},as $\epsilon\leq \epsilon_{K,\delta}$, we have $|J|\le \delta/2b$ (see Step 1). Therefore $|s_1-s_2|\le \delta/2b$.
An elementary computation gives, for every $0\le s\le \min(\ell(z_1),\ell(z_2))$, where $\ell(z_1),\ell(z_2)$ denote the length of the chords of $z_1,z_2$ respectively,
\begin{align*}
d(\phi_{s+T+s_1}(z_1),\phi_{s+T+s_2}(z_2))\le &\, d(\phi_{s+T+s_1}(z_1),\gamma(s+T+s_1))+d(\gamma(s+T+s_1),\gamma(s+T+s_2))\\
&+\,d(\gamma(s+T+s_2),\phi_{s+T+s_2}(z_2))\\
< &\, \delta'+\delta/2+\delta'\le \delta\,.
\end{align*}
As $E$ is a $(\delta, L-T-\sigma_{K,\delta,\epsilon})$-separating set, we deduce that $z_1=z_2$, so that $\theta$ is injective.
Therefore
\[\# E \leq  \# \Theta'_{K_0,L}(x,\epsilon,T).\]

{\bf Conclusion.} The above arguments show that
\begin{eqnarray*}
m_{K_0,L}(B(x,\epsilon,T))&\ge& \frac{\epsilon}{4b}\times \frac{1}{L\,\#\mathcal P_{K_0}(L,L+5\tau_{K_0})}\times \#\Theta'_{K_0,L}(x,\epsilon,T)\\
&\ge & \frac{\epsilon}{4b}\times \frac{1}{L\,\#\mathcal P_{K_0}(L,L+5\tau_{K_0})}\times \#E\\
&=& \frac{\epsilon}{4b}\times \frac{\mathcal{N}_{\mathcal C}(\phi_T(x),x,\eta,L-T-\sigma_{K,\delta,\epsilon},L-T-\sigma_{K,\delta,\epsilon}+5\tau_{K_0},\delta)}{L\,\#\mathcal P_{K_0}(L,L+5\tau_{K_0})}\,.
\end{eqnarray*}
\end{proof}

Our next goal is to bound $m_{K_0,L}(B(x,\epsilon,T))$ from above. We will need to bound from above the total amount of time that a periodic orbit $\gamma\in\mathcal{P}_{K_0}(L,L+5\tau_{K_0})$ spends in $B(x,\epsilon,T)$.
We know from Lemma~\ref{lemma Delta max} that each interval of time that such an orbit spends in $B(x,\epsilon,T)$ has a bounded length.
Lemma~\ref{lemma Delta''} is the second technical lemma. It allows to say that, up to increasing slightly $B(x,\epsilon,T)$ to make it more smooth, the distance between two such intervals is bounded from below by a uniform constant, which allows to bound as desired the total amount of time in $B(x,\epsilon,T)$.

\begin{lemm}[No immediate return]\label{lemma Delta''}
Let $K\subset M$ be a compact set.
There exist $\Delta_K>0$ and $\alpha_K>0$ such that for every $0<\epsilon\leq\alpha_K$, for every $x\in K$ and $T\geq 1$, there exists a set $C(x,\epsilon,T)$ satisfying $B(x,\epsilon,T)\subset C(x,\epsilon,T)$, such that for every $y\in B(x,\epsilon,T)$, and every $0< s < \Delta_K$,
\[
\phi_{J^{\rm max}_{C(x,\epsilon,T)}(y)+s}(y)\notin C(x,\epsilon,T), \quad \phi_{J^{\rm min}_{C(x,\epsilon,T)}(y)-s}(y)\notin C(x,\epsilon,T) \quad\mbox{and}\quad \mathrm{Leb}\left(J_{C(x,\epsilon,T)}(y)\right)\le 2\Delta_K\,.
\]
\end{lemm}
In other words, $C(x,\epsilon,T)$ contains portion of orbits of length at most $2\Delta_K$, after exiting $C(x,\epsilon,T)$ an orbit remains outside $C(x,\epsilon,T)$ for a time bounded below by $\Delta_K$ and a similar property is satisfied in the past.

\begin{proof} For every $x\in K$, we can find $0<r(x)<1$ and $0<\tau_0(x)<\tau_K/4$ such that $x$ admits a flow-box neighbourhood $\Omega(x)$ diffeomorphic to $B(0,r(x))\times (-\tau_0(x),\tau_0(x))$.
Let $\psi : \Omega\to B(0,r(x))\times (-\tau_0(x),\tau_0(x))$ be the associated flow-box chart. There exists $\epsilon_0(x)>0$ small enough such that $B(x,\epsilon_0(x),1)$ is included in a flow box of half height. More precisely, for every $T\ge 1$,
\[
\psi(B(x,\epsilon_0(x),T))\subset \psi(B(x,\epsilon_0(x),1))\subset B(0,r(x))\times (-\tau_0(x)/2,\tau_0(x)/2)\,.
\]
Fix $0<\epsilon\leq \epsilon_0(x)$.
It is not clear to us whether $\psi(B(x,\epsilon,T))$ is convex (at least in the direction of the flow) or not so an orbit may exit $B(x,\epsilon,T)$ for a very short time. To avoid these technical problems, we fill $B(x,\epsilon,T)$ in the direction as the flow, and define $C(x,\epsilon,T)$ as
\[
\psi^{-1}\left(
\{(z,\tau)\in B(0,r(x))\times (-\tau_0(x)/2,\tau_0(x)/2), \,\exists \tau'\in (-\tau_0(x)/2,\tau_0(x)/2),(z,\tau')\in \psi(B(x,\epsilon,T)) \}\right).
\]
Then, by construction, $\psi(C(x,\epsilon,T))\subset B(0,r(x))\times (-\tau_0(x)/2,\tau_0(x)/2)$ so that for every $0<s<\tau_0(x)/2$ and every $y\in B(x,\epsilon,T)$,
\begin{itemize}
 \item $\phi_{ J^\mathrm{max}_{C(x,\epsilon,T)}(y)+s}(y)\cap C(x,\epsilon,T) = \emptyset$;
 \item $\phi_{J^\mathrm{min}_{C(x,\epsilon,T)}(y)-s}(y)\cap C(x,\epsilon,T) = \emptyset$;
 \item $\mathrm{Leb}\left(J_{C(x,\epsilon,T)}\right)\le \tau_0(x)$
\end{itemize}
As $K$ is compact, $r$, $\tau_0$ and $\epsilon_0$ can be chosen uniformly in $x$. The result follows with $\Delta_K = \tau_0/2$ and $\alpha_K = \epsilon_0$.
\end{proof}

The upper bound for $m_{K_0,L}(B(x,\epsilon,T))$ is proven in the second key lemma.
\begin{lemm}\label{ineq Anne} Let $K_0$ be a compact set with nonempty interior as in Section~\ref{defi SPR}, and $K\supset K_0$ be a larger compact set.
There exist $\epsilon_{K_0,K}>0$, $\delta_{K_0,K}>0$ and $D_{K_0,K}>0$ such that for every $0<\epsilon<\epsilon_{K_0,K}$ and every $0<\delta<\delta_{K_0,K}$,
for every $x\in K$ and for all
$1\leq T<L $ such that $\phi_T(x)\in K$, the following inequality holds
\begin{align*}
 m_{K_0,L}(B(x,\epsilon,T))\leq
D_{K_0,K}\,\dfrac{\mathcal{N}_{\mathcal{C}}(\phi_T(x),x,\epsilon,L-T,L-T+5\tau_{K_0},\delta)}{L\#\mathcal{P}_{K_0}(L,L+5\tau_{K_0})}
 +\dfrac{\#\mathcal{P}_{K_0}\left(\frac{L+5\tau_{K_0}}{2}\right)}{\#\mathcal{P}_{K_0}(L,L+5\tau_{K_0})}.
\end{align*}
\end{lemm}

\begin{proof}
The proof follows the same lines as the one of  Lemma~\ref{ineq Anna}.

{\bf Step 1.} Choice of constants.
Lemma~\ref{lemma Delta''} gives us constants $\Delta_K$ and $\alpha_K>0$ associated with $K$.
Lemma~\ref{lemma on same po} applied with $\nu=\min(\Delta_K/3,1)$ and $\tau_1=5\tau_{K_0}\ge 1$ gives us $\tau_0>0$ and $\epsilon_1>0$.
Let $\epsilon_{K_0,K}=\min(\alpha_K,\epsilon_1)$
and
$\delta_{K_0,K} = \frac{\epsilon_1}{2}$.
Fix $\epsilon<\epsilon_{K_0,K}$ and $\delta<\delta_{K_0,K}$.
Fix $x\in K$.
Fix $1\leq T\leq L$ such that $\phi_T(x)\in K$.

{\bf Step 2.} First upper bound for $m_{K_0,L}(B(x,\epsilon,T))$.
Denote by $\mathcal{J}(L+5\tau_{K_0})$ the set of intervals included in $[0,L+5\tau_{K_0}]$.
For each periodic orbit $\gamma\in \mathcal{P}_{K_0}(L,L+5\tau_{K_0})$, choose a parametrization with an origin $\gamma(0)\notin C(x,\epsilon,T)$.
Denote by $\widetilde\Theta_{K_0,L}(x,\epsilon,T)$ the set of pairs $(\gamma,I) \subset \mathcal{P}_{K_0}(L,L+5\tau_{K_0})\times \mathcal J(L+5\tau_{K_0})$ such that for every $\gamma\in \mathcal{P}_{K_0}(L,L+5\tau_{K_0})$ (with its parametrization and associated origin), and every $s\in I$, we have $\gamma(s)\in C(x,\epsilon,T)$, $I$ is maximal for this property and $\gamma(I)\cap B(x,\epsilon,T)\neq\emptyset$.
Note that for all $(\gamma,I)\in \widetilde\Theta_{K_0,L}(x,\epsilon,T)$, by Lemma~\ref{lemma Delta''}, we have $\mathrm{Leb}(I)\leq 2\Delta_K$.

In the proof it will be important to focus on primitive orbits. Recall that $\mathcal{P}'_{K_0}(L,L+5\tau_{K_0})$ is the subset of primitive orbits and observe that if $\gamma$ is not primitive, then there exists a primitive periodic orbit with length at most $\ell(\gamma)/2$ with the same image. Therefore,
\[
\#\left(\mathcal P_{K0}(L,L+5\tau_{K_0})\setminus \mathcal{P}'_{K_0}(L,L+5\tau_{K_0})\right) \le \#\mathcal P_{K_0}\left(\frac{L+5\tau_{K_0}}{2}\right)\,.
\]

We then have
\begin{align*}
m_{K_0,L}(B(x,\epsilon,T)=& \frac{1}{\# \mathcal P_{K_0}(L,L+5\tau_{K_0})}\sum_{\gamma\in \mathcal P_{K_0}(L,L+5\tau_{K_0})} \mu_\gamma(B(x,\epsilon,T))\\
\leq & \frac{1}{\# \mathcal P_{K_0}(L,L+5\tau_{K_0})}\sum_{\gamma\in \mathcal P_{K_0}(L,L+5\tau_{K_0}) ,\gamma \text{ primitive}} \mu_\gamma(C(x,\epsilon,T))\\
&+\frac{1}{\# \mathcal P_{K_0}(L,L+5\tau_{K_0})}\sum_{\gamma\in \mathcal P_{K_0}(L,L+5\tau_{K_0}),\gamma \text{ non-primitive}} \mu_\gamma(B(x,\epsilon,T))\\
\leq& \frac{1}{\# \mathcal P_{K_0}(L,L+5\tau_{K_0})}\sum_{(\gamma,I)\in \widetilde \Theta_{K_0}(L,L+5\tau_{K_0})} \dfrac{2\Delta_K}{L} \\&
+ \frac{\# \{\gamma\in \mathcal P_{K_0}(L,L+5\tau_{K_0}),\gamma \text{ non-primitive}\}}{\# P_{K_0}(L,L+5\tau_{K_0})}\\
\leq& \frac{2\Delta_K\times \# \widetilde \Theta_{K_0}(L,L+5\tau_{K_0})}{L\, \# \mathcal P_{K_0}(L,L+5\tau_{K_0})}+\dfrac{\#\mathcal{P}_{K_0}\left(\frac{L+5\tau_{K_0}}{2}\right)}{\#\mathcal{P}_{K_0}(L,L+5\tau_{K_0})}\\
\end{align*}

{\bf Step 3.} From $\widetilde \Theta_{K_0}(L,L+5\tau_{K_0})$ to chords.
Every pair $(\gamma,I)\in \widetilde \Theta_{K_0}(L,L+5\tau_{K_0})$ gives us a set $\{\gamma(s),s\in I\}$ of points of $\gamma$ inside $C(x,\epsilon,T)$, with at least some $s_0\in I$ with $\gamma(s_0)\in B(x,\epsilon,T)$. Following $\gamma$ from $\gamma(s_0+T)$ to $\gamma(s_0+\ell(\gamma))=\gamma(s_0)$ defines a chord with length $\ell(\gamma)-T\in [L-T,L+5\tau_{K_0}-T)$ from $B(\phi_T(x,\epsilon))$ to $B(x,\epsilon)$.

Let $E$ be a $E(\phi_T(x),x,\epsilon, L-T,L-T+5\tau_{K_0},\delta)$-set with maximal cardinality. In particular, $\#E=\mathcal N_{\mathcal C}(\phi_T(x),x,\epsilon, L-T,L-T+5\tau_{K_0},\delta)$.

For each pair $(\gamma,I)\in  \widetilde \Theta_{K_0}(L,L+5\tau_{K_0})$, choose a point $z\in E$, such that $\gamma(s_0+T)\in B(z,\delta,L-T)$.
This gives us a map $\theta : \widetilde \Theta_{K_0}(L,L+5\tau_{K_0}) \to E$.

{\bf Step 4.} Control the (lack of) injectivity of $\theta$.
Let $(\gamma_1,I_1)$ and $(\gamma_2,I_2)$ be two pairs that lead to the same point $z\in E$.
In particular, there exist $s_1\in I_1, s_2\in I_2$ such that $\gamma_1(s_1),\gamma_2(s_2)\in B(x,\epsilon, T)$, so, for every $0\le s\le T$,
\[
d(   \gamma_1(s_1+s),\phi_s(x))\le \epsilon\quad\mbox{and}\quad d(\gamma_2(s_2+s),\phi_s(x))\le\epsilon\,.
\]
Moreover, as $\theta(\gamma_1,I_1)=\theta(\gamma_2,I_2)=z$, for every $0\le s\le L-T$, we have
\[
d(\gamma_1(s_1+T+s),\gamma_2(s_2+T+s))\le d(\gamma_1(s_1+T+s),\phi_s(z))+d(\phi_s(z),\gamma_2(s_2+s+T))\le 2\delta.
\]
Therefore, for all $0\leq s \leq L$, we have
\[ d(\gamma_1(s_1+s),\gamma_2(s_2+s))\leq \epsilon_1.\]
If $0\le \ell(\gamma_1)-\ell(\gamma_2)\le \tau_0$, using Lemma~\ref{lemma on same po} (see Step 1) and the fact that $\gamma_1$ and $\gamma_2$ are primitive, we deduce that $\gamma_1=\gamma_2$ and there exists $u\in [-\nu,\nu]$ such that  and $s_2=s_1+u$.
By Lemma~\ref{lemma Delta''}, we deduce that $I_1=I_2$ (otherwise, we would have $\vert s_1-s_2\vert=\vert u\vert \ge \Delta_k$, contradicting $\vert u\vert \le \nu=\min\left(\frac{\Delta_k}{3},1\right)$).

Cutting the interval $[L,L+5\tau_{K_0}]$ into intervals of length $\tau_0$, we deduce that the number of elements $(\gamma,I)\in\widetilde \Theta_{K_0}(L,L+5\tau_{K_0})$ leading to the same chord is bounded from above by $5\tau_{K_0}/\tau_0$.

The result of the lemma follows with $D_{K_0,K}=10\Delta_K\tau_{K_0}/\tau_0$.
\end{proof}


\subsection{Estimation of the measure of dynamical balls}\label{section:estimation_boules_dynamiques}

In this section, we gather all the inequalities proven in
Lemma~\ref{ineq Anna}, Lemma~\ref{ineq Anne}, Proposition~\ref{prop cirm 1} and Proposition~\ref{prop cirm 2} to obtain the following strong inequalities.

Let $K_0$ be a compact set as in Section~\ref{section_periodic_measures}
Let $m_\infty$ be any accumulation point of the family $(m_{K_0,L})_L$ and $m_\mathrm{max}$ be the probability measure obtained after renormalizing $m_\infty$.
In particular, $m_\infty(K_0)>0$.

Choose $K\supset \inter{K}\supset K_0$ such that $m_\infty(\inter{K})>\frac{3}{4}m_\infty(M)$.

Choose an increasing sequence $(L_k)_k$ such that $L_k\to +\infty$ as $k\to +\infty$ and $m_{K_0, L_k}\overset{*}{\rightharpoonup} m_\infty$.
Fix some point  $y_0\in K_0$.

\begin{theo}\label{prop_boule_dynamique} There exist $\delta_\mathrm{end}>0$, such that for every $0<\delta<\delta_\mathrm{end}$, there exists $\epsilon_{\mathrm{end},\delta}>0$ such that for every $0<\epsilon<\epsilon_{\mathrm{end},\delta}$ there exists $\eta_{\mathrm{end},\delta,\epsilon}$ such that for every $0<\eta<\eta_{\mathrm{end},\delta,\epsilon}$, there exist positive constants $S^-,S^+,D^-,D^+>0$, and $T_\mathrm{end}>0$ such that  the following holds.
For every $x\in K$ and $T>T_\mathrm{end}$ such that $\phi_T(x)\in K$, we have
\begin{equation}\label{longue 5}
 \dfrac{D^-}{\mathcal N_\mathcal C(x,y_0,\eta, T+S^-, T+S^-+5\tau_{K_0},\delta)}\le m_\infty(B(x,\epsilon,T))\leq \dfrac{D^+}{\mathcal N_\mathcal C(x,y_0,\epsilon, T-S^+, T-S^++5\tau_{K_0},\delta)}\, .
\end{equation}
\end{theo}

\rm Before proving it, let us emphasize the strength of this statement. Usually, in ergodic theory, invariant measures satisfy almost sure properties. The above inequalities hold for {\bf every} $x\in K$, and are therefore more geometric than ergodic. This strong uniform property will allow us to conclude that $m_\mathrm{max}$ is a measure of maximal entropy.

\begin{proof} The proof follows easily from the preceding work, and in particular from Propositions \ref{prop cirm 1} and \ref{prop cirm 2}, and from Lemmas \ref{ineq Anna} and \ref{ineq Anne}, as soon as parameters are carefully chosen. We start with this choice.

{\bf Step 1.} Choice of constants.
Set $\delta_\mathrm{end}=\min(\frac{\epsilon_1}{4},\delta_{K_0,K},1)$ where $\epsilon_1$ is given by Proposition~\ref{prop cirm 2} and $\delta_{K_0,K}$ by Lemma~\ref{ineq Anne}.
Choose an arbitrary $0<\delta<\delta_\mathrm{end}$.

Set $\epsilon_{\mathrm{end},\delta}=\min(\epsilon_{K_0,K},\epsilon_{K,\delta})$, where $\epsilon_{K_0,K}$ is given by Lemma~\ref{ineq Anne} and $\epsilon_{K,\delta}$ by Lemma~\ref{ineq Anna} applied with parameter $\delta$.
Fix $0<\epsilon<\epsilon_{\mathrm{end},\delta}$.

Lemma~\ref{ineq Anna} applied with $\epsilon/2$ gives constants $\sigma_{K,\delta,\epsilon/2}$ and $\eta_{K,\delta,\epsilon/2}$. Set $\eta_{\mathrm{end},\delta,\epsilon}=\min(\delta,\eta_{K,\delta,\epsilon/2},1)$.
Choose $0<\eta<\eta_{\mathrm{end},\delta,\epsilon}$.

Propositions~\ref{prop cirm 1} (with parameters $\epsilon$ and $\delta$) and \ref{prop cirm 2} (with parameters $\eta$ and $\delta$) give constants $S_1,D_1,S_2,D_2>0$. Set $T_\mathrm{end}=\max(T_{\rm min}+S_1, 5\tau_{K_0},T'_{\rm min})$, where $T_{\rm min}$ is given by Proposition~\ref{prop cirm 1} and $T'_{\rm min}$ by Lemma \ref{ineq Anna}.
Set $S^+=S_1$ and $S^-=2S_2+\sigma_{K,\delta,\epsilon/2}$, $D^+=D_{K_0,K}D_1$ and $D^-=\frac{\epsilon}{8bD_2}$.

Let $T\geq T_\mathrm{end}$. Let $x,y_0\in K$.

\noindent{\bf Step 2.} Upper bound.
By Lemma~\ref{ineq Anne} and Proposition~\ref{prop cirm 1} applied with $T_0=L_k-T$, $T_1=T-S_1$ and $S=S_1$, as soon as $L_k>T+T_{\rm min}$, we have
\begin{align*}
m_{K_0,L_k}(B(x,\epsilon,T)) &\le & D_{K_0,K} \frac{\mathcal N_{\mathcal C } (\phi_T (x),x,\epsilon,L_k-T, L_k-T+5\tau_{K_0},\delta) }{L_k\#\mathcal P_{K_0}(L_k,L_k+5\tau_{K_0})}+\frac{\#\mathcal P_{K_0}(\frac{L_k+5\tau_{K_0}}{2})}{\#\mathcal P_{K_0}(L_k,L_k+5\tau_{K_0})}\\
&\le & \frac{D_{K_0,K}\,D_1}{\mathcal{N}_{\mathcal C}(x,y_0,T-S_1,T-S_1+5\tau_{K_0})}+\frac{\#\mathcal P_{K_0}(\frac{L_k+5\tau_{K_0}}{2})}{\#\mathcal P_{K_0}(L_k,L_k+5\tau_{K_0})}\, .
\end{align*}
Using that $B(x,\epsilon,T)$ is open and from Corollary~\ref{corollaire_limite_nulle_quotient_P_K},
we get
\[
m_\infty(B(x,\epsilon,T))\le \liminf_{k\to +\infty}m_{K_0,L_k}(B(x,\epsilon,T))\le \frac{D_{K_0,K}\,D_1}{\mathcal{N}_{\mathcal C}(x,y_0,T-S_1,T-S_1+5\tau_{K_0})}\,.
\]
The upper bound follows with $D^+=D_{K_0,K}D_1$.

{\bf Step 3.} Lower bound.
Apply Lemma~\ref{ineq Anna} with parameter $\epsilon/2$ and Proposition~\ref{prop cirm 2} with $S= S_2$ and $T = T + S_2 + \sigma_{K,\delta,\epsilon/2}>5\tau_{K_0}$, and $x_1=\phi_T(x),y_1=x_0=x$. Then,
we get, for $k$ big enough,
\begin{align*}
m_{K_0,L_k}(B(x,\epsilon/2,T))
&\ge \frac{\epsilon}{8b}\times\dfrac{\mathcal N_{\mathcal C}(\phi_T(x),x,\eta,L_k-T-\sigma_{K,\delta,\epsilon/2},L_k-T-\sigma_{K,\delta,\epsilon/2}+5\tau_{K_0},\delta)}{L_k\#\mathcal P_{K_0}(L_k,L_k+5\tau_{K_0})} \\
&\ge  \frac{\epsilon}{8b}\times \frac{1}{D_2}\times \dfrac{1}{\mathcal N_{\mathcal C}(x,y_0,\eta,T+2S_2+\sigma_{K,\delta,\epsilon/2},T+2S_2+\sigma_{K,\delta,\epsilon/2}+5\tau_{K_0},\delta)}\, .
\end{align*}
As $\limsup_{L_k\to \infty}m_{K_0,L_k}(B(x,\epsilon/2,T))\le m_\infty(\overline{B(x,\epsilon/2,T)})\le m_\infty(B(x,\epsilon,T))$, the  lower bound follows with $D^-=\frac{\epsilon}{8bD_2}$ and $S^-=2S_2+\sigma_{K,\delta,\epsilon/2}$.
\end{proof}

\subsection{Computation of entropies of \texorpdfstring{$m_{\mathrm{max}}$}{TEXT}}

\begin{theo}\label{theo-computation-entropy}Let $\varphi$ be a H-flow on $M$ such that $h^\infty_{\mathrm{Gur}}(\varphi)<h_{\mathrm{Gur}}(\varphi)$. Let $m_{\max}$ be the probability measure obtained after renormalizing an accumulation point of the family of measures distributed on periodic orbits of increasing length, see Section \ref{section_periodic_measures}. Then
\[
\underline{h}_\mathrm{BK}(m_\mathrm{max})=\overline{h}_\mathrm{BK}(m_\mathrm{max})=h_\mathrm{KS}(m_\mathrm{max})=h_\mathrm{Kat}(m_\mathrm{max})=h_\mathrm{Gur}(\phi)\, .
\]
\end{theo}

\begin{proof} As at the beginning of Section~\ref{section:estimation_boules_dynamiques}, choose $K_0$, and $L_n\to \infty$ such that $m_{K_0,L_n}\overset{*}{\rightharpoonup} m_\infty$, and $K\supset\inter{K}\supset K_0$ such that $m_\infty(\inter{K})>\frac{3}{4}m_\infty(M)$.
Fix $\delta<\min(\alpha_0/2,\delta_\mathrm{end})$ where $\delta_\mathrm{end}$ comes from Theorem~\ref{prop_boule_dynamique} and $\alpha_0$ from Proposition~\ref{thm h cord=h gur}.
Fix $\epsilon<\min(\epsilon_{\mathrm{end},\delta},\alpha_0/4)$, where $\epsilon_{\mathrm{end},\delta}$ comes from Theorem~\ref{prop_boule_dynamique} wih parameter $\delta$, such that $\epsilon$ is small enough compatibly with
Proposition~\ref{Katok-noncompact},.
Let $x\in K$.

Theorems~\ref{prop_boule_dynamique} and Proposition~\ref{thm h cord=h gur} give us
\[\limsup_{\substack{T\to\infty \\ \phi_T(x)\in K}}\,-\frac{1}{T}\log m_\infty\left(B(x,T,\epsilon)\right) = \liminf_{\substack{T\to\infty \\ \phi_T(x)\in K}}\,-\frac{1}{T}\log m_\infty\left(B(x,T,\epsilon)\right) = h_\mathrm{Gur}(\phi).\]
Therefore
\[
\underline{h}_{BK}(m_\mathrm{max})
=\overline{h}_{BK}(m_\mathrm{max})=
h_\mathrm{Gur}(\phi).
\]
From Proposition~\ref{Katok-noncompact} we also have
\[
\int_K\limsup_{\substack{T\to+\infty\\ \phi_T(x)\in K}}-\dfrac 1 T \log(m_{\infty}(B(x,\varepsilon, T))\, dm_{\max} \leq h_{\mathrm{KS}}(m_{\max})\, .
\]
Thus we obtain the lower bound $h_\mathrm{KS}(m_\mathrm{max})\ge h_\mathrm{Gur}(\phi)$.

Corollary~\ref{entropie-var} gives the inequality $h_\mathrm{KS}(m_\mathrm{max})\le h_\mathrm{Gur}(\phi)$ so that $h_\mathrm{KS}(m_\mathrm{max})= h_\mathrm{Gur}(\phi)$.
We know by Theorem \ref{comparison-Kat-BK} that $h_\mathrm{Kat}(m_\mathrm{max})\le h_\mathrm{Gur}(\phi)$.

It only remains to prove the inequality $h_\mathrm{Kat}(m_\mathrm{max})\ge h_\mathrm{Gur}(\phi)$.
By definition of Katok entropy, for every $0<\nu<1$, there exist  $\alpha_\nu>0$, $\epsilon_\nu>0$  and $T_\nu>0$ such that, for every $0<\alpha<\alpha_\nu$, $0<\epsilon<\epsilon_\nu$ and $T\ge T_\nu$, the minimal cardinality $M(T,\epsilon,\alpha, m_\mathrm{max})$ of a set of $(\epsilon,T)$-dynamical balls covering a set of measure at least $\alpha$ satisfies
\[
M(T,\epsilon,\alpha,m_{\mathrm max})\le e^{(h_\mathrm{Kat}(m_\mathrm{max})+\nu)T}\,.
\]
Even if it means to increase $K$, we can assume that $m_\mathrm{max}(K)\ge 1-\frac{\alpha}{4}$.
By invariance of $m_\mathrm{max}$, we get $m_\mathrm{max}(K\cap \phi_{-T}(K))\ge 1-\frac{\alpha}{2}$.
Without loss of generality, we can assume that $T$ is large enough and $\epsilon>0$ small enough so that Theorem~\ref{prop_boule_dynamique} holds with parameters $K$, $\delta$ (where $\delta$ is some small enough constant), $2\epsilon$ and $T$.
We may also assume that Proposition~\ref{thm h cord=h gur} holds with parameters $K$, $C=5\tau_{K_0}$, $\delta$ and $\eta = 2\epsilon$.

Since $m_\mathrm{\max}(K\cap \phi_{-T}(K))\ge 1-\frac{\alpha}{2}$, we can use Lemma~\ref{lemma replace} to obtain a separating $(T,2\epsilon,\frac{\alpha}{2},m_\mathrm{max})$-spanning set $E'\subset K\cap\phi_{-T}(K)$ with $\#E'\le M(T,\epsilon,\alpha,m_\mathrm{max})$.
Therefore
\begin{align*}
\frac{\alpha}{2}\le m_\mathrm{max}\left(\cup_{x\in E'} B(x,2\epsilon,T)\right) &\le \sum_{x\in E'} m_\mathrm{max}(B(x,2\epsilon,T))\\
&\le M(T,\epsilon,\alpha,m_\mathrm{max})\times \max_{x\in E'} m_\mathrm{max}(B(x,2\epsilon,T)) \\
&\le e^{(h_\mathrm{Kat}(m_\mathrm{max})+\nu)T}\times \max_{x\in E'} m_\mathrm{max}(B(x,2\epsilon,T))\,.
\end{align*}
Using Theorem~\ref{prop_boule_dynamique} we get
\[
m_{\mathrm max}(B(x,2\epsilon, T))=\dfrac{1}{m_{\infty}(M)}m_{\infty}(B(x,2\epsilon, T))\leq \dfrac{1}{m_{\infty}(M)}\dfrac{D^+}{\mathcal{N}_{\mathcal{C}}(x,y_0,2\epsilon, T-S^+,T-S^++5\tau_{K_0},\delta)}\,,
\]
so that
\[
\frac{\alpha}{2D^+}\times m_\infty(M)\times  \mathcal{N}_{\mathcal{C}}(x,y_0,2\epsilon, T-S^+,T-S^++5\tau_{K_0},\delta)\le   e^{(h_\mathrm{Kat}(m_\mathrm{max})+\nu)T}\,.
\]
Thus,
\[
\dfrac 1 T \log\left(\dfrac{\alpha}{2D^+}\times m_{\infty}(M)\right)+\dfrac 1 T \log(\mathcal{N}_{\mathcal{C}}(x,y_0,2\epsilon, T-S^+,T-S^++5\tau_{K_0},\delta))\leq h_{\mathrm{Kat}}(m_\mathrm{max})+\nu\, .
\]
Taking the limit superior when $T\to+\infty$ of the above inequality, by Proposition~\ref{thm h cord=h gur}, we obtain $h_{\mathrm{Gur}}(\varphi)\leq h_{\mathrm{Kat}}(m_{\mathrm max})+\nu$. As $\nu$ can be chose arbitrarily small, we have $h_{\mathrm{Gur}}(\varphi)\leq h_{\mathrm{Kat}}(m_{\mathrm max})$.
This concludes the proof of the theorem.
\end{proof}

\section{Notations}

\begin{multicols}{2}

\begin{itemize}
    \item $B(x,\epsilon,T)$ - p. \pageref{p:dynamical_ball}
    \item $\mathcal C(x,y,\eta,T^-, T^+)$ - p. \pageref{p:C(c,y,eta,T-,T+)}
    \item $\mathcal C^{K^c}(x,y,\eta,T^-, T^+)$ - p. \pageref{p:cordes_exterieur_K}
    \item $E(x,y,\eta,T^-,T^+,\delta)$ - p. \pageref{defi E-set}
    \item $E^{K^c}(x,y,\eta, T^-,T^+,\delta)$-set - p. \pageref{defi E-set outside compact}
    \item $E_K(x,y,\eta,T^-,T^+,\delta)$-set - p. \pageref{cordes et orbite periodique version compacte}
    \item $h_{\mathcal C}(x,y,\eta,\delta)$ - p. \pageref{def:hC_x_y_eta_delta}
    \item $h_{\mathcal C} (\phi)$ - p. \pageref{def:chord-entropy}
    \item $h_{\mathcal C}^{K^c}(x,y,\eta,\delta)$ - p. \pageref{p:chord_entropy_outside_K}
    \item $h_{\mathcal C}^{K^c}(\eta,\delta)$ - p. \pageref{p:chord_entropy_outside_K}
    \item $h_{\mathcal C}^{K^c}(\phi)$ - p. \pageref{p:chord_entropy_outside_K}
    \item $h_\mathcal C^\infty(\phi)$ - p. \pageref{def:chord-entropy-at-infinity}
    \item $\underline{h}_\mathrm{BK}(\mu)$ - p. \pageref{p:BK_entropy}
    \item $\overline{h}_\mathrm{BK}(\mu)$ - p. \pageref{p:BK_entropy}
    \item $h_\mathrm{Gur}(\phi)$ - p. \pageref{theo:Gurevic}
    \item $h_\mathrm{Kat}(\mu)$ - p. \pageref{Katok entropy}
    \item $h_\mathrm{KS}(\mu)$ - p. \pageref{KS entropy}
    \item $h_\mathrm{var}(\phi)$ - p. \pageref{variational entropy}
    \item $h_\mathrm{Gur}^{K,\alpha}$ - p. \pageref{definition_gurevic_infini}
    \item $h_\mathrm{Gur}^\infty$ - p. \pageref{definition_gurevic_infini}
    \item $h_\mathrm{var}^\infty$ - p. \pageref{variational infty entropy}
    \item $K_{-\eta}$ - p. \pageref{p:eta_interior_compact_set}
    \item $\ell$ - p. \pageref{p:ell}
    \item $m_{K,L}$ - p. \pageref{eqn:m_{K,L}}
    \item $m_\mathrm{max}$ - p. \pageref{def:measures}
    \item $m_\infty$ - p. \pageref{def:measures}
    \item $M(T,\epsilon,\alpha,\mu)$ - p. \pageref{fact:M_Katok_entropy}
    \item $M'(T,\epsilon,\alpha,\mu)$ - p. \pageref{p:M'}
    \item $\mathcal M_\phi$ - p. \pageref{p:M_phi_M_erg_phi}
    \item $\mathcal M^\mathrm{erg}_\phi$ - p. \pageref{p:M_phi_M_erg_phi}
    \item $\mathcal{N}_{\mathcal{C}}(x,y,\eta,T^-,T^+,\delta)$ - p. \pageref{defi E-set}
    \item $\mathcal{N}^{K^c}_{\mathcal{C}}(x,y,\eta,T^-,T^+,\delta)$ - p. \pageref{defi E-set outside compact}
    \item $\mathcal{P}_K(T_0)$ - p. \pageref{p:P_K_T}
    \item $\mathcal{P}_K(T_0,T_1)$ - p. \pageref{p:P_K_T}
    \item $\mathcal{P}_K^\alpha(L,L+C)$ - p. \pageref{definition_gurevic_infini}
    \item $\mathcal P(K_1,K_2,\alpha,T,T+C)$ - p. \pageref{p:P_K1_K2}
    \item $\tau_K$ - p. \pageref{p:tau_K}
\end{itemize}

\end{multicols}

\bibliographystyle{alpha}
\bibliography{BiblioAAB.bib}

\end{document}